\documentclass[11pt,reqno,a4paper]{article}

\usepackage{avelinpaper}
\newcommand{\Pp}{\P}
\newcommand{\Ee}{\E}
\newcommand{\TV}{\mathrm{TV}}

\DeclareMathOperator{\Var}{Var}
\DeclareMathOperator{\arctanh}{arctanh}

\title{Absorption cutoff and stationary singularities for rounded Gaussian random dynamical systems}
\author{Benny Avelin\thanks{Department of Mathematics, Uppsala University,
751 05 Uppsala, Sweden. Email: \texttt{benny.avelin@math.uu.se}.}}
\date{August 16, 2026}
\hypersetup{
  pdftitle={Absorption cutoff and stationary singularities for rounded Gaussian random dynamical systems},
  pdfauthor={Benny Avelin},
  pdfsubject={Cutoff and stationary singularities for rounded Gaussian Markov chains},
  pdfkeywords={cutoff, absorption, finite precision, Gaussian Markov chains, metastability, renewal tails}
}

\begin{document}

\maketitle

\begin{abstract}
  We study Gaussian random dynamical systems with coordinatewise $\tanh$
  nonlinearity, where finite precision is modeled by nearest-grid rounding
  after each step. Gaussian symmetry reduces the
  dynamics to an exact Markov chain for the normalized squared radius.
  Rounding makes the
  origin absorbing, and the total variation distance to the absorbing
  equilibrium equals the survival probability of the absorption time. At
  fixed width, we identify the critical gain and prove an absorption cutoff
  with Gaussian profile as the mesh tends to zero. At fixed precision, global
  contraction yields a large-dimension absorption cutoff, while positive
  drift produces metastability. In the
  supercritical regime, we prove a large-dimension cutoff to a nonzero invariant
  law and show that, at fixed dimension, its mass near the repelling origin
  has a power-law asymptotic. All six main results are formalized in Lean~4 on
  top of Mathlib and independently checked against restatements that import
  only Mathlib.
\end{abstract}


\smallskip


\tableofcontents

\section{Introduction}
\label{sec:introduction}

We study cutoff for nonlinear Markov chains obtained by rounding random
iterated maps to a finite grid. The grid makes exact absorption possible. In
the absorbing regimes considered below, the total variation distance to
equilibrium equals the survival probability of the absorption time. The
cutoff problem therefore reduces to sharp first passage for a random radius.

Random neural-network dynamics provide the motivating example. In real-valued
models, the weights are scaled to control signal propagation through depth
\cite{glorot-bengio2010,he-etal2015,poole-etal2016,schoenholz-etal2017}.
Finite-precision arithmetic instead confines the numerical state to a finite
set \cite{higham2002}. The zero state can then be reached exactly, turning
gradual signal decay into absorption.

This phenomenon was observed numerically in a joint work with Karlsson
\cite{avelin-karlsson2022}. The model iterates random fully connected layers
with independent weights uniformly distributed on
$[-N^{-1/2},N^{-1/2}]$:
\begin{equation}\label{eq:intro-network}
  X_{t+1}^{(N)}
  =
  \tanh\left(W_{t+1}^{(N)}X_t^{(N)}\right),
  \qquad
  X_t^{(N)}\in[-1,1]^N,
\end{equation}
where $\tanh$ acts coordinatewise. Rounding makes the origin absorbing. With
continuous weights, the unrounded chain started away from zero does not hit the
origin at any finite time, so its total variation distance to $\delta_0$ remains
one. Numerical experiments at fixed precision showed a sharp transition in
depth for dimensions $1$ and $2$.

For the Gaussian model studied below, we prove a Gaussian-profile absorption
cutoff as the mesh tends to zero at fixed width. At fixed precision and growing
width, we prove a bounded-window absorption cutoff in the globally contracting
regime and exponential persistence in the positive-drift regime. For the
unrounded supercritical chain, we prove cutoff to a nonzero invariant law as the
width grows and, at fixed width, determine the power-law mass of this law near
the repelling origin.

\subsection{Gaussian model and rounding}
\label{subsec:intro-model}

For the remainder of the paper, we replace the bounded weights in the
motivating model by independent Gaussian weights with gain parameter $A>0$:
\begin{equation}\label{eq:intro-gaussian-weights}
  (\mathsf W_{t,A}^{(N)})_{ij}
  \sim
  \mathcal N(0,A^2/N),
  \qquad 1\leq i,j\leq N,\quad t\geq1.
\end{equation}
For $x\in[-1,1]^N$, we write $X^{(N)}$ for the resulting unrounded chain:
\begin{equation}\label{eq:intro-gaussian-chain}
  X_0^{(N)}=x,
  \qquad
  X_{t+1}^{(N)}
  =
  \tanh\left(\mathsf W_{t+1,A}^{(N)}X_t^{(N)}\right).
\end{equation}

For a mesh size $\varrho>0$, let $Q_\varrho$ denote coordinatewise
nearest-grid rounding to $\varrho\Z^N$, with ties rounded toward the grid
point of smaller absolute value. The rounded chain is
\begin{equation}\label{eq:intro-rounded-chain}
  Y_0^{(N)}=Q_\varrho x,
  \qquad
  Y_{t+1}^{(N)}
  =
  Q_\varrho\left(
    \tanh\left(\mathsf W_{t+1,A}^{(N)}Y_t^{(N)}\right)
  \right),
\end{equation}
and we write $P_{\varrho,A,N}$ for its transition kernel. Since the point $0$ is
absorbing, define the absorption time as
\begin{equation}\label{eq:intro-absorption-time}
  \tau_\varrho^{(N)}
  :=
  \inf\{t\geq0:Y_t^{(N)}=0\},
\end{equation}
then the total variation distance to the absorbing equilibrium is exactly
\begin{equation}\label{eq:intro-tv-absorption}
  \norm{
    P_{\varrho,A,N}^t(Q_\varrho x,\cdot)-\delta_0
  }_{\TV}
  =
  \Pp(\tau_\varrho^{(N)}>t).
\end{equation}
Thus finite-precision cutoff to $\delta_0$ is an absorption-time question.

Set
\begin{equation*}
  Z_t^{(N)}:=N^{-1}\norm{X_t^{(N)}}_2^2
\end{equation*}
and suppose that $Z_t^{(N)}=q$. The preactivation coordinates at the next
step are independent centered Gaussian variables with variance $A^2q$.
Hence $Z^{(N)}$ is a Markov chain, and, for independent standard normal
variables $G_1,\ldots,G_N$, its transition is
\begin{equation}\label{eq:intro-radius-chain}
  Z_{t+1}^{(N)}
  \stackrel{d}{=}
  \frac1N\sum_{i=1}^N
  \tanh^2\left(A\sqrt{Z_t^{(N)}}\,G_i\right).
\end{equation}
We denote the transition kernel of $Z^{(N)}$ by $K_{A,N}$.
The updates of $X^{(N)}$ and $Y^{(N)}$ are examples of iterated random
functions. For invariant laws and convergence under average contraction, see
Diaconis and Freedman \cite{diaconis-freedman1999}. The corresponding
scalar reductions and further notation are given in \cref{sec:setup}.

\smallskip

We first recall the threshold for the unrounded chain as $N\to\infty$, and
then define the three thresholds used for the rounded chain. Let $G$ be a
standard normal variable and define
\begin{equation}\label{eq:intro-mean-field-map}
  V_A(q):=\Ee\tanh^2(A\sqrt q\,G).
\end{equation}
Since $\tanh u=u+O(u^3)$ near zero, we have
$V_A(q)=A^2q+O(q^2)$ for $q>0$ small. Thus the threshold is $A=1$ as
$N\to\infty$. Strict concavity of $V_A$ gives a unique nonzero attracting
fixed point $q_\ast$ when $A>1$.

At fixed $N$, the transition from $q$ satisfies
\begin{equation*}
  \sqrt{\frac{Z_{t+1}^{(N)}}q}
  \Longrightarrow
  \frac A{\sqrt N}\chi_N,
  \qquad
  q\downarrow0,
\end{equation*}
where $\chi_N$ is the Euclidean norm of a standard Gaussian vector in
$\R^N$. Set
\begin{equation*}
  \ell_{A,N}:=\frac A{\sqrt N}\chi_N,
  \qquad
  \gamma_{A,N}:=-\Ee\log\ell_{A,N}.
\end{equation*}
The equation $\gamma_{A,N}=0$ has the unique solution
\begin{equation}\label{eq:intro-fixed-width-threshold}
  A_c(N)
  =
  \sqrt{\frac N2}\,
  \exp\left\{-\frac12\psi\left(\frac N2\right)\right\},
\end{equation}
where $\psi$ is the digamma function. Thus
$\gamma_{A,N}=\log(A_c(N)/A)$, and $A<A_c(N)$ is equivalent to
$\Ee\log\ell_{A,N}<0$. Moreover, $A_c(N)>1$ for finite $N$ and
$A_c(N)\downarrow1$ as $N\to\infty$.

For the rounded chain, measure the normalized squared radius in units of
$\varrho^2$, so that $h=q/\varrho^2$. As $N\to\infty$ at fixed $\varrho$,
the mean map for $h$ is
\begin{equation}\label{eq:intro-rounded-map}
  V_{A,\varrho}(h)
  :=
  \Ee\left[
    Q_1\left(
      \varrho^{-1}\tanh(\varrho A\sqrt h\,G)
    \right)^2
  \right].
\end{equation}
Here $Q_1$ is nearest-integer rounding, with ties rounded toward the integer
of smaller absolute value. Since $\abs{\tanh z}\leq\abs z$, we have
$V_{A,\varrho}(h)\leq L_A(h)$, where
\begin{equation*}
  L_A(h):=\Ee[Q_1(A\sqrt h\,G)^2].
\end{equation*}
Define
\begin{equation}\label{eq:intro-lattice-threshold}
  A_{\mathrm{lat}}^2
  :=
  \inf_{\alpha>0}
  \frac{\alpha^2}{\Ee[Q_1(\alpha G)^2]} .
\end{equation}
The condition $A<A_{\mathrm{lat}}$ is equivalent to
$L_A(h)<h$ for every $h>0$. In this case, both $L_A$ and
$V_{A,\varrho}$ contract globally toward zero. Set
\begin{equation}\label{eq:intro-metastable-threshold}
  A_{\mathrm{ex}}(\varrho)^2
  :=
  \inf_{\alpha>0}
  \frac{\alpha^2}
  {\Ee\left[
    Q_1\left(\varrho^{-1}\tanh(\varrho\alpha G)\right)^2
  \right]} .
\end{equation}
Then $A>A_{\mathrm{ex}}(\varrho)$ is equivalent to
$V_{A,\varrho}(h)>h$ for some $h>0$. In this case, the rightmost component of
$\{h:V_{A,\varrho}(h)>h\}$ produces metastability.

The three thresholds enter the rounded results as follows. At fixed $N$,
$A<A_c(N)$ gives the cutoff as $\varrho\downarrow0$. At fixed $\varrho$ and
as $N\to\infty$, $A<A_{\mathrm{lat}}$ gives the absorption cutoff, while
$A>A_{\mathrm{ex}}(\varrho)$ gives metastability.

For both absorption cutoffs, the relevant distance is
\cref{eq:intro-tv-absorption}.
In the supercritical large-dimension problem, the origin is repelling and
$X^{(N)}$ has a nonzero invariant law. The relevant distance is
then total variation distance to this law.
Here and below, nonzero means that the invariant probability law assigns zero
mass to the absorbing state. Mixtures with $\delta_0$ are excluded.

\subsection{Cutoff convention}
\label{subsec:intro-cutoff-convention}

We use the following convention throughout.

\begin{definition}[Total variation cutoff]\label{def:intro-cutoff}
  Let $\mathcal I\subset\R$, and let $r_\ast\in[-\infty,\infty]$ belong to the
  closure of $\mathcal I$ in the extended real line. For each $r\in\mathcal I$, let $P_r$ be a
  Markov kernel on a measurable space $E_r$, let $\pi_r$ be an invariant
  probability measure, and let $x_r\in E_r$. Set
  \begin{equation}\label{eq:intro-tv-distance}
    d_r(t)
    :=
    \norm{P_r^t(x_r,\cdot)-\pi_r}_{\TV},
    \qquad
    t\in\N_0 .
  \end{equation}
  For $u\in\R$, write
  $\lfloor u\rfloor_+:=\max\{0,\lfloor u\rfloor\}$.
  Let $t_r\to\infty$ and $a_r=o(t_r)$ with $a_r>0$. We say that
  $(P_r,x_r,\pi_r)_{r\in\mathcal I}$ has total variation cutoff at time $t_r$
  with window $a_r$ if
  \begin{equation}\label{eq:intro-cutoff-convention}
    \lim_{c\to\infty}\liminf_{r\to r_\ast}
    d_r\left(\lfloor t_r-ca_r\rfloor_+\right)
    =
    1,
    \qquad
    \lim_{c\to\infty}\limsup_{r\to r_\ast}
    d_r\left(\lfloor t_r+ca_r\rfloor_+\right)
    =
    0.
  \end{equation}
  A function $F:\R\to[0,1]$ is a cutoff profile on the scale $a_r$ if, for
  every $c\in\R$,
  \begin{equation}\label{eq:intro-cutoff-profile}
    d_r\left(\lfloor t_r+ca_r\rfloor_+\right)
    \to
    F(c),
    \qquad
    r\to r_\ast,
  \end{equation}
  where $F(c)\to1$ as $c\to-\infty$ and $F(c)\to0$ as $c\to\infty$.

  When no window is specified, cutoff at time $t_r$ means that, for every
  fixed $\delta\in(0,1)$,
  \begin{equation}\label{eq:intro-cutoff-window-free}
    d_r\left(\lfloor(1-\delta)t_r\rfloor_+\right)\to1,
    \qquad
    d_r\left(\lfloor(1+\delta)t_r\rfloor_+\right)\to0,
    \qquad
    r\to r_\ast.
  \end{equation}
  Equivalently, there exists $b_r=o(t_r)$ with $b_r>0$ such that the family
  has cutoff at time $t_r$ with window $b_r$.
\end{definition}

\begin{definition}[Mixing time]\label{def:intro-mixing-time}
  For $\varepsilon\in(0,1)$, define
  \begin{equation}\label{eq:intro-mixing-time}
    t_{\mathrm{mix}}^{(r)}(\varepsilon)
    :=
    \inf\{t\in\N_0:\ d_r(t)\leq\varepsilon\},
  \end{equation}
  where the infimum is $\infty$ if the set is empty.
\end{definition}

\begin{remark}
  Since $d_r$ is nonincreasing, cutoff with window $a_r$ implies, for every
  fixed $\varepsilon\in(0,1)$,
  \begin{equation*}
    t_{\mathrm{mix}}^{(r)}(\varepsilon)
    =
    t_r+O_\varepsilon(a_r).
  \end{equation*}
  Thus window-free cutoff gives
  $t_{\mathrm{mix}}^{(r)}(\varepsilon)=t_r+o(t_r)$. If $a_r\to\infty$, the
  family has profile $F$, and $F$ is continuous and strictly decreasing
  through the level $\varepsilon$, then
  \begin{equation*}
    t_{\mathrm{mix}}^{(r)}(\varepsilon)
    =
    t_r+c_\varepsilon a_r+o(a_r),
    \qquad
    F(c_\varepsilon)=\varepsilon.
  \end{equation*}
\end{remark}

\subsection{Main results}
\label{subsec:main-results}

We first
state the two absorption cutoffs, followed by metastability, the supercritical
cutoff, and the behavior of the invariant law near zero. We write
$\Phi_{\mathrm G}$ for the distribution function of a standard Gaussian
random variable.

For fixed $N\geq1$, $0<A<A_c(N)$, deterministic
$x_0\in\R^N\setminus\{0\}$, and $0<\varrho<R_0:=\norm{x_0}_2$, set
\begin{equation*}
  L_\varrho:=\log\frac{R_0}{\varrho},
  \qquad
  \sigma_N^2:=\Var(\log\chi_N)
  =\frac14\psi_1\left(\frac N2\right),
\end{equation*}
where $\psi_1=\psi'$ is the trigamma function. For $a\in\R$, define the
nonnegative integer
\begin{equation*}
  t_\varrho(a)
  :=
  \max\left\{0,
    \left\lfloor
      \frac{L_\varrho}{\gamma_{A,N}}
      +
      a\,\frac{\sigma_N}{\gamma_{A,N}^{3/2}}\sqrt{L_\varrho}
    \right\rfloor
  \right\} .
\end{equation*}

\begin{theorem}[Vanishing-mesh cutoff]
\label{thm:rounded-gaussian-nearest-cutoff}
  Assume $0<A<A_c(N)$. Then, for every fixed $a\in\R$,
  \begin{equation}\label{eq:rounded-gaussian-nearest-profile}
    \norm{
      P_{\varrho,A,N}^{t_\varrho(a)}(Q_\varrho x_0,\cdot)-\delta_0
    }_{\TV}
    =
    \Pp(\tau_\varrho^{(N)}>t_\varrho(a))
    \to
    \Phi_{\mathrm G}(-a),
    \qquad
    \varrho\downarrow0.
  \end{equation}
  In particular, $Y^{(N)}$ has total variation cutoff to $\delta_0$
  at time $L_\varrho/\gamma_{A,N}$ with window
  $\sigma_N \gamma_{A,N}^{-3/2}\sqrt{L_\varrho}$ and Gaussian profile
  $a\mapsto\Phi_{\mathrm G}(-a)$.
\end{theorem}

For $A>0$, $\varrho\in(0,1)$, and deterministic
$x_N\in[-1,1]^N$, define
\begin{equation*}
  \bar h_{0,N}:=\frac1{N\varrho^2}\norm{Q_\varrho x_N}_2^2,
  \qquad
  \bar h_{t+1,N}:=V_{A,\varrho}(\bar h_{t,N}),
\end{equation*}
and
\begin{equation*}
  \bar t_N
  :=
  \inf\left\{t\geq0:\bar h_{t,N}\leq\frac1{\log N}\right\}.
\end{equation*}

\begin{theorem}[Subcritical dimension cutoff]
\label{thm:subcritical-dimension-cutoff:intro}
  Assume $A<A_{\mathrm{lat}}$ and that the rounded initial radii are
  macroscopic:
  \begin{equation*}
    \liminf_{N\to\infty}\bar h_{0,N}>0 .
  \end{equation*}
  Then $\bar t_N\to\infty$ and
  \begin{equation*}
    \norm{
      P_{\varrho,A,N}^{\bar t_N-1}
      (Q_\varrho x_N,\cdot)-\delta_0
    }_{\TV}
    \to1,
    \qquad
    \norm{
      P_{\varrho,A,N}^{\bar t_N+2}
      (Q_\varrho x_N,\cdot)-\delta_0
    }_{\TV}
    \to0 .
  \end{equation*}
  Thus $Y^{(N)}$ has total variation cutoff to $\delta_0$
  at the exact deterministic center $\bar t_N$, with a bounded
  window.
\end{theorem}

\begin{remark}
\label{rem:subcritical-macroscopic-intro}
  The macroscopic assumption in
  \cref{thm:subcritical-dimension-cutoff:intro} is sufficient but not
  necessary. Continuity and positivity of $V_{A,\varrho}$ imply that it gives
  $\bar t_N\to\infty$. On the other hand, $A<A_{\mathrm{lat}}$ implies
  $\bar h_{t,N}\leq C_\varrho\theta_A^t$ for some $\theta_A\in(0,1)$, and
  hence $\bar t_N=O(\log\log N)$. In place of the macroscopic assumption,
  \cref{thm:subcritical-dimension-cutoff} requires only
  $\bar t_N\to\infty$ and therefore also permits initial radii that approach
  zero.
\end{remark}

For $\varrho\in(0,1)$, set
\begin{equation*}
  M_\varrho
  :=
  \max_{\abs u\leq\varrho^{-1}} Q_1(u)^2,
  \qquad
  \mathcal O_{A,\varrho}
  :=
  \{h\in(0,M_\varrho):V_{A,\varrho}(h)>h\},
\end{equation*}
and call its connected components positive-drift components.
Write
\begin{equation*}
  \mathscr H_t^{(N)}
  :=
  \frac1{N\varrho^2}\norm{Y_t^{(N)}}_2^2.
\end{equation*}

\begin{theorem}[Metastability]
\label{thm:rounded-qualitative-metastability:intro}
  Assume $A>A_{\mathrm{ex}}(\varrho)$, and let $(h_-,h_+)$ be the
  rightmost positive-drift component of $\mathcal O_{A,\varrho}$. Fix a
  compact set $B\subset(h_-,h_+]$ and deterministic initial states
  $Y_0^{(N)}$ satisfying $\mathscr H_0^{(N)}\in B$ for every $N\geq1$.
  Then there are
  $T\in\N_0$ and $c>0$ such that
  \begin{equation*}
    \Pp\left(
      \tau_\varrho^{(N)}
      >
      T+\lfloor\exp(cN)\rfloor
    \right)
    \to1 .
  \end{equation*}
  For each fixed $N$, however, absorption still occurs almost surely for every
  deterministic choice of $Y_0^{(N)}$.
\end{theorem}

The persistence of $Y^{(N)}$ describes its behavior before eventual
absorption. We next remove rounding and study relaxation as $N\to\infty$. In
the supercritical regime, $X^{(N)}$ approaches a nonzero invariant law, under
which $Z^{(N)}$ is concentrated near $q_\ast$.

Let $A>1$, and let $q_\ast$ and $\mu_A=V_A'(q_\ast)\in(0,1)$ be the nonzero
attracting fixed point and multiplier of $V_A$. For
$q_0\in(0,1]\setminus\{q_\ast\}$, let $C(q_0)\neq0$ be determined by
\begin{equation*}
  V_A^t(q_0)-q_\ast=C(q_0)\mu_A^t+o(\mu_A^t),
\end{equation*}
and set
\begin{equation}\label{eq:gaussian-cutoff-time}
  t_N(q_0)
  :=
  \frac{\frac12\log N+\log\abs{C(q_0)}}{\abs{\log\mu_A}}.
\end{equation}
For $c\in\R$, define the nonnegative integer
\begin{equation}\label{eq:gaussian-integer-cutoff-time}
  n_N(c):=\max\{0,\lfloor t_N(q_0)+c\rfloor\} .
\end{equation}

\begin{theorem}[Supercritical cutoff for $Z^{(N)}$]
\label{thm:gaussian-process-cutoff}
  Let $x_N\in[-1,1]^N$ be deterministic and suppose
  \begin{equation*}
    q_N:=N^{-1}\norm{x_N}_2^2
    \to
    q_0.
  \end{equation*}
  Then, for all sufficiently large $N$, $K_{A,N}$ has a
  unique nonzero invariant law $\nu_{A,N}$, and
  \begin{equation}\label{eq:gaussian-process-lower}
    \lim_{c\to\infty}\liminf_{N\to\infty}
    \norm{K_{A,N}^{n_N(-c)}(q_N,\cdot)-\nu_{A,N}}_{\TV}
    =1,
  \end{equation}
  while
  \begin{equation}\label{eq:gaussian-process-upper}
    \lim_{c\to\infty}\limsup_{N\to\infty}
    \norm{K_{A,N}^{n_N(c)}(q_N,\cdot)-\nu_{A,N}}_{\TV}
    =0.
  \end{equation}
\end{theorem}

\begin{corollary}[Supercritical cutoff for $X^{(N)}$]
\label{cor:gaussian-vector-cutoff}
  Under the assumptions of \cref{thm:gaussian-process-cutoff}, $X^{(N)}$ has
  a unique nonzero invariant law $\pi_{A,N}$ for all
  sufficiently large $N$, and
  \begin{equation}\label{eq:gaussian-vector-cutoff-lower}
    \lim_{c\to\infty}\liminf_{N\to\infty}
    \norm{P_{A,N}^{n_N(-c)}(x_N,\cdot)-\pi_{A,N}}_{\TV}
    =1,
  \end{equation}
  while
  \begin{equation}\label{eq:gaussian-vector-cutoff-upper}
    \lim_{c\to\infty}\limsup_{N\to\infty}
    \norm{P_{A,N}^{n_N(c)}(x_N,\cdot)-\pi_{A,N}}_{\TV}
    =0.
  \end{equation}
\end{corollary}

By \cref{eq:gaussian-cutoff-time},
$\abs{C(q_0)}\mu_A^{t_N(q_0)}=N^{-1/2}$. Shifting $t_N(q_0)$ by $c$
multiplies the left-hand side by $\mu_A^c$, so the cutoff window is bounded.
Since $n_N(c)=\max\{0,\lfloor t_N(q_0)+c\rfloor\}$, it depends on the
fractional part of $t_N(q_0)+c$. Consequently, no nonconstant continuous
profile on $\R$, including a Gaussian profile, can arise on this scale. A
sharper profile would have to retain the floor in $n_N(c)$ and may require
subsequences along which the fractional part of $t_N(q_0)$ converges.

Finally, fix $N$ and consider the nonzero invariant law near the repelling
origin. Its mass there has a power-law asymptotic.

\begin{theorem}[Stationary singularity in the supercritical regime]
\label{thm:nd-power-singularity:intro}
  Fix $N\geq1$ and assume $A>A_c(N)$. Let $\pi_{A,N}$ be the nonzero invariant
  law of $X^{(N)}$. Let $\beta_{A,N}\in(0,N)$ be the unique
  positive solution of
  \begin{equation*}
    A^{-\beta}N^{\beta/2}2^{-\beta/2}
    \frac{\Gamma((N-\beta)/2)}{\Gamma(N/2)}
    =
    1.
  \end{equation*}
  Here $\bar\sigma_N$ denotes normalized surface measure on
  $\mathbb S^{N-1}$. Then there is a constant $c_{A,N}>0$ such that, for every
  Borel set
  $B\subseteq\mathbb S^{N-1}$ satisfying $\bar\sigma_N(\partial B)=0$,
  \begin{equation*}
    s^{-\beta_{A,N}}\pi_{A,N}\left(
      0<\norm{x}_2\leq s,\,
      \frac{x}{\norm{x}_2}\in B
    \right)
    \longrightarrow
    c_{A,N}\bar\sigma_N(B)
    \qquad \mbox{as } s\downarrow0 .
  \end{equation*}
  In particular,
  \begin{equation*}
    \pi_{A,N}\left(0<\norm{x}_2\leq s\right)
    \sim
    c_{A,N}s^{\beta_{A,N}}
    \qquad \mbox{as } s\downarrow0 .
  \end{equation*}
\end{theorem}

\subsection{Previous work and novelty}
\label{subsec:previous-work}

The closest neural-network literature describes signal propagation through
deterministic recursions for variances and correlations. This includes
vanishing gradients and variance-based initialization
\cite{bengio-simard-frasconi1994,glorot-bengio2010,he-etal2015}, as well as the
order-to-chaos transition in the large-width limit
\cite{poole-etal2016,schoenholz-etal2017}. We instead study the absorption time
created by finite-precision rounding and its concentration as the mesh or
width varies.

The identity \cref{eq:intro-tv-absorption} is the standard relation between
distance to an absorbing state and the tail of the absorption time. Mart\'inez
and Ycart use this relation for families whose initial states leave every
finite set. They connect cutoff to concentration of the absorption times and
give moment criteria for tree-type birth-and-death chains
\cite{martinez-ycart2001}. Here rounding creates the absorbing state, and our
estimates use logarithmic-radius random walks and fluid limits.

For general background on cutoff and mixing times, see
\cite{diaconis1996,levin-peres-wilmer2017}. Cutoff criteria have been obtained
for $L^p$ semigroups, birth-and-death chains, reversible chains through
hitting-time concentration, and chains with nonnegative curvature
\cite{chen-saloff-coste2008,diaconis-saloff-coste2006,
ding-lubetzky-peres2010,basu-hermon-peres2017,salez2024-curvature}.
These results concern ergodic chains or semigroups and do not directly apply
to the absorbing rounded dynamics.
The local Poincar\'e criterion of Pedrotti and Salez
\cite{pedrotti-salez2026-local-product} likewise does not apply: it assumes
$\mu_t\ll\pi$ for every $t>0$, whereas here $\pi=\delta_0$ and the law before
absorption charges nonzero states.

Related results treat small random perturbations of deterministic dynamics.
Barrera and Jara prove total-variation cutoff for small Brownian perturbations
under coercivity and growth assumptions
\cite{barrera-jara2020}. Barrera, H\"ogele, and Pardo obtain
total-variation cutoff for a class of L\'evy-driven Langevin systems and
related Wasserstein results for generalized Ornstein--Uhlenbeck systems
\cite{barrera-hogele-pardo2021-tv,
barrera-hogele-pardo2021-wasserstein}. In joint work with Barrera, we study
cutoff and its failure for one-dimensional Langevin--Kolmogorov dynamics by
spectral Fokker--Planck methods
\cite{avelin-barrera2026-chi2-cutoff}. Here we instead study discrete-time
absorption and metastability caused by rounding. Together with Karlsson, we
previously observed this cutoff numerically at fixed precision in dimensions
one and two
\cite{avelin-karlsson2022}.

The contrast between cutoff and metastable escape also appears in
birth-and-death chains: downhill hitting times concentrate, while normalized
uphill escape times are asymptotically exponential
\cite{barrera-bertoncini-fernandez2009}. Our two behaviors likewise occur in
different parameter regimes. The metastability theorem gives exponential
persistence, but not a sharp exit-time law or a quasi-stationary limit. See
\cite{bovier-denhollander2015,collet-martinez-sanmartin2013,
champagnat-villemonais2023} for general background. The stationary singularity
uses Kesten--Goldie renewal theory
\cite{kesten1973,goldie1991,buraczewski-damek-mikosch2016}. Gaussian isotropy
refreshes the angle at each step, reducing the renewal equation to a scalar
one.
To the best of our knowledge, these absorption cutoffs, our metastability
estimate, and the stationary singularity have not previously
been proved for this random neural-network model.

\subsection{Formalization}
\label{subsec:formalization}

All six headline results in \cref{subsec:main-results} have been formalized in
Lean~4 on top of Mathlib. The Lean development is publicly available at
\url{https://github.com/BennyAvelin/AbsorptionCutoff}.\footnote{The repository
is pinned to Lean \texttt{v4.32.0} and Mathlib \texttt{v4.32.0}.}
The production library contains no \texttt{sorry} and declares no custom
axiom. A machine-checked audit shows that the seven declarations below use
only Mathlib's three standard foundational axioms
(\texttt{propext}, \texttt{Classical.choice}, \texttt{Quot.sound}). The
fixed-dimensional absorption clause of the metastability theorem is audited
separately.

\begin{center}
\renewcommand{\arraystretch}{1.2}
\begin{tabular}{ll}
  \toprule
  Result & Lean declaration \\
  \midrule
  \cref{thm:rounded-gaussian-nearest-cutoff}
    & \texttt{rounded\_gaussian\_nearest\_cutoff} \\
  \cref{thm:subcritical-dimension-cutoff:intro}
    & \texttt{subcritical\_dimension\_cutoff} \\
  \cref{thm:rounded-qualitative-metastability:intro}, persistence
    & \texttt{rounded\_qualitative\_metastability} \\
  \cref{thm:rounded-qualitative-metastability:intro}, absorption
    & \texttt{rounded\_fixed\_dimension\_absorption} \\
  \cref{thm:gaussian-process-cutoff}
    & \texttt{gaussian\_process\_cutoff} \\
  \cref{cor:gaussian-vector-cutoff}
    & \texttt{gaussian\_vector\_cutoff} \\
  \cref{thm:nd-power-singularity:intro}
    & \texttt{nd\_power\_singularity} \\
  \bottomrule
\end{tabular}
\end{center}

The Lean comparator checks the six printed results and the separate absorption
clause against restatements that import only Mathlib. Each restatement appears
in a challenge file with one intentional \texttt{sorry}. A matching solution
derives it from the development. Challenge files are excluded from the
production library. The comparator checks that each pair has the same Lean
type, that the solution is accepted by the kernel, and which axioms it uses. It
cannot verify that a restatement expresses the intended mathematics.

The paper and Lean code sometimes use different notation or equivalent
hypotheses. After unfolding the definitions and equivalences, every printed
result follows from its Lean declaration. No formalized result requires an
additional hypothesis.

Our use of autoformalization is inspired by the recent Lean formalizations of
De Giorgi--Nash--Moser theory by Armstrong and Kempe
\cite{armstrong-kempe2026}, and of coarse-graining theory
for elliptic equations by Armstrong and Kuusi
\cite{armstrong-kuusi2026-coarsegraining}. The development was produced by
autoformalization. Under the author's supervision, OpenAI's ChatGPT and
Anthropic's Claude generated the proofs from the arguments of this paper. The
Lean kernel checks every definition, statement, and proof in the production
library against Mathlib.

\subsection{Proof overview and organization}
\label{subsec:proof-overview}

To prove \cref{thm:rounded-gaussian-nearest-cutoff}, we couple $Y^{(N)}$ to
$X^{(N)}$ until the radius reaches the grid scale $\varrho$. In logarithmic
coordinates, the radius is a perturbed random walk, and a first-passage central
limit theorem gives the Gaussian profile. An estimate on the finite grid then
turns entrance at this scale into absorption.

For \cref{thm:subcritical-dimension-cutoff:intro}, the lattice comparison
$V_{A,\varrho}\leq L_A$ gives uniform contraction of the deterministic orbit
when $A<A_{\mathrm{lat}}$. The chain $\mathscr H^{(N)}$ concentrates around this
orbit until it reaches the absorption scale, and absorption occurs within two
further steps.
For \cref{thm:rounded-qualitative-metastability:intro}, the iterates starting in
$B\subset(h_-,h_+]$ approach $h_+$, and Hoeffding bounds keep
$\mathscr H^{(N)}$ near $h_+$ for exponentially many steps.

For \cref{thm:gaussian-process-cutoff}, the orbit $V_A^t(q_0)$ reaches an
$N^{-1/2}$-neighborhood of $q_\ast$, the scale of the stationary fluctuations.
Negative-moment bounds control visits of $Z^{(N)}$ near zero, and concentration
localizes the invariant law near $q_\ast$. We couple $Z^{(N)}$ to a stationary
copy. Synchronous contraction reduces their separation, and a one-step
total-variation estimate proves the cutoff. The comparison in
\cref{prop:gaussian-tv-reduction} gives \cref{cor:gaussian-vector-cutoff}.

To prove \cref{thm:nd-power-singularity:intro}, we write the invariant equation
in log-polar coordinates. Gaussian isotropy refreshes the angular variable in
one step, and the radial equation becomes a scalar renewal equation with
Cramér root $\beta_{A,N}$.

The common notation and scalar estimates are in \cref{sec:setup}.
\Cref{sec:subcritical-precision,sec:subcritical-dimension}
prove the two absorption cutoffs, and \cref{sec:rounded-metastability} treats
metastability. The supercritical cutoff and stationary singularity are
proved in \cref{sec:supercritical-dimension,sec:nd-supercritical-stationary}.
The auxiliary estimates are in
\cref{app:common-probabilistic-estimates,app:supercritical-cutoff-estimates}.

\section{Gaussian chains and scalar estimates}
\label{sec:setup}

We recall the vector chains and derive their scalar reductions. We then collect
the thresholds, properties of the mean maps, scalar estimates, transition
identities, and facts about invariant laws used later. Additional probabilistic
estimates are proved in
\cref{app:common-probabilistic-estimates}.

We write $\norm{\cdot}_2$ for the Euclidean norm on vectors. Fix $A>0$,
$N\geq1$, and $x\in\R^N$. Throughout the paper the Gaussian weights are
\begin{equation}\label{eq:gaussian-weights}
  (\mathsf W_{t,A}^{(N)})_{ij}
  \sim
  \mathcal N(0,A^2/N),
  \qquad
  1\leq i,j\leq N,\quad t\geq1,
\end{equation}
independently over all indices. The unrounded vector chain is
\begin{equation}\label{eq:unrounded-chain}
  X_0^{(N)}=x,
  \qquad
  X_{t+1}^{(N)}
  =
  \tanh\left(\mathsf W_{t+1,A}^{(N)}X_t^{(N)}\right).
\end{equation}
For $\varrho>0$, let $Q_\varrho:\R^N\to\varrho\Z^N$ be coordinatewise
nearest-grid rounding, with ties rounded toward the grid point of smaller
absolute value. The rounded vector chain is
\begin{equation}\label{eq:rounded-chain}
  Y_0^{(N)}=Q_\varrho x,
  \qquad
  Y_{t+1}^{(N)}
  =
  Q_\varrho\left(
    \tanh\left(\mathsf W_{t+1,A}^{(N)}Y_t^{(N)}\right)
  \right).
\end{equation}
We denote the transition kernel of $Y^{(N)}$ by $P_{\varrho,A,N}$ and its
absorption time by
\begin{equation}\label{eq:absorption-time}
  \tau_\varrho^{(N)}
  :=
  \inf\{t\geq0:Y_t^{(N)}=0\}.
\end{equation}

The point $0$ is absorbing, and hence $\delta_0$ is invariant. With the
convention
\begin{equation*}
  \norm{\mu-\nu}_{\TV}
  :=
  \sup_{B\in\mathcal B(\R^N)}\abs{\mu(B)-\nu(B)},
\end{equation*}
we have the exact identity
\begin{equation}\label{eq:tv-absorption}
  \norm{
    P_{\varrho,A,N}^t(Q_\varrho x,\cdot)-\delta_0
  }_{\TV}
  =
  \Pp(\tau_\varrho^{(N)}>t).
\end{equation}
Thus cutoff to $\delta_0$ is equivalent to concentration of the absorption
time.

\subsection{Scalar reductions}
\label{subsec:radius-observables}

\emph{The unrounded chain.}
By Gaussian rotational invariance, the conditional law of the next
preactivation vector depends on the present state only through its Euclidean
norm. For $u\in\R^N$, set
\begin{equation*}
  r_N(u):=\frac1N\norm{u}_2^2.
\end{equation*}
Define
\begin{equation*}
  R_t^{(N)}:=\norm{X_t^{(N)}}_2,
  \qquad
  Z_t^{(N)}:=r_N(X_t^{(N)}).
\end{equation*}
Since $R_t^{(N)}=\sqrt{NZ_t^{(N)}}$, both processes are Markov. We call
$R^{(N)}$ the unrounded radius chain and $Z^{(N)}$ the unrounded
squared-radius chain.
Conditionally on $X_t^{(N)}$, the vector
$\mathsf W_{t+1,A}^{(N)}X_t^{(N)}$ has the same law as
$(AR_t^{(N)}/\sqrt N)G_{t+1}$, where $G_{t+1}$ is a standard Gaussian vector in
$\R^N$. Equivalently, given $Z_t^{(N)}=q$,
\begin{equation}\label{eq:q-chain}
  Z_{t+1}^{(N)}
  \stackrel{d}{=}
  \frac1N\sum_{i=1}^N
  \tanh^2\left(A\sqrt{q}\,G_{t+1,i}\right).
\end{equation}
For $t\geq1$, the chain $Z^{(N)}$ takes values in $[0,1]$.

For $g=(g_1,\ldots,g_N)\in\R^N$, define
\begin{equation}\label{eq:gaussian-random-map}
  F_{A,N}(q,g)
  :=
  \frac1N\sum_{i=1}^N \tanh^2(A\sqrt q\,g_i),
  \qquad q\geq0.
\end{equation}
Let $K_{A,N}$ be the transition kernel of $Z^{(N)}$ on $[0,1]$, given by
\begin{equation}\label{eq:gaussian-q-kernel}
  K_{A,N}(q,B)
  =
  \int_{\R^N}
  \one_{\{F_{A,N}(q,g)\in B\}}\,
  \mathcal G_N(\mathrm{d}g),
\end{equation}
where $\mathcal G_N$ is standard Gaussian measure on $\R^N$. Thus, with
$G_{t+1}\sim\mathcal N(0,I_N)$,
\begin{equation}\label{eq:gaussian-random-map-chain}
  Z_{t+1}^{(N)}
  =
  F_{A,N}(Z_t^{(N)},G_{t+1}).
\end{equation}
For fixed $g$, the map $q\mapsto F_{A,N}(q,g)$ is nondecreasing.

The conditional mean map of $Z^{(N)}$ is
\begin{equation}\label{eq:mean-field-map}
  V_A(q):=\Ee\tanh^2(A\sqrt q\,G),
  \qquad
  G\sim\mathcal N(0,1).
\end{equation}
For fixed $q$, the conditional moments are
\begin{equation}\label{eq:gaussian-q-moments}
  \Ee\left[Z_{t+1}^{(N)}\mid Z_t^{(N)}=q\right]=V_A(q),
  \qquad
  \Var\left(Z_{t+1}^{(N)}\mid Z_t^{(N)}=q\right)
  =
  \frac1N
  \Var\left(\tanh^2(A\sqrt q\,G)\right).
\end{equation}
Moreover, near zero, $V_A$ expands as
\begin{equation}\label{eq:gaussian-mean-field-near-zero}
  V_A(q)
  =
  A^2q-2A^4q^2+O(q^3),
  \qquad q\downarrow0.
\end{equation}
With respect to the natural filtration of $Z^{(N)}$, define the centered
one-step fluctuation by
\begin{equation}\label{eq:gaussian-noise-average}
  \zeta_{t+1}^{(N)}
  =
  \frac1N\sum_{i=1}^N
  \left[
    \tanh^2(A\sqrt{Z_t^{(N)}}\,G_{t+1,i})
    -V_A(Z_t^{(N)})
  \right],
\end{equation}
so
\begin{equation}\label{eq:gaussian-noise-decomposition}
  Z_{t+1}^{(N)}
  =
  V_A(Z_t^{(N)})+\zeta_{t+1}^{(N)}.
\end{equation}
Conditionally on $\mathcal F_t$, the summands in
\cref{eq:gaussian-noise-average} are independent and centered, and each lies
in an interval of length at most one. Hence
\begin{equation}\label{eq:gaussian-noise-moments}
  \Ee\left[\zeta_{t+1}^{(N)}\mid\mathcal F_t\right]=0,
  \qquad
  \Var\left(\zeta_{t+1}^{(N)}\mid\mathcal F_t\right)
  \leq \frac1{4N},
\end{equation}
and Hoeffding's inequality gives, for every $u>0$,
\begin{equation}\label{eq:gaussian-noise-hoeffding}
  \Pp\left(
    \abs{\zeta_{t+1}^{(N)}}>u
    \,\middle|\,\mathcal F_t
  \right)
  \leq
  2\exp(-2Nu^2).
\end{equation}

\emph{The rounded chain.}
Let $Q_1:\R\to\Z$ be nearest-integer rounding with the same tie convention as
$Q_\varrho$. Then
\begin{equation*}
  Q_\varrho(y)_i=\varrho Q_1(y_i/\varrho),
  \qquad
  Q_1(u)=0 \quad\Longleftrightarrow\quad \abs u\leq\frac12,
  \qquad
  \abs{Q_1(u)-u}\leq\frac12.
\end{equation*}
Define
\begin{equation}\label{eq:subcritical-exact-grid-radius}
  \mathscr H_t^{(N)}
  :=
  \frac1{N\varrho^2}\norm{Y_t^{(N)}}_2^2 .
\end{equation}
For $h\geq0$ and $g\in\R$, put
\begin{equation}\label{eq:rounded-coordinate-observable}
  U_{A,\varrho}(h,g)
  :=
  Q_1\left(
    \varrho^{-1}\tanh(\varrho A\sqrt h\,g)
  \right)^2.
\end{equation}
Conditionally on $\mathscr H_t^{(N)}=h$, the Gaussian representation and
\cref{eq:rounded-coordinate-observable} give
\begin{equation}\label{eq:subcritical-exact-grid-transition}
  \mathscr H_{t+1}^{(N)}
  \stackrel{d}{=}
  \frac1N\sum_{i=1}^N U_{A,\varrho}(h,G_{t+1,i}),
\end{equation}
where the variables $G_{t+1,i}$ are independent standard Gaussians. Thus
$\mathscr H^{(N)}$ is Markov, and its absorption time satisfies
\begin{equation}\label{eq:subcritical-exact-absorption}
  \tau_\varrho^{(N)}
  =
  \inf\{t\in\N_0:\mathscr H_t^{(N)}=0\}.
\end{equation}

Its mean map is
\begin{equation}\label{eq:rounded-radius-map}
  V_{A,\varrho}(h)
  :=
  \Ee\left[
    Q_1\left(
      \varrho^{-1}\tanh(\varrho A\sqrt h\,G)
    \right)^2
  \right],
  \qquad h\geq0.
\end{equation}
With respect to the natural filtration of $\mathscr H^{(N)}$, define
\begin{equation}\label{eq:rounded-noise-decomposition}
  \widehat\zeta_{t+1}^{(N)}
  :=
  \mathscr H_{t+1}^{(N)}
  -V_{A,\varrho}(\mathscr H_t^{(N)}).
\end{equation}
If
$
  B_\varrho
  :=
  \max_{\abs u\leq\varrho^{-1}}Q_1(u)^2,
$
then $0\leq U_{A,\varrho}(h,G)\leq B_\varrho$, and hence
\begin{equation}\label{eq:rounded-noise-moments}
  \Ee[\widehat\zeta_{t+1}^{(N)}\mid\mathcal F_t]=0,
  \qquad
  \Ee[(\widehat\zeta_{t+1}^{(N)})^2\mid\mathcal F_t]
  \leq
  \frac{B_\varrho^2}{4N}.
\end{equation}
Moreover, Hoeffding's inequality yields
\begin{equation}\label{eq:rounded-noise-hoeffding}
  \Pp\left(
    \abs{\widehat\zeta_{t+1}^{(N)}}>u
    \,\middle|\,\mathcal F_t
  \right)
  \leq
  2\exp\left(-\frac{2Nu^2}{B_\varrho^2}\right),
  \qquad u>0.
\end{equation}

\subsection{Thresholds and mean maps}
\label{subsec:thresholds-mean-field}

The large-dimension threshold is $A=1$. Recall that $\chi_N$ denotes the
Euclidean norm of a standard Gaussian vector in $\R^N$. At fixed width, set
\begin{equation}\label{eq:gaussian-ell}
  \ell_{A,N}:=\frac A{\sqrt N}\chi_N.
\end{equation}
Define
\begin{equation*}
  \gamma_{A,N}:=-\Ee\log\ell_{A,N}.
\end{equation*}
The unique zero of $\gamma_{A,N}$ is
\begin{equation}\label{eq:fixed-width-critical}
  A=A_c(N)
  =
  \sqrt{\frac N2}\,
  \exp\left(-\frac12\psi\left(\frac N2\right)\right),
\end{equation}
where $\psi$ is the digamma function. Thus $A<A_c(N)$ is equivalent to
$\gamma_{A,N}>0$, and
\begin{equation}\label{eq:gaussian-log-multiplier-moments}
  \gamma_{A,N}=\log\frac{A_c(N)}A,
  \qquad
  \Var(\log\ell_{A,N})
  =
  \Var(\log\chi_N)
  =
  \frac14\psi_1\left(\frac N2\right).
\end{equation}
Here $\psi_1=\psi'$ is the trigamma function. Moreover,
\begin{equation}\label{eq:gaussian-squared-radius-multiplier}
  \frac{F_{A,N}(q,G)}{q}
  \to
  A^2\frac{\chi_N^2}{N}
  \qquad\mbox{as }q\downarrow0
\end{equation}
almost surely. In particular, $A_c(N)>1$ for finite $N$ and
$A_c(N)\downarrow1$ as $N\to\infty$.

Define
\begin{equation}\label{eq:subcritical-lattice-map}
  L_A(h):=\Ee[Q_1(A\sqrt h\,G)^2],
\end{equation}
and
\begin{equation}\label{eq:subcritical-lattice-threshold}
  m(\alpha):=\Ee[Q_1(\alpha G)^2],
  \qquad
  A_{\mathrm{lat}}^2
  :=
  \inf_{\alpha>0}\frac{\alpha^2}{m(\alpha)}.
\end{equation}
Thus $L_A(h)=m(A\sqrt h)$. The condition $A<A_{\mathrm{lat}}$ is
equivalent to $L_A(h)<h$ for every $h>0$.

We use the following fixed-point and linearization facts for $V_A$.
\begin{lemma}
\label{lem:gaussian-mean-field-concavity}
  The map $V_A$ is strictly increasing on $[0,1]$ and strictly concave on
  $[0,1]$, with $V_A(0)=0$, $V_A'(0+)=A^2$, and $V_A(1)<1$. If
  $A\leq1$, then $0$ is its only fixed point. If $A>1$, then it has a unique
  nonzero fixed point $q_\ast\in(0,1)$ satisfying
  \begin{equation}\label{eq:gaussian-fixed-point}
    q_\ast=V_A(q_\ast),
    \qquad
    \mu_A:=V_A'(q_\ast)\in(0,1).
  \end{equation}
  If $A>1$, every orbit of $V_A$ started from $(0,1]$ converges to $q_\ast$.
  In this case, for every $q_0\in(0,1]\setminus\{q_\ast\}$, there is a
  continuous function $C$ in a neighborhood of $q_0$, with $C(q_0)\neq0$,
  such that
  \begin{equation}\label{eq:gaussian-deterministic-asymptotic}
    V_A^t(q)-q_\ast
    =
    C(q)\mu_A^t+o(\mu_A^t)
  \end{equation}
  locally uniformly for $q$ near $q_0$.
\end{lemma}

\begin{proof}
  \emph{Step 1: Monotonicity, concavity, and fixed points.}
  Fix $a>0$ and put $f_a(q):=\tanh^2(a\sqrt q)$ for $q>0$. With
  $x=a\sqrt q$,
  \begin{equation*}
    f_a'(q)
    =
    a^2\,\frac{\tanh x}{x}\operatorname{sech}^2 x.
  \end{equation*}
  The functions $x\mapsto\tanh x/x$ and
  $x\mapsto\operatorname{sech}^2x$ are strictly decreasing on
  $(0,\infty)$. Hence $f_a'$ is strictly decreasing and $f_a$ is strictly
  concave. Since
  $\tanh^2(A\sqrt q\,G)=f_{A\abs G}(q)$ and $\Pp(\abs G>0)=1$, integration
  gives strict concavity of $V_A$. The same formula gives strict
  monotonicity. The expansion in \cref{eq:gaussian-mean-field-near-zero} gives
  $V_A'(0+)=A^2$, while $V_A(1)<1$ since
  $\tanh^2(AG)<1$ almost surely.

  Strict concavity now gives no positive fixed point when $A\leq1$ and exactly
  one when $A>1$. At the positive fixed point,
  concavity gives
  \begin{equation*}
    0<V_A'(q_\ast)<\frac{V_A(q_\ast)}{q_\ast}=1.
  \end{equation*}

  \emph{Step 2: Convergence of the orbit.}
  Suppose that $A>1$. Then $V_A(q)>q$ for $0<q<q_\ast$ and
  $V_A(q)<q$ for $q_\ast<q\leq1$. The orbit started at $q_\ast$ is constant.
  Since $V_A$ is increasing, every other orbit is monotone and remains on the
  same side of $q_\ast$. It is therefore bounded and converges to $q_\ast$.

  \emph{Step 3: Linearization at $q_\ast$.}
  We prove \cref{eq:gaussian-deterministic-asymptotic} by the standard Koenigs
  product argument; see
  \cite[Ch.~II, Sec.~2, Theorem~2.1]{carleson1993complex}. The Taylor expansion
  at the attracting fixed point is
  \begin{equation*}
    V_A(q)-q_\ast
    =
    \mu_A(q-q_\ast)+O((q-q_\ast)^2)
  \end{equation*}
  near $q_\ast$. Fix $q_0\in(0,1]\setminus\{q_\ast\}$. By Step~2, there are
  a neighborhood $U$ of $q_0$, an integer $s\geq0$, and $\rho\in(0,1)$ such
  that, for every $q\in U$ and $t\geq s$,
  \begin{equation*}
    \abs{V_A^{t+1}(q)-q_\ast}
    \leq
    \rho\abs{V_A^t(q)-q_\ast}.
  \end{equation*}
  For $t\geq s$, we have
  \begin{equation*}
    \frac{V_A^{t+1}(q)-q_\ast}{\mu_A^{t+1}}
    =
    \frac{V_A^t(q)-q_\ast}{\mu_A^t}
    \frac{V_A(V_A^t(q))-q_\ast}
    {\mu_A(V_A^t(q)-q_\ast)} .
  \end{equation*}
  The last factor is $1+O(\abs{V_A^t(q)-q_\ast})$, uniformly for $q$ in a
  sufficiently small choice of $U$. The contraction estimate implies that
  \begin{equation*}
    \sum_{t=s}^\infty
    \abs{V_A^t(q)-q_\ast}
  \end{equation*}
  converges uniformly for $q\in U$. Hence the Koenigs product converges
  uniformly on $U$, and we may define
  \begin{equation*}
    C(q)
    =
    (V_A^{s}(q)-q_\ast)\mu_A^{-s}
    \prod_{j=s}^\infty
    \frac{V_A(V_A^j(q))-V_A(q_\ast)}
    {\mu_A(V_A^j(q)-q_\ast)}
  \end{equation*}
  for $q\in U$. The finite products telescope, and therefore
  \begin{equation*}
    \frac{V_A^t(q)-q_\ast}{\mu_A^t}
    \longrightarrow
    C(q)
  \end{equation*}
  uniformly on $U$. The product representation shows that $C$ is continuous.
  Moreover, $C(q_0)\neq0$ because the orbit approaches $q_\ast$ monotonically
  without reaching or crossing it. This proves
  \cref{eq:gaussian-deterministic-asymptotic} and completes the proof.
\end{proof}

\subsection{Common scalar estimates}
\label{subsec:common-scalar-estimates}

Let $T_N\in\N$, let $(\mathcal F_t)_{0\leq t\leq T_N}$ be a filtration,
and let $(\mathcal R_{t,N})_{0\leq t\leq T_N}$ be adapted. Let $\tau_N$ be a
stopping time. Suppose that, on $\{t<\tau_N\}$ and for $0\leq t<T_N$,
\begin{equation}\label{eq:common-tracking-recursion}
  \mathcal R_{t+1,N}
  =
  A_{t,N}\mathcal R_{t,N}+\eta_{t+1,N},
\end{equation}
where $A_{t,N}$ is $\mathcal F_t$-measurable,
$\abs{A_{t,N}}\leq\alpha_{t,N}$ for deterministic $\alpha_{t,N}\geq0$, and
\begin{equation}\label{eq:common-tracking-noise}
  \Ee[\eta_{t+1,N}\mid\mathcal F_t]=0,
  \qquad
  \Ee[\eta_{t+1,N}^2\mid\mathcal F_t]
  \leq
  \frac{C\sigma_{t,N}^2}{N}
\end{equation}
for a constant $C<\infty$ and deterministic $\sigma_{t,N}\geq0$. Define
\begin{equation}\label{eq:common-tracking-amplification}
  K_N:=1\vee\max_{0\leq s<t\leq T_N}
  \bigg\{
  \prod_{u=s}^{t-1}(1\vee\alpha_{u,N}),
  \sigma_{s,N}
  \prod_{u=s+1}^{t-1}(1\vee\alpha_{u,N})
  \bigg\}.
\end{equation}

The following stopped second-moment estimate controls deviations from a
deterministic orbit up to a prescribed stopping time.
\begin{lemma}
\label{lem:common-stopped-orbit-tracking}
  Under the preceding assumptions, set
  $\overline{\mathcal R}_{t,N}:=\mathcal R_{t\wedge\tau_N,N}$. Then
  \begin{equation}\label{eq:common-stopped-orbit-tracking}
    \Ee\overline{\mathcal R}_{T_N,N}^2
    \leq
    CK_N^2\left(
      \Ee\mathcal R_{0,N}^2+\frac{T_N}{N}
    \right).
  \end{equation}
  In addition, for every $t\leq T_N$,
  \begin{equation}\label{eq:common-killed-orbit-tracking}
    \Ee\left[
      \mathcal R_{t,N}^2\one_{\{\tau_N>t\}}
    \right]
    \leq{}
    \left(\prod_{u=0}^{t-1}\alpha_{u,N}^2\right)
      \Ee\mathcal R_{0,N}^2
    +\frac{C}{N}\sum_{s=0}^{t-1}
      \sigma_{s,N}^2
      \prod_{u=s+1}^{t-1}\alpha_{u,N}^2,
  \end{equation}
  with empty products equal to one and empty sums equal to zero.
  In particular, if
  $\tau_N=\inf\{t:\abs{\mathcal R_{t,N}}>\delta\}$ for some $\delta>0$,
  then
  \begin{equation}\label{eq:common-orbit-exit-bound}
    \Pp(\tau_N\leq T_N)
    \leq
    \frac{CK_N^2}{\delta^2}
    \left(
      \Ee\mathcal R_{0,N}^2+\frac{T_N}{N}
    \right).
  \end{equation}
\end{lemma}

If $\Ee\log M_1<0$, the affine recursion below is contractive
\cite{buraczewski-damek-mikosch2016}. The next lemma quantifies its entrance
into a bounded set.
\begin{lemma}
\label{lem:negative-drift-affine-entrance}
  Let $(M_j)_{j\geq1}$ be independent and identically distributed positive
  random variables with $\Ee\log M_1<0$, and assume $\log(M_1+\eps)$ has
  exponential moments in a neighborhood of the origin for each fixed
  $\eps>0$. For every $b<\infty$, there exist $K_\ast<\infty$ and
  $c_1,c_2,c_3>0$, depending only on $b$ and the law of $M_1$, such that the
  affine recursion
  \begin{equation*}
    U_{t+1}=M_{t+1}U_t+b,
    \qquad
    U_0=K,
  \end{equation*}
  satisfies, uniformly for $K\geq2$,
  \begin{equation}\label{eq:negative-drift-affine-entrance}
    \Pp\left(\inf\{t:U_t\leq K_\ast\}>c_1\log K+r\right)
    \leq
    c_2 e^{-c_3r},
    \qquad r\geq0.
  \end{equation}
\end{lemma}

See \cref{app:common-stopped-orbit-tracking,app:common-affine-entrance} for the
proofs.

\subsection{Transition and reconstruction kernels}
\label{subsec:vector-radius-kernels}

For $q\in[0,1]$, let $J_{A,N}(q,\cdot)$ be the law of
$(\tanh(A\sqrt q\,G_1),\ldots,\tanh(A\sqrt q\,G_N))$, where
$G\sim\mathcal N(0,I_N)$.
For a probability measure $\eta$ on $[0,1]$, we write
$\eta J_{A,N}$ for the probability measure on $[-1,1]^N$ defined by
\begin{equation*}
  \eta J_{A,N}(B)
  :=
  \int_0^1 J_{A,N}(q,B)\,\eta(\mathrm{d}q).
\end{equation*}

The next proposition transfers invariant laws and total variation estimates
between $K_{A,N}$ and $P_{A,N}$.
\begin{proposition}
\label{prop:gaussian-tv-reduction}
  Let $P_{A,N}$ be the transition kernel of $X^{(N)}$ in
  \cref{eq:unrounded-chain}. Then, for every bounded measurable
  $f:[-1,1]^N\to\R$,
  \begin{equation}\label{eq:gaussian-vector-kernel-factorization}
    P_{A,N}f(x)
    =
    \int_{[-1,1]^N} f(z)\,J_{A,N}(r_N(x),\mathrm{d}z).
  \end{equation}
  Consequently, for every deterministic $x\in[-1,1]^N$, with
  $q=r_N(x)$, and every $t\geq1$,
  \begin{equation}\label{eq:gaussian-vector-radius-laws}
    P_{A,N}^t(x,\cdot)=K_{A,N}^{t-1}(q,\cdot)J_{A,N},
    \qquad
    (r_N)_\#P_{A,N}^t(x,\cdot)=K_{A,N}^{t}(q,\cdot).
  \end{equation}

  If $\nu_{A,N}$ is invariant for $K_{A,N}$, then
  $\pi_{A,N}:=\nu_{A,N}J_{A,N}$, that is,
  \begin{equation}\label{eq:gaussian-vector-stationary-from-q}
    \pi_{A,N}
    =
    \int_0^1 J_{A,N}(q,\cdot)\,\nu_{A,N}(\mathrm{d}q)
  \end{equation}
  is an invariant law for $P_{A,N}$. Conversely, if $\pi_{A,N}$ is invariant
  for $P_{A,N}$ and $\nu_{A,N}:=(r_N)_\#\pi_{A,N}$, then $\nu_{A,N}$ is
  invariant for $K_{A,N}$ and $\pi_{A,N}$ is given by
  \cref{eq:gaussian-vector-stationary-from-q}.

  For this pair of invariant laws and $t\geq1$,
  \begin{equation}\label{eq:gaussian-tv-bracket}
    \norm{K_{A,N}^t(q,\cdot)-\nu_{A,N}}_{\TV}
    \leq
    \norm{P_{A,N}^t(x,\cdot)-\pi_{A,N}}_{\TV}
    \leq
    \norm{K_{A,N}^{t-1}(q,\cdot)-\nu_{A,N}}_{\TV}.
  \end{equation}
\end{proposition}

\begin{proof}
  For a fixed vector $x$, the random vector $\mathsf W_A^{(N)}x$ has law
  $A\sqrt{r_N(x)}G$. Applying $\tanh$ coordinatewise gives
  \cref{eq:gaussian-vector-kernel-factorization}. Since
  $(r_N)_\#J_{A,N}(q,\cdot)=K_{A,N}(q,\cdot)$, induction gives
  \cref{eq:gaussian-vector-radius-laws}. The same identities also give the
  statement about invariant laws. Indeed, if $\nu_{A,N}K_{A,N}=\nu_{A,N}$,
  then
  \begin{equation*}
    \left(\nu_{A,N}J_{A,N}\right)P_{A,N}
    =
    \left(\nu_{A,N}K_{A,N}\right)J_{A,N}
    =
    \nu_{A,N}J_{A,N}.
  \end{equation*}
  Conversely, if $\pi_{A,N}P_{A,N}=\pi_{A,N}$ and
  $\nu_{A,N}=(r_N)_\#\pi_{A,N}$, then
  \begin{equation*}
    \pi_{A,N}
    =
    \pi_{A,N}P_{A,N}
    =
    \nu_{A,N}J_{A,N}
  \end{equation*}
  and, after applying $(r_N)_\#$, we obtain
  $\nu_{A,N}K_{A,N}=\nu_{A,N}$.

  The lower bound in \cref{eq:gaussian-tv-bracket} is the contraction of total
  variation under the measurable map $r_N$. For the upper bound, use
  \cref{eq:gaussian-vector-radius-laws} and
  \cref{eq:gaussian-vector-stationary-from-q}. The reconstruction kernel
  $J_{A,N}$ contracts total variation, and hence
  \begin{equation*}
    \norm{
      K_{A,N}^{t-1}(q,\cdot)J_{A,N}
      -
      \nu_{A,N}J_{A,N}
    }_{\TV}
    \leq
    \norm{K_{A,N}^{t-1}(q,\cdot)-\nu_{A,N}}_{\TV}.
  \end{equation*}
\end{proof}

\subsection{Invariant laws}
\label{subsec:invariant-laws}

For later use, we write the invariance relation for $K_{A,N}$ on $(0,1]$ in
integral form:
\begin{equation}\label{eq:gaussian-q-stationary-equation}
  \int_0^1\phi(q)\,\nu_{A,N}(\mathrm{d}q)
  =
  \int_0^1
  \int_{\R^N}
  \phi\left(F_{A,N}(q,g)\right)
  \mathcal G_N(\mathrm{d}g)\,
  \nu_{A,N}(\mathrm{d}q)
\end{equation}
for every bounded measurable $\phi:[0,1]\to\R$.

For $A>0$, $N\geq1$, $q\in[0,1]$, and $T\geq1$, define
\begin{equation}\label{eq:gaussian-cesaro-selection}
  \overline\nu_{T}^{\,q}
  :=
  \frac1T\sum_{t=0}^{T-1}K_{A,N}^t(q,\cdot).
\end{equation}

\begin{proposition}
\label{prop:gaussian-compactness-selection}
  The family $(\overline\nu_T^{\,q})_{T\geq1}$ is relatively compact for weak
  convergence. Every subsequential limit $\overline\nu$ is invariant for
  $K_{A,N}$ and satisfies the stationary equation
  \cref{eq:gaussian-q-stationary-equation}. We denote by
  $\mathcal I_{A,N}(q)$ the nonempty set of all such limits.

  If $\overline\nu\in\mathcal I_{A,N}(q)$ and
  $\overline\nu\neq\delta_0$, then
  \begin{equation*}
    \nu^+
    :=
    \frac{\overline\nu|_{(0,1]}}{1-\overline\nu(\{0\})}
  \end{equation*}
  is invariant for $K_{A,N}$ on $(0,1]$. We write
  $\mathcal I_{A,N}^+(q)$ for the collection of all such nonzero components.
  For every $\nu^+\in\mathcal I_{A,N}^+(q)$, the measure
  \begin{equation}\label{eq:gaussian-compactness-vector-law}
    \pi^+:=\int_0^1 J_{A,N}(s,\cdot)\,\nu^+(\mathrm{d}s)
  \end{equation}
  is an invariant law for $X^{(N)}$. Equivalently,
  \begin{equation}\label{eq:gaussian-vector-weak-stationary-equation}
    \int_{[-1,1]^N}\varphi(y)\,\pi^+(\mathrm{d}y)
    =
    \int_{[-1,1]^N}
    \int_{\R^N}
    \varphi\bigl(\tanh(A\sqrt{r_N(x)}\,g)\bigr)\,
    \mathcal G_N(\mathrm{d}g)\,
    \pi^+(\mathrm{d}x)
  \end{equation}
  for every bounded measurable $\varphi:[-1,1]^N\to\R$.
\end{proposition}

For $A>A_c(N)$, a negative-moment Foster--Lyapunov estimate shows that Cesàro
limits started from positive radii do not charge the origin.
\begin{proposition}
\label{prop:gaussian-nonzero-invariant-existence}
  Let $N\geq1$ and assume $A>A_c(N)$. Then, for every $q\in(0,1]$,
  $\mathcal I_{A,N}^+(q)$ is nonempty. In particular, $K_{A,N}$ admits at
  least one invariant probability on $(0,1]$, and $X^{(N)}$ admits a
  corresponding nonzero invariant law.
\end{proposition}

The proofs of \cref{prop:gaussian-compactness-selection,prop:gaussian-nonzero-invariant-existence} are in \cref{app:common-invariant-law}.

\clearpage
\subsection{Notation guide}
\label{subsec:notation-guide}

We use the following notation throughout the paper.

\begin{table}[H]
  \centering
  \small
  \setlength{\tabcolsep}{4pt}
  \renewcommand{\arraystretch}{1.15}
  \begin{tabular}{@{}p{0.22\textwidth}p{0.70\textwidth}@{}}
    \toprule
    Symbol & Definition \\
    \midrule
    $\mathsf W_{t,A}^{(N)}$
      & matrix of weights at layer $t$, with independent
        $\mathcal N(0,A^2/N)$ entries,
        \cref{eq:gaussian-weights}. \\
    $\mathcal G_N$
      & law of a standard Gaussian vector in $\R^N$,
        \cref{eq:gaussian-q-kernel}. \\
    $X_t^{(N)}$, $Y_t^{(N)}$
      & vector chains before and after rounding,
        \cref{eq:unrounded-chain,eq:rounded-chain}. \\
    $Q_\varrho$, $Q_1$
      & nearest-grid rounding with mesh $\varrho$ or $1$,
        \cref{eq:rounded-chain,eq:subcritical-exact-grid-radius}. \\
    $\tau_\varrho^{(N)}$
      & first time at which $Y_t^{(N)}=0$, \cref{eq:absorption-time}. \\
    $Z_t^{(N)}$, $R_t^{(N)}$
      & normalized squared radius and Euclidean radius of $X_t^{(N)}$,
        introduced before \cref{eq:q-chain}. \\
    $\mathscr H_t^{(N)}$, $U_{A,\varrho}$
      & normalized squared radius of $Y_t^{(N)}$ and the contribution of one
        coordinate to the next radius,
        \cref{eq:subcritical-exact-grid-radius,eq:rounded-coordinate-observable}. \\
    $F_{A,N}$, $K_{A,N}$
      & update map for $Z^{(N)}$ and its transition kernel,
        \cref{eq:gaussian-random-map,eq:gaussian-q-kernel}. \\
    $P_{A,N}$, $P_{\varrho,A,N}$
      & transition kernels of $X^{(N)}$ and $Y^{(N)}$, defined before
        \cref{eq:gaussian-vector-kernel-factorization,eq:absorption-time}. \\
    $J_{A,N}$
      & law of $X_{t+1}^{(N)}$ given $Z_t^{(N)}=q$, introduced before
        \cref{prop:gaussian-tv-reduction}. \\
    $V_A$
      & conditional mean of $Z_{t+1}^{(N)}$ given $Z_t^{(N)}=q$,
        \cref{eq:mean-field-map}. \\
    $A_c(N)$, $A_c$
      & solution of $\Ee\log\ell_{A,N}=0$, with $A_c=A_c(N)$ when $N$ is
        fixed, \cref{eq:fixed-width-critical}. \\
    $\ell_{A,N}$, $\gamma_{A,N}$, $\sigma_N$
      & multiplier near zero, with $\gamma_{A,N}=-\Ee\log\ell_{A,N}$ and
        $\sigma_N^2=\Var(\log\ell_{A,N})$,
        \cref{eq:gaussian-ell,eq:gaussian-log-multiplier-moments}. \\
    $\chi_N$
      & norm of a standard Gaussian vector in $\R^N$,
        \cref{eq:gaussian-ell}. \\
    $L_A$
      & upper comparison map for $V_{A,\varrho}$,
        \cref{eq:subcritical-lattice-map}. \\
    $V_{A,\varrho}$
      & conditional mean of $\mathscr H_{t+1}^{(N)}$ given
        $\mathscr H_t^{(N)}=h$, \cref{eq:rounded-radius-map}. \\
    $\bar h_{t,N}$, $\bar t_N$
      & orbit of $V_{A,\varrho}$ and the first time it falls below
        $(\log N)^{-1}$,
        \cref{eq:subcritical-exact-lattice-orbit,eq:subcritical-exact-hitting-time}. \\
    $A_{\mathrm{lat}}$
      & threshold below which $L_A(h)<h$ for every $h>0$,
        \cref{eq:subcritical-lattice-threshold}. \\
    $A_{\mathrm{ex}}(\varrho)$, $\mathcal O_{A,\varrho}$
      & threshold above which
        $\mathcal O_{A,\varrho}=\{h:V_{A,\varrho}(h)>h\}$ is nonempty,
        \cref{eq:metastable-exact-threshold,eq:metastable-positive-region}. \\
    $q_\ast$, $\mu_A$, $C(q_0)$
      & fixed point of $V_A$, derivative at that point, and coefficient in
        the asymptotic for $V_A^t(q_0)$,
        \cref{eq:gaussian-fixed-point,eq:gaussian-deterministic-asymptotic}. \\
    $\beta_{A,N}$
      & positive solution of the Cramér equation $r_{A,N}(\beta)=1$,
        \cref{eq:nd-beta-equation,eq:nd-beta-equation-explicit}. \\
    $\nu_{A,N}$, $\pi_{A,N}$
      & invariant laws of $Z^{(N)}$ and $X^{(N)}$ away from zero,
        \cref{eq:gaussian-q-stationary-equation,eq:gaussian-vector-stationary-from-q}. \\
    \bottomrule
  \end{tabular}
  \caption{Frequently used notation.}
  \label{tab:notation-guide}
\end{table}

\section{Subcritical cutoff as the mesh vanishes}
\label{sec:subcritical-precision}

We prove total variation cutoff to the absorbing state for $Y^{(N)}$ as the
mesh vanishes at fixed width.
Throughout the section $N\geq1$ and
$0<A<A_c(N)$ are fixed, and $x_0\in\R^N\setminus\{0\}$ is deterministic. We
write
\begin{equation*}
  R_t:=\norm{X_t^{(N)}}_2,
  \qquad
  \widehat R_t:=\norm{Y_t^{(N)}}_2,
\end{equation*}
where $X^{(N)}$ starts from $x_0$ and $Y^{(N)}$ starts from $Q_\varrho x_0$.
Generic constants $c,C,c_N,C_N$ may change from line to line. A
subscript $N$ indicates dependence only on the fixed width unless further
dependence is displayed.

We write $-\log R_t$ as a positive-drift random walk plus an almost surely
finite correction. A synchronous coupling transfers its first-passage
asymptotics to $\widehat R_t$ down to the grid scale. A final estimate then
gives absorption.

\subsection{Unrounded first-passage asymptotics}
\label{subsec:rounded-gaussian-first-passage}

Let $(\chi_{N,j})_{j\geq1}$ be independent copies of $\chi_N$. The fixed-width
subcritical condition is $\gamma_{A,N}>0$. Put
\begin{equation}\label{eq:rounded-gaussian-log-increments}
  \xi_j:=
  -\log\left(\frac A{\sqrt N}\chi_{N,j}\right),
  \qquad
  \sigma_N^2:=\Var(\xi_1).
\end{equation}
By \cref{eq:gaussian-log-multiplier-moments},
\begin{equation}\label{eq:rounded-gaussian-drift-variance-critical}
  \Ee\xi_1=\gamma_{A,N}=\log\frac{A_c(N)}A,
  \qquad
  \sigma_N^2
  =
  \Var(\log\chi_N)
  =
  \frac14\psi_1\left(\frac N2\right),
\end{equation}
where $\psi_1$ is the trigamma function. For $0<\varrho<R_0$, the cutoff scale is
\begin{equation}\label{eq:rounded-gaussian-cutoff-scale}
  L_\varrho:=\log\frac{R_0}{\varrho},
  \qquad
  t_\varrho(a)
  :=
  \max\left\{0,
    \left\lfloor
      \frac{L_\varrho}{\gamma_{A,N}}
      +
      a\,\frac{\sigma_N}{\gamma_{A,N}^{3/2}}\sqrt{L_\varrho}
    \right\rfloor
  \right\} .
\end{equation}
Here and below $\Phi_{\mathrm G}$ denotes the distribution function of a
standard Gaussian random variable.

\begin{lemma}
\label{lem:rounded-gaussian-log-reduction}
  There are an increasing adapted sequence $(E_t)_{t\geq0}$ and an event
  $\Omega_{\log}$ of probability one such that, on $\Omega_{\log}$,
  $R_t>0$ for every $t\geq0$,
  \begin{equation}\label{eq:rounded-gaussian-log-reduction}
    -\log R_t
    =
    -\log R_0
    +
    \sum_{j=1}^t\xi_j
    +
    E_t
    \qquad\mbox{for every }t\geq0
  \end{equation}
  and $E_t\uparrow E_\infty<\infty$. Consequently, $(E_t)_{t\geq0}$ is
  tight.
\end{lemma}

\begin{proof}
  \emph{Step 1: Gaussian normalization.}
  For each fixed $t$, we have $R_t>0$ almost surely. Hence
  \begin{equation*}
    \Omega_{\mathrm{nz}}
    :=
    \bigcap_{t\geq0}\{R_t>0\}
  \end{equation*}
  has probability one. On $\{R_{j-1}>0\}$, define
  \begin{equation*}
    G^{(j)}
    :=
    \frac{\sqrt N}{AR_{j-1}}
    \mathsf W_{j,A}^{(N)}X_{j-1}^{(N)}.
  \end{equation*}
  Extend $G^{(j)}$ arbitrarily across the null event
  $\{R_{j-1}=0\}$. Conditionally on the past, $G^{(j)}$ is a standard
  Gaussian vector whose law does not depend on the past. Therefore,
  $(G^{(j)})_{j\geq1}$ are independent standard Gaussian vectors. With
  $\chi_{N,j}:=\norm{G^{(j)}}_2$, we have
  \begin{equation*}
    R_j
    =
    \norm{
      \tanh\left(\frac A{\sqrt N}G^{(j)}R_{j-1}\right)
    }_2.
  \end{equation*}
  Set $M_j:=(A/\sqrt N)\chi_{N,j}$. The coordinatewise bound
  $\abs{\tanh z}\leq\abs z$ gives
  \begin{equation}\label{eq:rounded-gaussian-linear-domination}
    R_j
    \leq
    M_jR_{j-1}
    \qquad\mbox{for every }j\geq1.
  \end{equation}

  \emph{Step 2: The logarithmic error.}
  Define
  \begin{equation*}
    \eta_j
    :=
    \log\frac{M_jR_{j-1}}{R_j}
    \geq0,
    \qquad
    E_t:=\sum_{j=1}^t\eta_j,
    \qquad
    \xi_j:=-\log M_j,
  \end{equation*}
  Then $E_t$ is increasing, and telescoping gives
  \cref{eq:rounded-gaussian-log-reduction} simultaneously for every $t\geq0$
  on $\Omega_{\mathrm{nz}}$.

  For $u\in\R$, with the value at zero defined by continuity, we have
  \begin{equation*}
    0\leq
    \log\frac{\abs u}{\abs{\tanh u}}
    \leq
    C u^2 .
  \end{equation*}
  Let $u=(u_1,\ldots,u_N)$ and $v_i=\tanh(u_i)$, where the ratios below are
  understood by continuity when $u_i=0$. Then
  \begin{equation*}
    \norm u_2^2
    =
    \sum_{i=1}^N v_i^2\frac{u_i^2}{v_i^2}
    \leq
    \max_{1\leq i\leq N}\frac{u_i^2}{\tanh^2(u_i)}\norm v_2^2,
  \end{equation*}
  and hence $\log(\norm u_2/\norm v_2)\leq C\norm u_2^2$. With
  $u=(A/\sqrt N)G^{(j)}R_{j-1}$ and $v=\tanh(u)$, this yields
  \begin{equation}\label{eq:rounded-gaussian-nonlinear-error}
    0\leq\eta_j
    \leq
    CM_j^2R_{j-1}^2.
  \end{equation}

  \emph{Step 3: Summability of the error.}
  Since $A<A_c(N)$, the variables $\log M_j$ have negative mean. By the strong
  law of large numbers, there is a $\kappa>0$ such that, almost surely,
  \begin{equation*}
    \prod_{i=1}^jM_i\leq e^{-\kappa j}
  \end{equation*}
  for all sufficiently large $j$. By
  \cref{eq:rounded-gaussian-linear-domination}, it follows that
  $R_j\leq R_0e^{-\kappa j}$ for all sufficiently large $j$. The Gaussian tail
  bound and Borel--Cantelli also give
  $\chi_{N,j}^2\leq e^{\kappa j}$ for all sufficiently large $j$. Thus,
  almost surely,
  \begin{equation*}
    C\sum_{j\ \mathrm{large}}M_j^2R_{j-1}^2
    \leq
    C\frac{A^2}{N}R_0^2
    \sum_{j\ \mathrm{large}}e^{-2\kappa(j-1)}\chi_{N,j}^2
    \leq
    C\frac{A^2}{N}R_0^2
    \sum_{j\ \mathrm{large}}e^{-\kappa(j-2)}
    <\infty.
  \end{equation*}
  The finitely many remaining terms are finite almost surely. Intersecting
  $\Omega_{\mathrm{nz}}$ with the probability-one events furnished by the
  strong law and Borel--Cantelli gives an event $\Omega_{\log}$ on which
  \cref{eq:rounded-gaussian-log-reduction} holds for every $t\geq0$ and
  $E_t\uparrow E_\infty<\infty$. Since $E_t\leq E_\infty$ on this event, the
  family $(E_t)_{t\geq0}$ is tight. This completes the proof.
\end{proof}

For a positive threshold $\varepsilon_\varrho$, define
\begin{equation*}
  L_{\varepsilon_\varrho}
  :=
  \log\frac{R_0}{\varepsilon_\varrho},
  \qquad
  \theta_{\varepsilon_\varrho}^{\mathrm{rad}}
  :=
  \inf\{t\geq0:R_t\leq\varepsilon_\varrho\}.
\end{equation*}
The same profile holds for moving thresholds.
\begin{proposition}
\label{prop:rounded-gaussian-threshold-cutoff}
  Let $\varepsilon_\varrho\downarrow0$ satisfy
  \begin{equation*}
    L_{\varepsilon_\varrho}
    =L_\varrho+o(\sqrt{L_\varrho}).
  \end{equation*}
  Then, for every fixed $a\in\R$,
  \begin{equation}\label{eq:rounded-gaussian-moving-threshold-profile}
    \Pp\left(
      \theta_{\varepsilon_\varrho}^{\mathrm{rad}}>t_\varrho(a)
    \right)
    \to
    \Phi_{\mathrm G}(-a),
    \qquad
    \varrho\downarrow0.
  \end{equation}
  The same limit holds if $t_\varrho(a)$ is replaced by
  $\max\{0,\lfloor t_\varrho(a)+u_\varrho\rfloor\}$ for any deterministic
  $u_\varrho=o(\sqrt{L_\varrho})$.
\end{proposition}

\begin{proof}
  \emph{Step 1: Reduction to a perturbed random walk.}
  By \cref{lem:rounded-gaussian-log-reduction}, there is a probability-one
  event on which, simultaneously for every $t\geq0$,
  \begin{equation*}
    -\log R_t
    =
    -\log R_0+S_t+E_t,
    \qquad
    S_t:=\sum_{j=1}^t\xi_j,
  \end{equation*}
  where $0\leq E_t\leq E_\infty<\infty$. Since $R_t>0$ for every $t$ on this
  event, $R_t\leq\varepsilon_\varrho$ is equivalent to
  $S_t+E_t\geq L_{\varepsilon_\varrho}$. Hence, pathwise,
  \begin{equation*}
    \theta_{\varepsilon_\varrho}^{\mathrm{rad}}
    =
    \inf\{t\geq0:S_t+E_t\geq L_{\varepsilon_\varrho}\}.
  \end{equation*}

  \emph{Step 2: First-passage asymptotics.}
  Let $H_u=\inf\{t\geq0:S_t\geq u\}$, and let
  $\nu(u)=\inf\{t\geq0:S_t>u\}$ be the strict first-passage time used in
  \cite{gut2009stopped}. Since the law of $\xi_1$ is absolutely continuous,
  the same is true of $S_t$ for every $t\geq1$. Hence, for every fixed $u>0$,
  \begin{equation*}
    \Pp\bigl(S_t=u\text{ for some }t\geq0\bigr)=0,
  \end{equation*}
  and $H_u=\nu(u)$ almost surely. By
  \cref{eq:rounded-gaussian-drift-variance-critical}, the increments have mean
  $\gamma_{A,N}>0$ and variance $\sigma_N^2<\infty$. Thus,
  \cite[Theorem~3.5.1]{gut2009stopped}, with
  $\mu=\gamma_{A,N}$ and $\sigma=\sigma_N$, gives
  \begin{equation*}
    \frac{H_u-u/\gamma_{A,N}}
    {\sigma_N\gamma_{A,N}^{-3/2}\sqrt u}
    \Rightarrow N(0,1)
    \qquad\mbox{as }u\to\infty.
  \end{equation*}
  The same limit holds with $H_u$ replaced by $H_{u+b_u}$ whenever
  $b_u=o(\sqrt u)$, because the resulting changes in the centering and
  normalization are negligible.

  \emph{Step 3: Comparison of the entrance times.}
  Fix $B<\infty$. On $\{E_\infty\leq B\}$, the entrance time satisfies the
  pathwise bounds
  \begin{equation}\label{eq:rounded-gaussian-hitting-bounds}
    H_{L_{\varepsilon_\varrho}-B}
    \leq
    \theta_{\varepsilon_\varrho}^{\mathrm{rad}}
    \leq
    H_{L_{\varepsilon_\varrho}}.
  \end{equation}
  Therefore, for $t=t_\varrho(a)$,
  \begin{equation*}
    \Pp(H_{L_{\varepsilon_\varrho}-B}>t)-\Pp(E_\infty>B)
    \leq
    \Pp(\theta_{\varepsilon_\varrho}^{\mathrm{rad}}>t)
    \leq
    \Pp(H_{L_{\varepsilon_\varrho}}>t)+\Pp(E_\infty>B).
  \end{equation*}
  Applying the first-passage limit with $u=L_\varrho$, both hitting
  probabilities in the last display converge to $\Phi_{\mathrm G}(-a)$, since
  $L_{\varepsilon_\varrho}-L_\varrho=o(\sqrt{L_\varrho})$ and
  $B=o(\sqrt{L_\varrho})$.
  Thus
  \begin{equation*}
    \Phi_{\mathrm G}(-a)-\Pp(E_\infty>B)
    \leq
    \liminf_{\varrho\downarrow0}
    \Pp(\theta_{\varepsilon_\varrho}^{\mathrm{rad}}>t_\varrho(a))
  \end{equation*}
  and
  \begin{equation*}
    \limsup_{\varrho\downarrow0}
    \Pp(\theta_{\varepsilon_\varrho}^{\mathrm{rad}}>t_\varrho(a))
    \leq
    \Phi_{\mathrm G}(-a)+\Pp(E_\infty>B).
  \end{equation*}
  Sending $B\to\infty$ proves the first assertion.

  \emph{Step 4: Perturbation of the time scale.}
  Fix $\delta>0$. Since
  \begin{equation*}
    \frac{u_\varrho}
    {\sigma_N\gamma_{A,N}^{-3/2}\sqrt{L_\varrho}}
    \longrightarrow0,
  \end{equation*}
  monotonicity places the perturbed survival probability, for all sufficiently
  small $\varrho$, between those at $t_\varrho(a-\delta)$ and
  $t_\varrho(a+\delta)$, up to a uniformly bounded time shift. The first part
  and continuity of $a\mapsto\Phi_{\mathrm G}(-a)$ give the result after
  sending $\delta\downarrow0$. This completes the proof.
\end{proof}

\subsection{Rounding comparison and absorption}
\label{subsec:rounded-gaussian-comparison-absorption}

We couple $Y^{(N)}$ and $X^{(N)}$ with the same Gaussian matrices. The
following estimate bounds the accumulated rounding error by the
grid size times a polylogarithmic factor on the cutoff time scale.
\begin{lemma}
\label{lem:rounded-gaussian-synchronous-comparison}
  Assume $0<A<A_c(N)$, and let $x_0\in\R^N$ be deterministic. Run the
  chain $X^{(N)}$ from $X_0^{(N)}=x_0$ and $Y^{(N)}$ from
  $Y_0^{(N)}=Q_\varrho(x_0)$ with the same Gaussian matrices.
  There is $C_N<\infty$, depending only on $N$, such that the following holds.
  For $0<C_0<\infty$, define
  \begin{equation*}
    T_\varrho:=\left\lfloor C_0\log(1/\varrho)\right\rfloor .
  \end{equation*}
  There exists $p=p(A,N,C_0)$ satisfying
  \begin{equation}\label{eq:rounded-gaussian-p-dependence}
    0
    \leq
    p
    \leq
    C_N(1+C_0)
    \left(
      1+\frac1{\gamma_{A,N}}
    \right),
  \end{equation}
  and
  \begin{equation}\label{eq:rounded-gaussian-coupling-input}
    \Pp\left(
      \max_{0\leq t\leq T_\varrho}
      \norm{Y_t^{(N)}-X_t^{(N)}}_2
      >
      \varrho\,\left(\log(1/\varrho)\right)^p
    \right)
    \to0
  \end{equation}
  as $\varrho\downarrow0$.
\end{lemma}

\begin{proof}
  \emph{Step 1: A scalar recursion for the coupling error.}
  Let $D_t:=Y_t^{(N)}-X_t^{(N)}$ and $d_t:=\norm{D_t}_2$.
  For the common Gaussian matrix at time $t+1$, set
  \begin{equation*}
    \mathcal T_{t+1}(z)
    :=
    \tanh\left(\mathsf W_{t+1,A}^{(N)}z\right).
  \end{equation*}
  The update rules \cref{eq:unrounded-chain,eq:rounded-chain} give
  \begin{equation*}
    d_{t+1}
    \leq
    \norm{\mathcal T_{t+1}(Y_t^{(N)})-
      \mathcal T_{t+1}(X_t^{(N)})}_2
    +\norm{Q_\varrho(\mathcal T_{t+1}(Y_t^{(N)}))-
      \mathcal T_{t+1}(Y_t^{(N)})}_2.
  \end{equation*}
  The second term is at most $c_N\varrho$. By the one-Lipschitz property of
  $\tanh$, the first term is at most
  \begin{equation*}
    \norm{
      \mathsf W_{t+1,A}^{(N)}Y_t^{(N)}
      -
      \mathsf W_{t+1,A}^{(N)}X_t^{(N)}
    }_2
    =
    \norm{\mathsf W_{t+1,A}^{(N)}D_t}_2 .
  \end{equation*}
  Combining these bounds gives
  \begin{equation}\label{eq:rounded-gaussian-difference-recursion}
    d_{t+1}
    \leq
    \norm{\mathsf W_{t+1,A}^{(N)}D_t}_2+c_N\varrho .
  \end{equation}
  Let $\mathcal F_t:=\sigma(\mathsf W_{1,A}^{(N)},\ldots,\mathsf W_{t,A}^{(N)})$,
  so that $D_t$ is $\mathcal F_t$-measurable. Fix a reference unit vector
  $e_1\in\R^N$ and define the $\mathcal F_t$-measurable unit vector and the
  multiplier
  \begin{equation*}
    \theta_t
    :=
    \frac{D_t}{\norm{D_t}_2}\one_{\{D_t\neq0\}}+e_1\one_{\{D_t=0\}},
    \qquad
    M_{t+1}:=\norm{\mathsf W_{t+1,A}^{(N)}\theta_t}_2 .
  \end{equation*}
  Conditionally on $\mathcal F_t$, the vector $\theta_t$ is a fixed unit vector,
  and hence the Gaussian weights give
  \begin{equation*}
    \mathsf W_{t+1,A}^{(N)}\theta_t
    \sim
    \mathcal N(0,(A^2/N)I_N),
  \end{equation*}
  independently of $\mathcal F_t$. Thus $(M_j)_{j\geq1}$ are independent
  copies of $\ell_{A,N}$ from \cref{eq:gaussian-ell}. In particular,
  $\Ee\log M_j=-\gamma_{A,N}<0$. Since both sides vanish on $\{D_t=0\}$,
  we also have
  \begin{equation*}
    \norm{\mathsf W_{t+1,A}^{(N)}D_t}_2=M_{t+1}d_t
  \end{equation*}
  identically. Therefore, \cref{eq:rounded-gaussian-difference-recursion}
  yields $d_{t+1}\leq M_{t+1}d_t+c_N\varrho$. Since
  $d_0\leq c_N\varrho$ and $M_{t+1}\geq0$, induction gives $d_t\leq U_t$ for
  every $t\geq0$, where
  \begin{equation}\label{eq:rounded-gaussian-scalar-error-recursion}
    U_{t+1}=M_{t+1}U_t+c_N\varrho,
    \qquad
    U_0=c_N\varrho .
  \end{equation}

  \emph{Step 2: Control of the multiplier subproducts.}
  Since $M_1=(A/\sqrt N)\chi_N$ and $\chi_N^2$ is
  chi-squared with $N$ degrees of freedom, the Gamma moment formula gives, for
  every $\alpha\geq0$,
  \begin{equation}\label{eq:rounded-gaussian-chi-log-moment}
    \log\Ee M_1^\alpha
    =
    \alpha\log\frac A{\sqrt N}
    +\frac\alpha2\log2
    +\log\Gamma\!\left(\frac{N+\alpha}2\right)
    -\log\Gamma\!\left(\frac N2\right).
  \end{equation}
  In particular, all positive moments are finite. Differentiating
  \cref{eq:rounded-gaussian-chi-log-moment} twice in $\alpha$ gives
  \begin{equation*}
    \frac{\mathrm{d}^2}{\mathrm{d}\alpha^2}\log\Ee M_1^\alpha
    =
    \frac14\,\psi_1\!\left(\frac{N+\alpha}2\right),
  \end{equation*}
  where $\psi_1$ is the trigamma function. Since $\psi_1$ is positive and
  decreasing on $(0,\infty)$, the second derivative is bounded on $[0,1]$ by
  $K:=\tfrac14\psi_1(N/2)$. Define
  \begin{equation*}
    \Lambda(\alpha):=\log\Ee M_1^\alpha,
    \qquad
    \alpha_\gamma
    :=
    \min\left\{1,\frac{\gamma_{A,N}}{2K}\right\}.
  \end{equation*}
  Since $\Lambda(0)=0$ and
  $\Lambda'(0)=\Ee\log M_1=-\gamma_{A,N}$, Taylor's theorem yields
  \begin{equation*}
    \Ee M_1^{\alpha_\gamma}
    \leq
    \exp\left(-\frac{\alpha_\gamma\gamma_{A,N}}2\right)
    =:q<1,
    \qquad
    \alpha_\gamma^{-1}
    \leq
    C_N\left(1+\gamma_{A,N}^{-1}\right).
  \end{equation*}
  Let $L_\varrho:=\log(1/\varrho)$. Markov's inequality and a union bound
  give, for every $p_0>0$,
  \begin{equation*}
    \Pp\left(
      \max_{0\leq s\leq t\leq T_\varrho}
      \prod_{j=s+1}^tM_j
      >
      L_\varrho^{p_0}
    \right)
    \leq
    L_\varrho^{-\alpha_\gamma p_0}
    \sum_{0\leq s\leq t\leq T_\varrho}q^{t-s}
    \leq
    \frac{T_\varrho+1}{1-q}
    L_\varrho^{-\alpha_\gamma p_0}.
  \end{equation*}
  Choose $p_0$ so that $\alpha_\gamma p_0>2$. After increasing $C_N$, we may
  also arrange that
  \begin{equation*}
    p_0
    \leq
    C_N(1+C_0)\left(1+\gamma_{A,N}^{-1}\right).
  \end{equation*}
  Since $T_\varrho=O(L_\varrho)$, the preceding probability tends to zero.
  Hence
  \begin{equation}\label{eq:rounded-gaussian-subproduct-bound}
    \Pp\left(
      \max_{0\leq s\leq t\leq T_\varrho}
      \prod_{j=s+1}^tM_j
      >
      L_\varrho^{p_0}
    \right)
    \to0 .
  \end{equation}

  \emph{Step 3: Accumulation of the rounding errors.}
  On the complementary event, the explicit solution of
  \cref{eq:rounded-gaussian-scalar-error-recursion} gives, for every
  $t\leq T_\varrho$,
  \begin{equation*}
    U_t
    \leq
    c_N\varrho
    \sum_{s=0}^t
    \prod_{j=s+1}^tM_j
    \leq
    C_N\varrho\,T_\varrho\left(\log(1/\varrho)\right)^{p_0}
    \leq
    \varrho\left(\log(1/\varrho)\right)^{p_0+C_N(1+C_0)}
  \end{equation*}
  for all sufficiently small $\varrho$. Set
  $p:=p_0+\lceil C_N(1+C_0)\rceil$. After increasing $C_N$, this
  exponent satisfies \cref{eq:rounded-gaussian-p-dependence}. Since
  $d_t\leq U_t$, we obtain \cref{eq:rounded-gaussian-coupling-input}. This
  completes the proof.
\end{proof}

Once $Y^{(N)}$ enters a ball of radius $K\varrho$, the number of additional
steps to absorption is $O(\log K)$, with an exponentially decaying upper tail.
The constants below may depend on $(A,N)$, but not on $\varrho$, $K$, or
$Y_0^{(N)}$.

\begin{lemma}
\label{lem:rounded-gaussian-final-absorption}
  Let $0<A<A_c(N)$. There are constants $c>0$ and $C<\infty$ such that,
  uniformly over $\varrho\in(0,1)$, $K\geq2$, and deterministic
  $Y_0^{(N)}=y\in\varrho\Z^N$ satisfying $\norm y_2\leq K\varrho$,
  \begin{equation}\label{eq:rounded-gaussian-final-absorption}
    \Pp_y\left(
      \tau_\varrho^{(N)}>C\log K+r
    \right)
    \leq
    C e^{-cr},
    \qquad r\geq0.
  \end{equation}
  Here $\Pp_y$ is the law of $Y^{(N)}$ started from $y$.
\end{lemma}

\begin{proof}
  \emph{Step 1: Entrance into a fixed grid-scale ball.}
  Let
  \begin{equation*}
    K_t:=\frac{\widehat R_t}{\varrho}.
  \end{equation*}
  Since $\norm{Q_\varrho(z)-z}_2\leq c_N\varrho$ and $\tanh$ is one-Lipschitz
  coordinatewise, we have
  \begin{equation*}
    \widehat R_{t+1}
    \leq
    \norm{\mathsf W_{t+1,A}^{(N)}Y_t^{(N)}}_2+c_N\varrho.
  \end{equation*}
  Let $\mathcal F_t:=\sigma(\mathsf W_{1,A}^{(N)},\ldots,\mathsf W_{t,A}^{(N)})$
  and fix a reference unit vector $e_1\in\R^N$. Set
  \begin{equation*}
    \theta_t
    :=
    \frac{Y_t^{(N)}}{\norm{Y_t^{(N)}}_2}\one_{\{Y_t^{(N)}\neq0\}}
    +e_1\one_{\{Y_t^{(N)}=0\}},
    \qquad
    M_{t+1}:=\norm{\mathsf W_{t+1,A}^{(N)}\theta_t}_2 .
  \end{equation*}
  The Gaussian calculation in the proof of
  \cref{lem:rounded-gaussian-synchronous-comparison} shows that the variables
  $M_j$ are independent copies of
  $\ell_{A,N}=(A/\sqrt N)\chi_N$. Moreover,
  $\norm{\mathsf W_{t+1,A}^{(N)}Y_t^{(N)}}_2=M_{t+1}\widehat R_t$ identically,
  with both sides zero when $Y_t^{(N)}=0$. Hence, pathwise,
  \begin{equation}\label{eq:rounded-gaussian-grid-radius-domination}
    K_{t+1}
    \leq
    M_{t+1}K_t+c_N .
  \end{equation}

  Define
  \begin{equation*}
    U_0=K,
    \qquad
    U_{t+1}=M_{t+1}U_t+c_N .
  \end{equation*}
  We have $\Ee\log M_1=-\gamma_{A,N}<0$. For every $\eps>0$,
  $\log(M_1+\eps)$ has exponential moments near the origin because $M_1$ has
  all positive moments by \cref{eq:rounded-gaussian-chi-log-moment}. Therefore,
  the hypotheses of \cref{lem:negative-drift-affine-entrance} hold with
  $b=c_N$.
  Since $K_0\leq K=U_0$,
  \cref{eq:rounded-gaussian-grid-radius-domination} and induction give
  $K_t\leq U_t$ for every $t\geq0$. We may thus apply
  \cref{lem:negative-drift-affine-entrance}. There are
  $K_\ast=K_\ast(A,N)<\infty$ and $c_1,c_2,c_3>0$ such that,
  with $\kappa:=\inf\{t\geq0:K_t\leq K_\ast\}$,
  \begin{equation}\label{eq:rounded-gaussian-grid-entrance}
    \Pp\left(\kappa>c_1\log K+r\right)
    \leq
    c_2e^{-c_3r},
    \qquad r\geq0.
  \end{equation}

  \emph{Step 2: Return after a failed absorption attempt.}
  Once $K_t\leq K_\ast$, $Y^{(N)}$ has a uniform positive probability of
  absorption in the next step. For a matrix $B$, write
  $\norm{B}_{\ell^2\to\ell^2}:=
  \sup_{\norm u_2=1}\norm{Bu}_2$.
  Since $\mathsf W_{t+1,A}^{(N)}$ has a nondegenerate Gaussian law, there is
  $p_\ast=p_\ast(A,N,K_\ast)>0$ such that
  \begin{equation*}
    \Pp\left(\norm{\mathsf W_{t+1,A}^{(N)}}_{\ell^2\to\ell^2}
      \leq \frac1{2K_\ast}\right)
    \geq p_\ast.
  \end{equation*}
  On this small operator norm event, if $\widehat R_t\leq K_\ast\varrho$, then
  \begin{equation*}
    \norm{\tanh(\mathsf W_{t+1,A}^{(N)}Y_t^{(N)})}_\infty
    \leq
    \norm{\mathsf W_{t+1,A}^{(N)}Y_t^{(N)}}_2
    \leq
    \frac{\varrho}{2},
  \end{equation*}
  and hence $Y_{t+1}^{(N)}=0$. Thus each visit to
  $\{K_t\leq K_\ast\}$ is followed by absorption with conditional probability
  at least $p_\ast$.

  We next control the return time after a failed attempt. Suppose that
  $K_0\leq K_\ast$ and define
  \begin{equation*}
    \kappa^+:=\inf\{t\geq1:K_t\leq K_\ast\}.
  \end{equation*}
  After one step, $Y^{(N)}$ satisfies
  $K_1\leq M_1K_\ast+c_N$, where $M_1$ has law
  $(A/\sqrt N)\chi_N$. Hence
  \cref{eq:rounded-gaussian-grid-entrance}, applied conditionally from time one
  with initial radius $\max\{2,K_1\}$, gives
  \begin{equation*}
    \Pp\left(
      \kappa^+>1+c_1\log(\max\{2,K_1\})+r
      \,\middle|\,\mathcal F_1
    \right)
    \leq
    c_2e^{-c_3r},
    \qquad r\geq0.
  \end{equation*}
  Choose $s_0\in(0,c_3)$. The layer-cake formula and
  $K_1\leq M_1K_\ast+c_N$ give
  \begin{equation*}
    B
    :=
    \sup_{\varrho\in(0,1)}
    \sup_{K_0\leq K_\ast}
    \Ee e^{s_0\kappa^+}
    \leq
    C\Ee(\max\{2,M_1K_\ast+c_N\})^{c_1s_0}
    <\infty,
  \end{equation*}
  since $\chi_N$ has moments of every positive order.

  The small-operator-norm event depends only on the next Gaussian matrix and
  gives, uniformly over all starting points in the fixed ball,
  \begin{equation*}
    \sup_{\varrho\in(0,1)}
    \sup_{K_0\leq K_\ast}
    \Pp(Y_1^{(N)}\neq0)
    \leq
    1-p_\ast.
  \end{equation*}
  For $s\in(0,s_0)$, H\"older's inequality consequently gives
  \begin{equation*}
    q_s
    :={}
    \sup_{\varrho\in(0,1)}
    \sup_{K_0\leq K_\ast}
    \Ee\left[e^{s\kappa^+}\one_{\{Y_1^{(N)}\neq0\}}\right]
    \leq
    B^{s/s_0}(1-p_\ast)^{1-s/s_0}.
  \end{equation*}
  The right-hand side converges to $1-p_\ast$ as $s\downarrow0$. We may
  therefore fix $s\in(0,s_0)$ sufficiently small that $q_s<1$.

  \emph{Step 3: Exponential tail of the absorption time.}
  For $n\geq1$, set
  \begin{equation*}
    F_{s,n}:=
    \sup_{\varrho\in(0,1)}
    \sup_{K_0\leq K_\ast}
    \Ee e^{s(\min\{\tau_\varrho^{(N)}, n\})}<\infty.
  \end{equation*}
  On $\{Y_1^{(N)}=0\}$ we have $\tau_\varrho^{(N)}=1$. On its complement,
  the strong Markov property at $\kappa^+$, together with
  $K_{\kappa^+}\leq K_\ast$, gives
  \begin{equation*}
    F_{s,n}
    \leq
    e^s+q_sF_{s,n}.
  \end{equation*}
  Since $q_s<1$, it follows that $F_{s,n}\leq e^s/(1-q_s)$ uniformly in $n$.
  Monotone convergence then shows that
  \begin{equation*}
    \sup_{\varrho\in(0,1)}
    \sup_{K_0\leq K_\ast}
    \Ee e^{s\tau_\varrho^{(N)}}
    \leq
    \frac{e^s}{1-q_s}<\infty.
  \end{equation*}
  Markov's inequality gives constants $C_0<\infty$ and $c_0>0$ such that,
  uniformly in
  $\varrho$ and deterministic starts with $K_0\leq K_\ast$,
  \begin{equation}\label{eq:rounded-gaussian-small-ball-absorption-tail}
    \Pp\left(\tau_\varrho^{(N)}>u\right)
    \leq
    C_0e^{-c_0u},
    \qquad u\geq0.
  \end{equation}

  Combining \cref{eq:rounded-gaussian-grid-entrance} with
  \cref{eq:rounded-gaussian-small-ball-absorption-tail} and the Markov property
  at $\kappa$ gives, for $r\geq0$,
  \begin{equation*}
    \Pp\left(\tau_\varrho^{(N)}>c_1\log K+r\right)
    \leq
    c_2e^{-c_3r/2}+C_0e^{-c_0r/2}
    \leq
    \left(c_2+C_0\right)e^{-\frac12\min\{c_3,c_0\}r}.
  \end{equation*}
  After increasing the prefactor and the coefficient of $\log K$, this is
  \cref{eq:rounded-gaussian-final-absorption}. This completes the proof.
\end{proof}

\subsection{Proof of the cutoff theorem}
\label{subsec:rounded-gaussian-cutoff-proof}

\begin{proof}[Proof of \cref{thm:rounded-gaussian-nearest-cutoff}]
  Fix $a\in\R$ and write $\ell_\varrho:=\log(1/\varrho)$ and
  $t_\varrho:=t_\varrho(a)$. By \cref{eq:tv-absorption}, it is enough to prove
  $\Pp(\tau_\varrho^{(N)}>t_\varrho)\to\Phi_{\mathrm G}(-a)$.

  \emph{Step 1: Coupling and threshold preparation.}
  Since $L_\varrho=\ell_\varrho+\log R_0$ and $a$ is fixed,
  \cref{eq:rounded-gaussian-cutoff-scale} gives
  $t_\varrho=O(\ell_\varrho)$. Choose $C_0<\infty$ such that, for all
  sufficiently small $\varrho$, $0\leq t_\varrho\leq
  T_\varrho:=\lfloor C_0\ell_\varrho\rfloor$. Let
  $p=p(A,N,C_0)$ be the exponent from
  \cref{lem:rounded-gaussian-synchronous-comparison}, put
  $d_\varrho:=\varrho\ell_\varrho^p$ and
  $M_\varrho:=\ell_\varrho^{p+3}$. For the synchronous coupling, define
  \begin{equation}\label{eq:rounded-gaussian-cutoff-good-coupling}
    \mathcal G_\varrho
    :=
    \left\{
      \max_{0\leq s\leq T_\varrho}
      \norm{Y_s^{(N)}-X_s^{(N)}}_2
      \leq d_\varrho
    \right\},
    \qquad
    \Pp(\mathcal G_\varrho^{\mathrm c})=o(1),
  \end{equation}
  where the estimate follows from
  \cref{eq:rounded-gaussian-coupling-input}.

  We compare absorption with the unrounded entrance times
  $\Theta_\varrho^-:=\theta_{d_\varrho}^{\mathrm{rad}}$ and
  $\Theta_\varrho^+:=\theta_{M_\varrho\varrho}^{\mathrm{rad}}$.
  Since both $\log(R_0/d_\varrho)$ and
  $\log(R_0/(M_\varrho\varrho))$ satisfy the hypotheses of
  \cref{prop:rounded-gaussian-threshold-cutoff},
  \begin{equation}\label{eq:rounded-gaussian-cutoff-thresholds}
    \Pp(\Theta_\varrho^->t_\varrho)
    \longrightarrow\Phi_{\mathrm G}(-a),
    \qquad
    \Pp(\Theta_\varrho^+>t_\varrho)
    \longrightarrow\Phi_{\mathrm G}(-a)
  \end{equation}
  as $\varrho\downarrow0$.

  \emph{Step 2: Lower bound from the inner threshold.}
  On $\{\tau_\varrho^{(N)}\leq t_\varrho\}\cap\mathcal G_\varrho$,
  absorption gives $Y_{\tau_\varrho^{(N)}}^{(N)}=0$, while
  $\tau_\varrho^{(N)}\leq t_\varrho\leq T_\varrho$. Hence
  $R_{\tau_\varrho^{(N)}}=
  \norm{X_{\tau_\varrho^{(N)}}^{(N)}-
  Y_{\tau_\varrho^{(N)}}^{(N)}}_2\leq d_\varrho$ and thus
  $\Theta_\varrho^-\leq\tau_\varrho^{(N)}$. Consequently,
  \begin{equation}\label{eq:rounded-gaussian-cutoff-lower-bound}
    \{\Theta_\varrho^->t_\varrho\}\cap\mathcal G_\varrho
    \subseteq\{\tau_\varrho^{(N)}>t_\varrho\},
    \qquad
    \Pp(\tau_\varrho^{(N)}>t_\varrho)
    \geq\Pp(\Theta_\varrho^->t_\varrho)-o(1),
  \end{equation}
  where we used
  \cref{eq:rounded-gaussian-cutoff-good-coupling} in the last estimate.

  \emph{Step 3: Upper bound from the outer threshold and final absorption.}
  Let $c_{\mathrm{abs}}>0$ and $C_{\mathrm{abs}}<\infty$ be constants for
  which \cref{eq:rounded-gaussian-final-absorption} holds. Choose
  $B>C_{\mathrm{abs}}(p+3)$ and set
  $h_\varrho:=\lceil B\log\ell_\varrho\rceil$,
  $K_\varrho:=2M_\varrho$, and
  $r_\varrho:=h_\varrho-C_{\mathrm{abs}}\log K_\varrho$.
  Since $\log K_\varrho=\log2+(p+3)\log\ell_\varrho$, we have
  $r_\varrho\to\infty$, while $h_\varrho=o(\sqrt{L_\varrho})$ and
  $t_\varrho-h_\varrho\geq0$ for all sufficiently small $\varrho$.

  Write $\mathcal E_\varrho:=
  \{\Theta_\varrho^+\leq t_\varrho-h_\varrho\}\cap
  \mathcal G_\varrho$. On $\mathcal E_\varrho$, we have
  $\Theta_\varrho^+\leq T_\varrho$, and therefore
  \begin{equation*}
    \widehat R_{\Theta_\varrho^+}
    =\norm{Y_{\Theta_\varrho^+}^{(N)}}_2
    \leq
    \norm{X_{\Theta_\varrho^+}^{(N)}}_2
    +\norm{Y_{\Theta_\varrho^+}^{(N)}-X_{\Theta_\varrho^+}^{(N)}}_2
    \leq M_\varrho\varrho+d_\varrho
    \leq K_\varrho\varrho
  \end{equation*}
  for all sufficiently small $\varrho$. Define the
  $\mathcal F_{\Theta_\varrho^+}$-measurable event
  $\mathcal A_\varrho:=\{\Theta_\varrho^+\leq
  t_\varrho-h_\varrho,\ \widehat R_{\Theta_\varrho^+}\leq
  K_\varrho\varrho\}$. The preceding bound gives
  $\mathcal E_\varrho\subseteq\mathcal A_\varrho$. Therefore
  \begin{align*}
    \Pp(\tau_\varrho^{(N)}>t_\varrho,\mathcal E_\varrho)
    &\leq
    \Pp(\tau_\varrho^{(N)}>t_\varrho,\mathcal A_\varrho)
    \\
    &=
    \Ee\left[
      \one_{\mathcal A_\varrho}
      \Pp_{Y_{\Theta_\varrho^+}^{(N)}}\left(
        \tau_\varrho^{(N)}>t_\varrho-\Theta_\varrho^+
      \right)
    \right]
    \\
    &\leq
    \Ee\left[
      \one_{\mathcal A_\varrho}
      \Pp_{Y_{\Theta_\varrho^+}^{(N)}}\left(
        \tau_\varrho^{(N)}>h_\varrho
      \right)
    \right]
    \\
    &\leq
    C_{\mathrm{abs}}e^{-c_{\mathrm{abs}}r_\varrho}
    =o(1).
  \end{align*}
  Here $\Pp_y$ denotes the law of $Y^{(N)}$ started from $y$. The
  equality is the strong Markov property at $\Theta_\varrho^+$. The next
  inequality uses $t_\varrho-\Theta_\varrho^+\geq h_\varrho$ on
  $\mathcal A_\varrho$. The last inequality follows from
  \cref{lem:rounded-gaussian-final-absorption}, since the entrance radius is
  at most $K_\varrho\varrho$ and
  $h_\varrho=C_{\mathrm{abs}}\log K_\varrho+r_\varrho$.

  Hence
  \begin{align}
    \{\tau_\varrho^{(N)}>t_\varrho\}
    &\subseteq
    \{\Theta_\varrho^+>t_\varrho-h_\varrho\}
    \cup\mathcal G_\varrho^{\mathrm c}
    \cup\bigl(\{\tau_\varrho^{(N)}>t_\varrho\}
      \cap\mathcal E_\varrho\bigr),
    \notag\\
    \Pp(\tau_\varrho^{(N)}>t_\varrho)
    &\leq
    \Pp(\Theta_\varrho^+>t_\varrho-h_\varrho)
    +\Pp(\mathcal G_\varrho^{\mathrm c})
    +\Pp(\tau_\varrho^{(N)}>t_\varrho,\mathcal E_\varrho)
    \notag\\
    &\leq
    \Pp(\Theta_\varrho^+>t_\varrho-h_\varrho)+o(1).
    \label{eq:rounded-gaussian-cutoff-upper-bound}
  \end{align}

  \emph{Step 4: Identification of the profile and cutoff window.}
  By the level calculations in Step~1, the conclusion of
  \cref{eq:rounded-gaussian-cutoff-thresholds} also holds for
  $\Theta_\varrho^+$ at $t_\varrho-h_\varrho$, because
  $h_\varrho=o(\sqrt{L_\varrho})$.
  Combining these limits with
  \cref{eq:rounded-gaussian-cutoff-lower-bound,eq:rounded-gaussian-cutoff-upper-bound}
  yields
  \begin{equation*}
    \Phi_{\mathrm G}(-a)
    \leq
    \liminf_{\varrho\downarrow0}
    \Pp(\tau_\varrho^{(N)}>t_\varrho)
    \leq
    \limsup_{\varrho\downarrow0}
    \Pp(\tau_\varrho^{(N)}>t_\varrho)
    \leq\Phi_{\mathrm G}(-a).
  \end{equation*}
  Thus the absorption tail converges to $\Phi_{\mathrm G}(-a)$, and
  \cref{eq:tv-absorption} gives
  \cref{eq:rounded-gaussian-nearest-profile}.

  Since $A$ and $N$ are fixed,
  $\sigma_N\gamma_{A,N}^{-3/2}\sqrt{L_\varrho}
  =o(L_\varrho/\gamma_{A,N})$. The profile just
  proved satisfies $\Phi_{\mathrm G}(-a)\to1$ as $a\to-\infty$ and
  $\Phi_{\mathrm G}(-a)\to0$ as $a\to\infty$. Therefore
  \cref{eq:intro-cutoff-convention} holds with center
  $L_\varrho/\gamma_{A,N}$ and window
  $\sigma_N\gamma_{A,N}^{-3/2}\sqrt{L_\varrho}$.
\end{proof}

\begin{remark}
\label{rem:rounded-gaussian-critical-dependence}
  As $A\uparrow A_c(N)$,
  \begin{equation*}
    \gamma_{A,N}
    =
    \log\frac{A_c(N)}A
    =
    \frac{A_c(N)-A}{A_c(N)}+O_N((A_c(N)-A)^2).
  \end{equation*}
  Hence the cutoff center and window satisfy
  \begin{equation*}
    \frac{L_\varrho}{\gamma_{A,N}}
    \sim
    \frac{A_c(N)L_\varrho}{A_c(N)-A},
    \qquad
    \frac{\sigma_N}{\gamma_{A,N}^{3/2}}\sqrt{L_\varrho}
    \sim
    \frac{\sigma_N A_c(N)^{3/2}\sqrt{L_\varrho}}
    {(A_c(N)-A)^{3/2}}.
  \end{equation*}
  For fixed $C_0$, \cref{eq:rounded-gaussian-p-dependence} allows the exponent
  in \cref{lem:rounded-gaussian-synchronous-comparison} to be chosen with
  \begin{equation*}
    p=O_{N,C_0}((A_c(N)-A)^{-1})
    \qquad\mbox{as }A\uparrow A_c(N).
  \end{equation*}
  The constants in \cref{lem:rounded-gaussian-final-absorption} may also
  diverge as $A\uparrow A_c(N)$. For fixed $A<A_c(N)$, $p$ appears only in the
  bound $\varrho(\log(1/\varrho))^p$ for
  $\max_{0\leq t\leq T_\varrho}\norm{Y_t^{(N)}-X_t^{(N)}}_2$ in
  \cref{eq:rounded-gaussian-coupling-input}. The absorption constants determine
  only the $O(\log\log(1/\varrho))$ additional time allowed after the chains
  are close to the origin. Neither changes the cutoff center or window.
\end{remark}

\section{Subcritical cutoff as \texorpdfstring{$N\to\infty$}{N tends to infinity}}
\label{sec:subcritical-dimension}

Fix $\varrho\in(0,1)$ and let $N\to\infty$. We track $\mathscr H^{(N)}$ along
the iterates of $V_{A,\varrho}$ until they reach $(\log N)^{-1}$. The next
step makes $\mathscr H^{(N)}$ polynomially small with high probability, and
one further step gives absorption. The same tracking estimate excludes
absorption one step before the deterministic entrance time.

\subsection{Deterministic rounded dynamics}
\label{subsec:subcritical-deterministic-rounded-dynamics}

Recall the transition of $\mathscr H^{(N)}$ and its centered noise decomposition
from \cref{eq:subcritical-exact-grid-transition,eq:rounded-noise-decomposition}.
The deterministic orbit and terminal-scale entrance time are
\begin{equation}\label{eq:subcritical-exact-lattice-orbit}
  \bar h_{0,N}
  :=
  \frac1{N\varrho^2}\norm{Q_\varrho x_N}_2^2,
  \qquad
  \bar h_{t+1,N}
  =
  V_{A,\varrho}(\bar h_{t,N}),
\end{equation}
and
\begin{equation}\label{eq:subcritical-exact-hitting-time}
  \bar t_N
  :=
  \inf\left\{
    t\in\N_0:\bar h_{t,N}\leq\frac1{\log N}
  \right\}.
\end{equation}
The tie convention in \cref{sec:setup} gives
\begin{equation}\label{eq:subcritical-rounding-monotonicity}
  \abs{Q_1(u)}\leq \abs{Q_1(v)}
  \quad\mbox{whenever }\abs u\leq\abs v .
\end{equation}
The threshold $A_{\mathrm{lat}}$ is defined by
\cref{eq:subcritical-lattice-threshold}. Equivalently,
$A<A_{\mathrm{lat}}$ means that $L_A(h)<h$ for every $h>0$.

\begin{lemma}
\label{lem:subcritical-lattice-regimes}
  The threshold in \cref{eq:subcritical-lattice-threshold} belongs to
  $(0,1)$. If $A<A_{\mathrm{lat}}$, then there exists
  $\theta_A\in(0,1)$ such that
  \begin{equation}\label{eq:subcritical-uniform-lattice-contraction}
    L_A(h)\leq \theta_A h
    \qquad\mbox{for every }h>0 .
  \end{equation}
  Consequently,
  \begin{equation}\label{eq:subcritical-exact-lattice-map-contraction}
    V_{A,\varrho}(h)\leq \theta_A h,
    \qquad h\geq0 .
  \end{equation}
\end{lemma}

\begin{proof}
  \emph{Step 1: Endpoint behavior of $m$.}
  For every $\alpha>0$, we have $m(\alpha)>0$, since
  $\Pp(\abs{\alpha G}>1/2)>0$. The map $\alpha\mapsto m(\alpha)$ is
  continuous on $(0,\infty)$. Indeed, if $0<\alpha\leq R$, then
  \begin{equation*}
    Q_1(\alpha G)^2
    \leq
    (\abs{\alpha G}+\tfrac12)^2
    \leq
    (R\abs G+\tfrac12)^2 .
  \end{equation*}
  Thus continuity follows by dominated convergence. Near the origin,
  $Q_1(u)=0$ on $\{\abs u\leq1/2\}$ and
  $\abs{Q_1(u)}\leq \abs u+1/2\leq2\abs u$ on
  $\{\abs u>1/2\}$,
  \begin{equation}\label{eq:subcritical-lattice-small-alpha}
    0
    \leq
    \frac{m(\alpha)}{\alpha^2}
    \leq
    4\Ee\left[
      G^2\one_{\{\abs G>1/(2\alpha)\}}
    \right]
    \xrightarrow[\alpha\downarrow0]{}0 .
  \end{equation}
  At infinity, the bound $\abs{Q_1(u)-u}\leq1/2$ gives
  $Q_1(\alpha G)/\alpha\to G$ pointwise and
  \begin{equation*}
    \left(\frac{\abs{Q_1(\alpha G)}}{\alpha}\right)^2
    \leq
    \left(\abs G+\frac1{2\alpha}\right)^2
    \leq
    (\abs G+1)^2
    \qquad\mbox{for }\alpha\geq\frac12 .
  \end{equation*}
  Dominated convergence therefore gives
  \begin{equation}\label{eq:subcritical-lattice-large-alpha}
    \frac{m(\alpha)}{\alpha^2}
    =
    \Ee\left[
      \left(\frac{Q_1(\alpha G)}{\alpha}\right)^2
    \right]
    \xrightarrow[\alpha\to\infty]{}
    \Ee[G^2]
    =
    1 .
  \end{equation}

  \emph{Step 2: The threshold and uniform contraction.}
  It follows from \cref{eq:subcritical-lattice-small-alpha} that
  $\alpha^2/m(\alpha)\to\infty$ as $\alpha\downarrow0$, while
  \cref{eq:subcritical-lattice-large-alpha} gives
  $\alpha^2/m(\alpha)\to1$ as $\alpha\to\infty$. Since
  $\alpha^2/m(\alpha)$ is continuous and positive, it is bounded below away
  from zero outside a compact subinterval of $(0,\infty)$ and attains a
  positive minimum on that subinterval. Hence
  $A_{\mathrm{lat}}\in(0,1]$. To prove the strict upper bound, take
  $\alpha=1/2$. Then
  \begin{equation*}
    Q_1(\alpha G)^2\geq1
    \quad\mbox{on }\{\abs G>1\},
  \end{equation*}
  and hence
  \begin{equation*}
    m(1/2)\geq\Pp(\abs G>1)>\frac14.
  \end{equation*}
  The final strict inequality can be checked by an elementary Gaussian tail
  bound. Therefore,
  $A_{\mathrm{lat}}^2\leq (1/2)^2/m(1/2)<1$.

  By the change of variables $\alpha=A\sqrt h$,
  \begin{equation*}
    \frac{L_A(h)}h
    =
    A^2\,\frac{m(A\sqrt h)}{(A\sqrt h)^2},
    \qquad h>0 .
  \end{equation*}
  Thus $A<A_{\mathrm{lat}}$ implies $L_A(h)<h$ for every $h>0$.
  Moreover,
  \cref{eq:subcritical-lattice-small-alpha,eq:subcritical-lattice-large-alpha}
  give
  \begin{equation*}
    \frac{L_A(h)}h\to0
    \quad\mbox{as }h\downarrow0,
    \qquad
    \frac{L_A(h)}h\to A^2<1
    \quad\mbox{as }h\to\infty .
  \end{equation*}
  Since $h\mapsto L_A(h)/h$ is continuous on $(0,\infty)$, choose
  $0<r<R<\infty$ so that the ratio is bounded away from one on
  $(0,r]\cup[R,\infty)$. On $[r,R]$, the strict pointwise inequality and
  compactness give the same bound. This proves
  \cref{eq:subcritical-uniform-lattice-contraction}.

  \emph{Step 3: Comparison with the exact rounded map.}
  The inequality $\abs{\tanh z}\leq\abs z$ and
  \cref{eq:subcritical-rounding-monotonicity} give, pointwise in $G$,
  \begin{equation*}
    \left|
      Q_1\left(
        \varrho^{-1}\tanh(\varrho A\sqrt h\,G)
      \right)
    \right|
    \leq
    \abs{Q_1(A\sqrt h\,G)} .
  \end{equation*}
  Taking squares and expectations yields
  $V_{A,\varrho}(h)\leq L_A(h)$ for $h>0$. The case $h=0$ is immediate,
  since both sides vanish. Combining this with
  \cref{eq:subcritical-uniform-lattice-contraction} proves
  \cref{eq:subcritical-exact-lattice-map-contraction}.
\end{proof}

Let
$\bar\Phi(x):=\Pp(G>x)$ and
$\phi(x):=(2\pi)^{-1/2}\exp(-x^2/2)$. For $0<\varrho<1$ set
\begin{equation}\label{eq:subcritical-layer-thresholds}
  b_{\varrho,k}
  :=
  \varrho^{-1}\arctanh\left(\varrho(k+\tfrac12)\right),
  \qquad
  \mathcal K_\varrho
  :=
  \{k\geq0:\varrho(k+\tfrac12)<1\}.
\end{equation}
Then
\begin{equation}\label{eq:subcritical-layer-formulas-exact}
  V_{A,\varrho}(h)=2\sum_{k\in\mathcal K_\varrho}(2k+1)
  \bar\Phi\left(\frac{b_{\varrho,k}}{A\sqrt h}\right),
  \qquad
  V_{A,\varrho}'(h)=\sum_{k\in\mathcal K_\varrho}(2k+1)
  \frac{b_{\varrho,k}}{Ah^{3/2}}
  \phi\left(\frac{b_{\varrho,k}}{A\sqrt h}\right).
\end{equation}
Indeed, $Q_1(u)^2=\sum_{k\geq0}(2k+1)
\one_{\{\abs u>k+\frac12\}}$ outside Gaussian-null boundary events, and the
preimage of the layer $k+\frac12$ under
$u\mapsto\varrho^{-1}\tanh(\varrho u)$ is $b_{\varrho,k}$. Since
$\mathcal K_\varrho$ is finite, differentiating term by term is immediate.
The small-radius estimate in
\cref{eq:subcritical-fixed-small-radius-map-derivative} also gives
$V_{A,\varrho}'(h)\to0$ as $h\downarrow0$. We use this one-sided extension at
the origin.

By \cref{eq:subcritical-layer-formulas-exact}, $V_{A,\varrho}$ is positive and
$V_{A,\varrho}'$ is bounded on compact subsets of $(0,\infty)$. The threshold
$b_{\varrho,0}>0$ in \cref{eq:subcritical-layer-thresholds} yields exponential
decay in $1/h$ near the origin.

We will use regularity of $V_{A,\varrho}$ both away from and near
the origin.
\begin{lemma}
\label{lem:subcritical-fixed-precision-map-estimates}
  Fix $\varrho\in(0,1)$. For every $0<r<R<\infty$ there are constants
  $m,L\in(0,\infty)$, depending on $A,\varrho,r,R$, such that
  \begin{equation}\label{eq:subcritical-fixed-compact-bounds}
    m\leq V_{A,\varrho}(h),
    \qquad
    V_{A,\varrho}'(h)\leq L,
    \qquad r\leq h\leq R .
  \end{equation}
  For every $r\in(0,1)$ there are constants $c,C\in(0,\infty)$, depending on
  $A,\varrho,r$, such that, for $0<h\leq r$,
  \begin{equation}\label{eq:subcritical-fixed-small-radius-map-derivative}
    V_{A,\varrho}(h)
    \leq
    C\exp\left(-\frac c h\right),
    \qquad
    V_{A,\varrho}'(h)
    \leq
    Ch^{-2}\exp\left(-\frac c h\right).
  \end{equation}
\end{lemma}

\begin{proof}
  Since
  $0<\varrho<1$, the set $\mathcal K_\varrho$ is finite and nonempty. In
  particular $0\in\mathcal K_\varrho$ and
  \begin{equation*}
    b_{\varrho,0}
    =
    \varrho^{-1}\arctanh(\varrho/2)
    >0 .
  \end{equation*}
  The representation
  \cref{eq:subcritical-layer-formulas-exact} is a finite sum of
  continuous functions of $h>0$. Hence $V_{A,\varrho}$ is continuous on
  $(0,\infty)$. Moreover, for every $h>0$,
  \begin{equation*}
    V_{A,\varrho}(h)
    \geq
    2\bar\Phi\left(\frac{b_{\varrho,0}}{A\sqrt h}\right)
    >
    0 .
  \end{equation*}
  It follows that
  \begin{equation*}
    m
    :=
    \min_{h\in[r,R]} V_{A,\varrho}(h)
    >
    0 .
  \end{equation*}
  Similarly, \cref{eq:subcritical-layer-formulas-exact} is a finite sum of
  continuous functions on $[r,R]$. Thus
  \begin{equation*}
    L
    :=
    \max_{h\in[r,R]} V_{A,\varrho}'(h)
    <
    \infty ,
  \end{equation*}
  which proves \cref{eq:subcritical-fixed-compact-bounds}.

  For the small-radius bounds, $C<\infty$ may
  change from line to line and depends only on $A,\varrho,r$. Put
  $b_0:=b_{\varrho,0}$ and
  \begin{equation*}
    c:=\frac{b_0^2}{2A^2}.
  \end{equation*}
  Since $k\mapsto b_{\varrho,k}$ is increasing on $\mathcal K_\varrho$, we have
  $b_{\varrho,k}\geq b_0$ for every $k\in\mathcal K_\varrho$. Using
  $\bar\Phi(x)\leq \exp(-x^2/2)$ for $x\geq0$ in
  \cref{eq:subcritical-layer-formulas-exact}, we obtain, for $h>0$,
  \begin{equation*}
    V_{A,\varrho}(h)
    \leq
    2\sum_{k\in\mathcal K_\varrho}(2k+1)
    \exp\left(-\frac{b_{\varrho,k}^2}{2A^2h}\right)
    \leq
    C\exp\left(-\frac c h\right).
  \end{equation*}
  For the derivative, the same lower bound on the thresholds and
  \cref{eq:subcritical-layer-formulas-exact} give
  \begin{equation*}
    V_{A,\varrho}'(h)
    \leq
    \frac1{\sqrt{2\pi}}
    \sum_{k\in\mathcal K_\varrho}
    (2k+1)\frac{b_{\varrho,k}}{Ah^{3/2}}
    \exp\left(-\frac{b_{\varrho,k}^2}{2A^2h}\right)
    \leq
    C h^{-3/2}\exp\left(-\frac c h\right).
  \end{equation*}
  If $0<h\leq r<1$, then $h^{-3/2}\leq h^{-2}$, and the desired derivative
  estimate follows after increasing $C$. This proves
  \cref{eq:subcritical-fixed-small-radius-map-derivative}.
\end{proof}

\subsection{Tracking \texorpdfstring{$\mathscr H^{(N)}$}{H(N)}}
\label{subsec:subcritical-radius-tracking}

Assume $A<A_{\mathrm{lat}}$, put $a_N=(\log N)^{-1}$, and consider an orbit
$h_{t+1,N}=V_{A,\varrho}(h_{t,N})$ with $h_{0,N}\leq C_0$. For
$\delta\in(0,1/4)$, set
\begin{equation*}
  L_{t,N}
  :=
  \sup_{\substack{u\geq0\\
    \abs{u-h_{t,N}}\leq\delta(h_{t,N}+a_N)}}
  V_{A,\varrho}'(u),
  \qquad
  \beta_{t,N}:=\frac{h_{t,N}+a_N}{h_{t+1,N}+a_N},
  \qquad
  \alpha_{t,N}:=L_{t,N}\beta_{t,N}.
\end{equation*}

The tracking argument below normalizes
$\mathscr H_t^{(N)}-\bar h_{t,N}$ by $\bar h_{t,N}+a_N$. When
$h_{t,N}=\bar h_{t,N}$, the factors $\alpha_{t,N}$ and $\beta_{t,N}$ control
the propagation of this ratio. The following estimate bounds the resulting
products.
\begin{lemma}
\label{lem:subcritical-normalized-lattice-amplification}
  There is a constant $C<\infty$, depending only on
  $A,\varrho,\delta,C_0$, such that, for all $T_N\in\N_0$ and all
  sufficiently large $N$,
  \begin{equation}\label{eq:subcritical-global-amplification}
    \prod_{u=s}^{t-1}(1\vee\alpha_{u,N})
    +
    \beta_{s,N}
    \prod_{u=s+1}^{t-1}(1\vee\alpha_{u,N})
    \leq
    \exp\{C(\log\log N)^2\}
  \end{equation}
  for all $0\leq s<t\leq T_N$.
\end{lemma}

\begin{proof}
  Fix $r\in(0,1/2)$ and set
  \begin{equation*}
    \tau_r:=\inf\{t:h_{t,N}\leq r\},
    \qquad
    \tau_a:=\inf\{t:h_{t,N}\leq a_N\}.
  \end{equation*}
  Let $I_{t,N}$ denote the interval over which the supremum defining
  $L_{t,N}$ is taken.

  \emph{Step 1: The region $h_{t,N}>r$.}
  By \cref{eq:subcritical-exact-lattice-map-contraction},
  $h_{t,N}\leq\theta_A^tC_0$, and hence $\tau_r\leq C$. For $t<\tau_r$, we
  have $h_{t,N}\in[r,C_0]$. Since $a_N\leq r$ for all sufficiently large $N$
  and $\delta<1/4$, it follows that
  $I_{t,N}\subseteq[r/2,2(C_0+r)]$. The bounds in
  \cref{eq:subcritical-fixed-compact-bounds} therefore give
  \begin{equation*}
    h_{t+1,N}=V_{A,\varrho}(h_{t,N})\geq m,
    \qquad
    L_{t,N}\leq L .
  \end{equation*}
  Thus $\alpha_{t,N}$ and $\beta_{t,N}$ are bounded uniformly in $N$ for
  $t<\tau_r$.

  \emph{Step 2: The region $a_N<h_{t,N}\leq r$.}
  The contraction bound also gives
  $\tau_a\leq C\log(1/a_N)=C\log\log N$. For $\tau_r\leq t<\tau_a$ and
  $u\in I_{t,N}$,
  \begin{equation*}
    (1-2\delta)h_{t,N}
    \leq
    u
    \leq
    (1+2\delta)h_{t,N},
  \end{equation*}
  because $a_N<h_{t,N}$. In particular, $I_{t,N}\subseteq[0,2r]$. By
  \cref{eq:subcritical-fixed-small-radius-map-derivative} and the boundedness
  of $u^{-2}\exp(-c/u)$ on $(0,2r]$, we have
  $L_{t,N}\leq C$. Moreover,
  \begin{equation*}
    \beta_{t,N}
    =
    \frac{h_{t,N}+a_N}{h_{t+1,N}+a_N}
    \leq
    \frac{2r}{a_N}
    \leq
    C\log N .
  \end{equation*}
  Hence $\alpha_{t,N}\leq C\log N$ for $\tau_r\leq t<\tau_a$.

  \emph{Step 3: The region $h_{t,N}\leq a_N$.}
  For $t\geq\tau_a$, the contraction gives $h_{t,N}\leq a_N$, and
  $I_{t,N}\subseteq[0,2a_N]$. Since
  $u\mapsto u^{-2}\exp(-c/u)$ is increasing in a sufficiently small
  right-neighborhood of the origin, the derivative bound in
  \cref{eq:subcritical-fixed-small-radius-map-derivative} gives
  \begin{equation*}
    L_{t,N}
    \leq
    Ca_N^{-2}\exp\left(-\frac{c}{2a_N}\right)
    =
    C(\log N)^2N^{-c/2}
    \leq
    N^{-c/3}
  \end{equation*}
  for all sufficiently large $N$, after decreasing $c>0$ if necessary.
  Moreover, $h_{t+1,N}+a_N\geq a_N$, and therefore
  \begin{equation*}
    \beta_{t,N}
    =
    \frac{h_{t,N}+a_N}{h_{t+1,N}+a_N}
    \leq2 .
  \end{equation*}
  Consequently, after decreasing $c$ once more,
  $\alpha_{t,N}=L_{t,N}\beta_{t,N}\leq N^{-c}$.

  Thus $1\vee\alpha_{t,N}$ is uniformly bounded for at most $C$ times,
  bounded by $C\log N$ for at most $C\log\log N$ times, and equal to $1$ for
  $t\geq\tau_a$. It follows that
  \begin{equation*}
    \prod_{u=s}^{t-1}(1\vee\alpha_{u,N})
    \leq
    \exp\{C(\log\log N)^2\}.
  \end{equation*}
  The preceding three steps also give $\beta_{s,N}\leq C\log N$ for every
  $s$. This factor is absorbed into the same bound, which proves
  \cref{eq:subcritical-global-amplification}.
\end{proof}

Conditional independence and \cref{eq:subcritical-global-amplification} show
that $\mathscr H^{(N)}$ tracks $(\bar h_{t,N})_{t\geq0}$ in relative error.
The estimate below is the discrete-time counterpart of the stopped $L^2$
fluid-limit estimate of Darling and Norris \cite{darling-norris2008}. Compare
also with the bounded-increment formulation of Warnke \cite{warnke2019}.
\begin{proposition}
\label{prop:subcritical-exact-radius-concentration}
  Assume $A<A_{\mathrm{lat}}$, fix $\varrho\in(0,1)$, and let
  $x_N\in[-1,1]^N$. Let $\bar T_N\in\N_0$ be a sequence such that
  \begin{equation*}
    \frac{\bar T_N}{N}\exp\{C(\log\log N)^2\}\to0
    \quad\mbox{for every fixed }C<\infty .
  \end{equation*}
  Then
  \begin{equation}\label{eq:subcritical-exact-radius-concentration}
    \sup_{0\leq t\leq \bar T_N}
    \frac{
      \abs{\mathscr H_t^{(N)}-\bar h_{t,N}}
    }{
      \bar h_{t,N}+(\log N)^{-1}
    }
    \xrightarrow{\Pp}0 .
  \end{equation}
\end{proposition}

\begin{proof}
  Put $a_N=(\log N)^{-1}$, and let $U_{A,\varrho}$ be the coordinate
  observable in \cref{eq:rounded-coordinate-observable}.
  If $\bar T_N=0$, then the assertion is immediate because
  $\mathscr H_0^{(N)}=\bar h_{0,N}$. We therefore assume $\bar T_N\geq1$.

  \emph{Step 1: Bounding $\bar h_{0,N}$ and the moments of
  $U_{A,\varrho}$.}
  Since $x_N\in[-1,1]^N$ and $\abs{Q_1(u)-u}\leq1/2$, we have
  $\abs{Q_\varrho(x_N)_i}\leq1+\varrho/2$ for every coordinate. Hence,
  \begin{equation}\label{eq:subcritical-initial-radius-bound}
    \bar h_{0,N}
    \leq
    C_\varrho
    :=
    \frac{(1+\varrho/2)^2}{\varrho^2}.
  \end{equation}
  Furthermore, $U_{A,\varrho}(h,G)$ is uniformly bounded in $h$, because
  \begin{equation*}
    \varrho^{-1}\tanh(\varrho A\sqrt h\,G)
    \in[-\varrho^{-1},\varrho^{-1}].
  \end{equation*}
  For $0<h\leq1$, the definition of $b_{\varrho,0}$ in
  \cref{eq:subcritical-layer-thresholds} gives
  \begin{equation*}
    \{U_{A,\varrho}(h,G)\neq0\}
    \subseteq
    \{A\sqrt h\,\abs G>b_{\varrho,0}\}.
  \end{equation*}
  The probability of this event is at most $C\exp(-c/h)$. Therefore, for
  every integer $p\geq1$,
  \begin{equation}\label{eq:subcritical-fixed-moment-bound}
    \Ee U_{A,\varrho}(h,G)^p
    \leq
    C_p h^p,
    \qquad h\geq0 .
  \end{equation}
  Indeed, for $0<h\leq1$ we use $e^{-c/h}\leq C_ph^p$, for $h=0$ the
  left-hand side vanishes, and for $h\geq1$ the bound follows from the
  uniform boundedness of $U_{A,\varrho}$ and $h^p\geq1$.

  \emph{Step 2: The conditional variance of
  $\widehat\zeta_{t+1}^{(N)}$ on $\{t<\tau_\delta\}$.}
  Fix $\delta\in(0,1/4)$ and set
  \begin{equation*}
    e_t:=\mathscr H_t^{(N)}-\bar h_{t,N},
    \qquad
    \mathcal R_{t,N}:=\frac{e_t}{\bar h_{t,N}+a_N},
    \qquad
    \tau_\delta:=\inf\{t\geq0:\abs{\mathcal R_{t,N}}>\delta\}.
  \end{equation*}
  Let $\mathcal F_t$ be the natural filtration. By
  \cref{eq:rounded-noise-decomposition,eq:rounded-noise-moments},
  $\Ee[\widehat\zeta_{t+1}^{(N)}\mid\mathcal F_t]=0$. Since
  $\mathscr H_{t+1}^{(N)}$ is an average of $N$ conditionally independent
  copies of $U_{A,\varrho}(\mathscr H_t^{(N)},G)$, the conditional variance is
  bounded by
  \begin{equation*}
    \Ee[(\widehat\zeta_{t+1}^{(N)})^2\mid\mathcal F_t]
    \leq
    \frac1N
    \Ee U_{A,\varrho}(\mathscr H_t^{(N)},G)^2 .
  \end{equation*}
  On $\{t<\tau_\delta\}$, we have
  \begin{equation*}
    \mathscr H_t^{(N)}
    \leq
    (1+\delta)\bar h_{t,N}+\delta a_N
    \leq
    C(\bar h_{t,N}+a_N).
  \end{equation*}
  Thus \cref{eq:subcritical-fixed-moment-bound} with $p=2$ gives, on this
  event,
  \begin{equation}\label{eq:subcritical-fixed-noise-variance}
    \Ee[(\widehat\zeta_{t+1}^{(N)})^2\mid\mathcal F_t]
    \leq
    \frac{C(\bar h_{t,N}+a_N)^2}{N}.
  \end{equation}

  \emph{Step 3: The error recursion.}
  Define the $\mathcal F_t$-measurable difference quotient
  \begin{equation*}
    D_{t,N}
    :=
    \begin{cases}
      \displaystyle
      \frac{
        V_{A,\varrho}(\mathscr H_t^{(N)})
        -V_{A,\varrho}(\bar h_{t,N})
      }{
        \mathscr H_t^{(N)}-\bar h_{t,N}
      },
      & \mathscr H_t^{(N)}\neq\bar h_{t,N},
      \\[1.2em]
      V_{A,\varrho}'(\bar h_{t,N}),
      & \mathscr H_t^{(N)}=\bar h_{t,N}.
    \end{cases}
  \end{equation*}
  On $\{t<\tau_\delta\}$, the interval between $\mathscr H_t^{(N)}$ and
  $\bar h_{t,N}$ is contained in the interval defining $L_{t,N}$. The
  mean-value theorem and the definition of $D_{t,N}$ therefore give
  $\abs{D_{t,N}}\leq L_{t,N}$.

  The definitions of $e_t$ and $\widehat\zeta_{t+1}^{(N)}$, together with
  $\bar h_{t+1,N}=V_{A,\varrho}(\bar h_{t,N})$, give the error recursion
  \begin{equation*}
    e_{t+1}
    =V_{A,\varrho}(\mathscr H_t^{(N)})
    -V_{A,\varrho}(\bar h_{t,N})+\widehat\zeta_{t+1}^{(N)}
    =D_{t,N}e_t+\widehat\zeta_{t+1}^{(N)}.
  \end{equation*}
  Use $\alpha_{t,N}$ and $\beta_{t,N}$ from
  \cref{lem:subcritical-normalized-lattice-amplification} with
  $h_{t,N}=\bar h_{t,N}$, and set
  \begin{equation*}
    \zeta_{t+1} :=
    \frac{\widehat\zeta_{t+1}^{(N)}}{\bar h_{t,N}+a_N},
    \qquad
    A_{t,N}:=
    \one_{\{t<\tau_\delta\}}
    \frac{D_{t,N}(\bar h_{t,N}+a_N)}
    {\bar h_{t+1,N}+a_N}.
  \end{equation*}
  Dividing the error recursion by $\bar h_{t+1,N}+a_N$ shows that, on
  $\{t<\tau_\delta\}$,
  \begin{equation}\label{eq:subcritical-normalized-error-recursion}
    \mathcal R_{t+1,N}
    =
    A_{t,N}\mathcal R_{t,N}+\beta_{t,N}\zeta_{t+1}.
  \end{equation}
  Moreover, $\abs{A_{t,N}}\leq\alpha_{t,N}$: on
  $\{t<\tau_\delta\}$ this follows from $\abs{D_{t,N}}\leq L_{t,N}$, while
  $A_{t,N}=0$ on the complementary event.

  The hypotheses of
  \cref{lem:subcritical-normalized-lattice-amplification} hold with
  $C_0=C_\varrho$, because
  \cref{eq:subcritical-initial-radius-bound} gives
  $\bar h_{0,N}\leq C_\varrho$ and
  $\bar h_{t+1,N}=V_{A,\varrho}(\bar h_{t,N})$. Let
  \begin{equation*}
    K_N
    :=
    1\vee
    \max_{0\leq s<t\leq\bar T_N}
    \left\{
      \prod_{u=s}^{t-1}(1\vee\alpha_{u,N}),\,
      \beta_{s,N}
      \prod_{u=s+1}^{t-1}(1\vee\alpha_{u,N})
    \right\}.
  \end{equation*}
  By \cref{lem:subcritical-normalized-lattice-amplification},
  $K_N\leq\exp\{C(\log\log N)^2\}$.

  \emph{Step 4: $\Pp(\tau_\delta\leq\bar T_N)\to0$.}
  Set
  \begin{equation*}
    \eta_{t+1,N}
    :=
    \one_{\{t<\tau_\delta\}}\beta_{t,N}\zeta_{t+1},
    \qquad
    \sigma_{t,N}:=\beta_{t,N}.
  \end{equation*}
  On $\{t<\tau_\delta\}$,
  \cref{eq:subcritical-normalized-error-recursion} becomes
  $\mathcal R_{t+1,N}=A_{t,N}\mathcal R_{t,N}+\eta_{t+1,N}$.
  The variables $A_{t,N}$ and $\eta_{t+1,N}$ are respectively
  $\mathcal F_t$- and $\mathcal F_{t+1}$-measurable. Since
  $\widehat\zeta_{t+1}^{(N)}$ is conditionally centered and its denominator in
  $\zeta_{t+1}$ is deterministic,
  \begin{equation*}
    \Ee[\eta_{t+1,N}\mid\mathcal F_t]=0.
  \end{equation*}
  \cref{eq:subcritical-fixed-noise-variance} gives
  \begin{equation*}
    \Ee[\eta_{t+1,N}^2\mid\mathcal F_t]
    =
    \one_{\{t<\tau_\delta\}}\beta_{t,N}^2
    \frac{
      \Ee[(\widehat\zeta_{t+1}^{(N)})^2\mid\mathcal F_t]
    }{
      (\bar h_{t,N}+a_N)^2
    }
    \leq
    \frac{C\beta_{t,N}^2}{N}.
  \end{equation*}
  Thus the hypotheses of \cref{lem:common-stopped-orbit-tracking} hold with
  $T_N=\bar T_N$, $\tau_N=\tau_\delta$, and the amplification constant $K_N$
  defined above. Since $\mathscr H_0^{(N)}=\bar h_{0,N}$, we have
  $\mathcal R_{0,N}=0$. By \cref{eq:common-orbit-exit-bound},
  \begin{align*}
    \Pp\left(
      \sup_{0\leq t\leq\bar T_N}\abs{\mathcal R_{t,N}}>\delta
    \right)
    &=
    \Pp(\tau_\delta\leq\bar T_N)
    \\
    &\leq
    \frac{CK_N^2\bar T_N}{\delta^2N}
    \leq
    \frac{C\bar T_N}{\delta^2N}
    \exp\{C'(\log\log N)^2\}
    \longrightarrow0.
  \end{align*}
  This holds for every fixed $\delta\in(0,1/4)$, and hence
  $\sup_{0\leq t\leq\bar T_N}\abs{\mathcal R_{t,N}}\to0$ in probability. By
  the definition of $\mathcal R_{t,N}$, this is
  \cref{eq:subcritical-exact-radius-concentration}.
\end{proof}

\subsection{Terminal layer and cutoff}
\label{subsec:subcritical-terminal-cutoff}

\begin{theorem}[Fixed-precision subcritical dimension cutoff]
\label{thm:subcritical-dimension-cutoff}
  Assume $A<A_{\mathrm{lat}}$, fix $\varrho\in(0,1)$, and let
  $x_N\in[-1,1]^N$ be deterministic. Suppose that $\bar t_N\to\infty$.
  Then
  \begin{equation}\label{eq:subcritical-dimension-cutoff}
    \Pp\left(
      \tau_{\varrho}^{(N)}
      >
      \bar t_N-1
    \right)
    \to1,
    \qquad
    \Pp\left(
      \tau_{\varrho}^{(N)}
      >
      \bar t_N+2
    \right)
    \to0.
  \end{equation}
  Equivalently, by the total-variation/absorption identity
  \cref{eq:tv-absorption},
  \begin{equation}\label{eq:subcritical-dimension-tv-cutoff}
    \norm{
      P_{\varrho,A,N}^{\bar t_N-1}
      (Q_\varrho x_N,\cdot)-\delta_0
    }_{\TV}
    \to1,
    \qquad
    \norm{
      P_{\varrho,A,N}^{\bar t_N+2}
      (Q_\varrho x_N,\cdot)-\delta_0
    }_{\TV}
    \to0 .
  \end{equation}
\end{theorem}

Thus $Y^{(N)}$ has total-variation cutoff to $\delta_0$ at $\bar t_N$ with
a window of order one. In particular, for every fixed $\varepsilon\in(0,1)$,
\begin{equation*}
  t_{\mathrm{mix}}^{(N)}(\varepsilon)
  =
  \bar t_N+O(1).
\end{equation*}

\begin{proof}
  \emph{Step 1: Concentration up to the terminal time.}
  Put $a_N:=(\log N)^{-1}$ and $\bar T_N:=\bar t_N+1$.
  By \cref{lem:subcritical-lattice-regimes,eq:subcritical-exact-lattice-map-contraction},
  \begin{equation*}
    \bar h_{t,N}\leq C_\varrho\theta_A^t,
    \qquad t\geq0,
  \end{equation*}
  for some $\theta_A\in(0,1)$. The definition of $\bar t_N$ therefore gives
  $\bar t_N=O(\log\log N)$, and the same bound holds for $\bar T_N$. Since
  $(\log\log N)^2=o(\log N)$, there is $C_0<\infty$ such that, for every
  fixed $C<\infty$ and all sufficiently large $N$,
  \begin{equation*}
    \frac{\bar T_N}{N}\exp\{C(\log\log N)^2\}
    \leq
    \frac{C_0\log\log N}{N}\exp\{C(\log\log N)^2\}
    \longrightarrow0.
  \end{equation*}
  The assumptions
  $A<A_{\mathrm{lat}}$, fixed $\varrho\in(0,1)$, and
  $x_N\in[-1,1]^N$ are precisely the remaining hypotheses of
  \cref{prop:subcritical-exact-radius-concentration}. Hence
  \cref{eq:subcritical-exact-radius-concentration} holds on
  $[0,\bar T_N]$.

  \emph{Step 2: The lower bound.}
  Since $\bar t_N\to\infty$, we have
  $\bar t_N\geq1$ for all sufficiently large $N$. By the definition of
  $\bar t_N$,
  \begin{equation*}
    \bar h_{\bar t_N-1,N}>a_N .
  \end{equation*}
  On the event $\{\mathscr H_{\bar t_N-1}^{(N)}=0\}$, it follows that
  \begin{equation*}
    \frac{
      \abs{\mathscr H_{\bar t_N-1}^{(N)}-\bar h_{\bar t_N-1,N}}
    }{
      \bar h_{\bar t_N-1,N}+a_N
    }
    =
    \frac{\bar h_{\bar t_N-1,N}}{\bar h_{\bar t_N-1,N}+a_N}
    >
    \frac12 .
  \end{equation*}
  \cref{prop:subcritical-exact-radius-concentration} therefore gives
  $\Pp(\mathscr H_{\bar t_N-1}^{(N)}=0)\to0$. By
  \cref{eq:subcritical-exact-absorption} and the fact that $0$ is absorbing,
  this is equivalent to
  $\Pp(\tau_{\varrho}^{(N)}>\bar t_N-1)\to1$.

  \emph{Step 3: The one-step absorption criterion.}
  Define
  \begin{equation}\label{eq:subcritical-exact-round-to-zero-threshold}
    b_\varrho:=\varrho^{-1}\arctanh(\varrho/2).
  \end{equation}
  By the definition of $Q_1$, conditionally on $\mathscr H_t^{(N)}=h$,
  $Y_{t+1}^{(N)}=0$ precisely when
  \begin{equation*}
    \abs{
      \varrho^{-1}\tanh(\varrho A\sqrt h\,G_{t+1,i})
    }
    \leq\frac12,
    \qquad 1\leq i\leq N.
  \end{equation*}
  Since $\tanh$ is increasing, this event is equivalently
  \begin{equation}\label{eq:subcritical-exact-round-to-zero}
    \max_{1\leq i\leq N}\abs{A\sqrt h\,G_{t+1,i}}
    \leq b_\varrho .
  \end{equation}

  \emph{Step 4: Entrance below the terminal scale.}
  Set
  \begin{equation*}
    s_N:=\bar t_N+1=\bar T_N.
  \end{equation*}
  By the definition of $\bar t_N$, we have
  $\bar h_{\bar t_N,N}\leq a_N$. The small-radius estimate in
  \cref{lem:subcritical-fixed-precision-map-estimates} therefore yields,
  for some constants $c,C>0$ and all sufficiently large $N$,
  \begin{equation*}
    \bar h_{s_N,N}
    =
    V_{A,\varrho}(\bar h_{\bar t_N,N})
    \leq
    C\exp\left(-\frac{c}{a_N}\right)
    =
    CN^{-c} .
  \end{equation*}
  Here the inequality is immediate if $\bar h_{\bar t_N,N}=0$. Consequently,
  $\bar h_{s_N,N}\log N\to0$ and
  \begin{equation*}
    (\bar h_{s_N,N}+a_N)\log N\to1 .
  \end{equation*}
  Since $s_N=\bar T_N$, the uniform concentration estimate
  \cref{eq:subcritical-exact-radius-concentration} applies at time $s_N$.
  Writing the exact radius as its deterministic value plus the normalized
  tracking error, we obtain
  \begin{equation*}
    \begin{split}
      \mathscr H_{s_N}^{(N)}\log N
      &=
      \bar h_{s_N,N}\log N
      \\
      &\quad+
      \frac{
        \mathscr H_{s_N}^{(N)}-\bar h_{s_N,N}
      }{
        \bar h_{s_N,N}+a_N
      }
      (\bar h_{s_N,N}+a_N)\log N
      \xrightarrow{\Pp}0 .
    \end{split}
  \end{equation*}

  \emph{Step 5: Final absorption.}
  Fix
  \begin{equation*}
    0<\eta<\frac{b_\varrho^2}{2A^2}
  \end{equation*}
  and define
  \begin{equation*}
    M_{s_N+1,N}
    :=
    \max_{1\leq i\leq N}\abs{G_{s_N+1,i}},
  \end{equation*}
  where $G_{s_N+1}$ is the standard Gaussian vector used in the update from
  $s_N$ to $s_N+1$. By \cref{eq:subcritical-exact-round-to-zero}, the event
  $\{\tau_{\varrho}^{(N)}>s_N+1\}$ requires
  $A\sqrt{\mathscr H_{s_N}^{(N)}}M_{s_N+1,N}>b_\varrho$. Consequently,
  \begin{equation*}
    \{\tau_{\varrho}^{(N)}>s_N+1\}
    \cap
    \left\{\mathscr H_{s_N}^{(N)}\leq\frac{\eta}{\log N}\right\}
    \subseteq
    \left\{
      M_{s_N+1,N}>
      \frac{b_\varrho}{A}
      \sqrt{\frac{\log N}{\eta}}
    \right\}.
  \end{equation*}
  Combining the convergence
  $\mathscr H_{s_N}^{(N)}\log N\to0$ established in Step 4 with Gaussian tails
  and a union bound over the $N$ coordinates gives
  \begin{align*}
    \Pp(\tau_{\varrho}^{(N)}>\bar t_N+2)
    &\leq
    \Pp\left(
      \mathscr H_{s_N}^{(N)}\log N>\eta
    \right)
    +
    \Pp\left(
      M_{s_N+1,N}>
      \frac{b_\varrho}{A}
      \sqrt{\frac{\log N}{\eta}}
    \right)
    \\
    &\leq
    \Pp\left(
      \mathscr H_{s_N}^{(N)}\log N>\eta
    \right)
    +
    2N^{1-b_\varrho^2/(2A^2\eta)}
    \longrightarrow0.
  \end{align*}
  Hence \cref{eq:subcritical-dimension-cutoff} holds. The total variation formulation
  \cref{eq:subcritical-dimension-tv-cutoff} follows from the absorption
  identity \cref{eq:tv-absorption}. The mixing-time estimate follows from the
  monotonicity of total variation distance.
\end{proof}

\begin{remark}
\label{rem:subcritical-bounded-window-mechanism}
  Put
  \begin{equation*}
    b_\varrho:=\varrho^{-1}\arctanh(\varrho/2),
    \qquad
    \kappa_{A,\varrho}:=\frac{b_\varrho^2}{2A^2}.
  \end{equation*}
  At radius $h=(\log N)^{-1}$, the expected number of coordinates not rounded
  to zero is
  \begin{equation*}
    2N\Pp\left(G>\frac{b_\varrho\sqrt{\log N}}A\right)
    \sim
    \frac{2A}{b_\varrho\sqrt{2\pi}}
    \frac{N^{1-\kappa_{A,\varrho}}}{\sqrt{\log N}}.
  \end{equation*}
  Here $\sim$ denotes asymptotic equivalence as $N\to\infty$. The relation
  follows from the Gaussian Mills ratio.
  Thus, when $\kappa_{A,\varrho}<1$, entrance below $(\log N)^{-1}$ does not
  uniformly imply absorption in the next step. One further deterministic
  iterate is $o((\log N)^{-1})$, and the following update absorbs with
  probability tending to one. The upper time improves from $\bar t_N+2$ to
  $\bar t_N+1$ if
  \begin{equation*}
    \limsup_{N\to\infty}
    (\log N)\bar h_{\bar t_N,N}
    <
    \kappa_{A,\varrho}.
  \end{equation*}
\end{remark}

\section{Metastability at fixed precision}
\label{sec:rounded-metastability}

Fix $A>0$ and $\varrho\in(0,1)$. We identify the positive-drift components of
$V_{A,\varrho}$. When the rightmost component $(h_-,h_+)$ exists,
$\mathscr H^{(N)}$ enters a neighborhood of $h_+$ in bounded time and remains
there for $\exp(cN)$ steps. For each fixed $N$, the finite-state chain
$Y^{(N)}$ nevertheless absorbs almost surely.

\subsection{Positive-drift components}
\label{subsec:metastable-positive-rounded-wells}

Let $\mathscr H_t^{(N)}$ be the chain from
\cref{eq:subcritical-exact-grid-radius}. Its mean map is
$V_{A,\varrho}$ from \cref{eq:rounded-radius-map}. Throughout this section we
write
\begin{equation*}
  V:=V_{A,\varrho}.
\end{equation*}
Let
\begin{equation}\label{eq:metastable-radius-bound}
  M_\varrho
  :=
  \max_{\abs u\leq\varrho^{-1}} Q_1(u)^2
  <\infty .
\end{equation}
For $\alpha>0$, define
\begin{equation}\label{eq:metastable-rounded-profile}
  m_\varrho(\alpha)
  :=
  \Ee\left[
    Q_1\left(
      \varrho^{-1}\tanh(\varrho\alpha G)
    \right)^2
  \right].
\end{equation}
The rounding monotonicity in \cref{eq:subcritical-rounding-monotonicity}
implies that $V$ is nondecreasing on $[0,\infty)$. In the layer notation of
\cref{eq:subcritical-layer-formulas-exact},
\begin{equation}\label{eq:metastable-rounded-profile-layer}
  m_\varrho(\alpha)
  =
  2\sum_{k\in\mathcal K_\varrho}
  (2k+1)
  \bar\Phi\left(\frac{b_{\varrho,k}}{\alpha}\right).
\end{equation}
Recall the threshold introduced in \cref{eq:intro-metastable-threshold}:
\begin{equation}\label{eq:metastable-exact-threshold}
  A_{\mathrm{ex}}(\varrho)^2
  :=
  \inf_{\alpha>0}
  \frac{\alpha^2}{m_\varrho(\alpha)} .
\end{equation}

Define the positive-drift set
\begin{equation}\label{eq:metastable-positive-region}
  \mathcal O_{A,\varrho}
  :=
  \{h\in(0,M_\varrho):V_{A,\varrho}(h)>h\}.
\end{equation}
We call the connected components of $\mathcal O_{A,\varrho}$ positive-drift
components. The rightmost component produces metastability near its attracting
right endpoint. The corresponding vector states form a grid shell.

We first identify the fixed points that bound the positive-drift components.
\begin{proposition}
\label{prop:metastable-positive-fixed-points}
  The threshold $A_{\mathrm{ex}}(\varrho)$ belongs to $(0,\infty)$. If
  $A>A_{\mathrm{ex}}(\varrho)$, then $V_{A,\varrho}$ has at least two
  positive fixed points in
  $(0,M_\varrho)$. There exist
  $0<\alpha_-<\alpha_+$ such that
  \begin{equation}\label{eq:metastable-alpha-fixed-points}
    A^2m_\varrho(\alpha_\pm)=\alpha_\pm^2,
  \end{equation}
  and $\widehat h_\pm:=\alpha_\pm^2/A^2$ are fixed points of
  $V_{A,\varrho}$.

  Moreover, $\mathcal O_{A,\varrho}$ is nonempty and has finitely many
  connected components. If $(h_-,h_+)$ is its rightmost positive-drift
  component, then
  \begin{equation}\label{eq:metastable-component-fixed-points}
    0<h_-<h_+<M_\varrho,
    \qquad
    V(h_-)=h_-,
    \qquad
    V(h_+)=h_+,
  \end{equation}
  and $V(h)>h$ on $(h_-,h_+)$. The fixed point $h_+$ is isolated, and there
  exists $\eta_0>0$ such that
  \begin{equation}\label{eq:metastable-right-stability}
    h_+-\eta_0>h_-,
    \qquad
    h_++\eta_0<M_\varrho,
    \qquad
    V(h)<h
    \quad\mbox{for }h\in(h_+,h_++\eta_0].
  \end{equation}
\end{proposition}

\begin{proof}
  \emph{Step 1: $A_{\mathrm{ex}}(\varrho)\in(0,\infty)$.}
  By the definitions in \cref{eq:subcritical-layer-thresholds}, the set
  $\mathcal K_\varrho$ is finite and nonempty, and each threshold
  $b_{\varrho,k}$ is positive. Hence
  \cref{eq:metastable-rounded-profile-layer} shows that $m_\varrho$ is
  continuous and positive on $(0,\infty)$.

  To control the quotient in \cref{eq:metastable-exact-threshold}, put
  $b_{\varrho,0}=\varrho^{-1}\arctanh(\varrho/2)$. The random variable in
  \cref{eq:metastable-rounded-profile} can be nonzero only when
  $\abs G>b_{\varrho,0}/\alpha$, and its square is bounded by $M_\varrho$.
  Hence
  \begin{equation*}
    m_\varrho(\alpha)
    \leq
    2M_\varrho
    \bar\Phi\left(\frac{b_{\varrho,0}}{\alpha}\right),
    \qquad
    \alpha>0,
  \end{equation*}
  and the Gaussian tail bound implies
  $m_\varrho(\alpha)/\alpha^2\to0$ as $\alpha\downarrow0$. The bound
  $m_\varrho(\alpha)\leq M_\varrho$ also gives
  $m_\varrho(\alpha)/\alpha^2\to0$ as $\alpha\to\infty$. Therefore
  \begin{equation}\label{eq:metastable-threshold-endpoints}
    \frac{\alpha^2}{m_\varrho(\alpha)}
    \longrightarrow\infty
    \quad\mbox{as }\alpha\downarrow0
    \quad\mbox{and as }\alpha\to\infty .
  \end{equation}
  Since the quotient is continuous and strictly positive on $(0,\infty)$,
  \cref{eq:metastable-threshold-endpoints} implies that the infimum in
  \cref{eq:metastable-exact-threshold} is attained on a compact subinterval
  of $(0,\infty)$ and is positive and finite. This proves
  $A_{\mathrm{ex}}(\varrho)\in(0,\infty)$.

  \emph{Step 2: $V_{A,\varrho}$ has two fixed points in
  $(0,M_\varrho)$.}
  Assume now that $A>A_{\mathrm{ex}}(\varrho)$. Then there is
  $\alpha_0>0$ such that
  $A^2m_\varrho(\alpha_0)-\alpha_0^2>0$. By
  \cref{eq:metastable-threshold-endpoints},
  $A^2m_\varrho(\alpha)-\alpha^2<0$ for all sufficiently small and all
  sufficiently large $\alpha$. The intermediate value theorem therefore gives
  two roots
  $0<\alpha_-<\alpha_+$ of
  $A^2m_\varrho(\alpha)=\alpha^2$.

  Since
  \begin{equation*}
    V_{A,\varrho}(h)=m_\varrho(A\sqrt h),
  \end{equation*}
  the points $\widehat h_\pm=\alpha_\pm^2/A^2$ satisfy
  \begin{equation*}
    V_{A,\varrho}(\widehat h_\pm)
    =
    m_\varrho(\alpha_\pm)
    =
    \frac{\alpha_\pm^2}{A^2}
    =
    \widehat h_\pm .
  \end{equation*}
  They are positive because $\alpha_\pm>0$. Moreover,
  \begin{equation*}
    \widehat h_\pm
    =
    V_{A,\varrho}(\widehat h_\pm)
    =
    \Ee\left[
      Q_1\left(
        \varrho^{-1}
        \tanh(\varrho\alpha_\pm G)
      \right)^2
    \right],
  \end{equation*}
  where the integrand is bounded by $M_\varrho$ and vanishes when
  $\abs G\leq b_{\varrho,0}/\alpha_\pm$. Thus
  $0<\widehat h_\pm<M_\varrho$.

  \emph{Step 3: $\mathcal O_{A,\varrho}$ is nonempty and has finitely many
  components.}
  Set $D(h):=V(h)-h$. The point $h_0=\alpha_0^2/A^2$ satisfies
  $D(h_0)>0$. Since $V(h_0)\leq M_\varrho$, we also have $h_0<M_\varrho$,
  and hence $\mathcal O_{A,\varrho}$ is nonempty.

  Substituting $\alpha=A\sqrt h$ into
  \cref{eq:metastable-threshold-endpoints} gives
  \begin{equation*}
    \frac{V(h)}h
    =
    A^2\frac{m_\varrho(A\sqrt h)}{(A\sqrt h)^2}
    \longrightarrow0
    \qquad\mbox{as }h\downarrow0.
  \end{equation*}
  Thus $D(h)<0$ for all sufficiently small $h>0$. Moreover,
  $V(M_\varrho)<M_\varrho$, because the integrand in its definition is bounded
  by $M_\varrho$ and vanishes with positive probability. By continuity, there
  is $\delta>0$ such that
  \begin{equation*}
    D(h)<0
    \quad\mbox{for }
    h\in(0,\delta]\cup[M_\varrho-\delta,M_\varrho].
  \end{equation*}

  By \cref{eq:metastable-rounded-profile-layer}, $V$ and hence $D$ are real
  analytic on $(0,\infty)$. The function $D$ is not identically zero because
  it is negative near the origin. Its zeros are therefore isolated, so it has
  only finitely many zeros in $[\delta,M_\varrho-\delta]$. Every component of
  $\mathcal O_{A,\varrho}$ is contained in this interval and has endpoints in
  this zero set. Consequently, $\mathcal O_{A,\varrho}$ has finitely many
  components.

  \emph{Step 4: $V(h)<h$ immediately to the right of $h_+$.}
  Let $(h_-,h_+)$ be the rightmost component. Since $D$ is continuous,
  $D(h_-)=D(h_+)=0$, while $D(h)>0$ for $h\in(h_-,h_+)$. This proves
  \cref{eq:metastable-component-fixed-points}.

  The zero $h_+$ is isolated, so $D$ has a constant nonzero sign immediately
  to its right. This sign cannot be positive, since that would produce a
  component of $\mathcal O_{A,\varrho}$ to the right of $(h_-,h_+)$. We may
  therefore choose $\eta_0>0$ such that
  $h_+-\eta_0>h_-$, $h_++\eta_0<M_\varrho$, and $V(h)<h$ for every
  $h\in(h_+,h_++\eta_0]$. This proves
  \cref{eq:metastable-right-stability}.
\end{proof}

\subsection{Entrance and exponential persistence}
\label{subsec:metastable-entrance-persistence}

Conditionally on $\mathscr H_t^{(N)}=h$, the random variables in
\cref{eq:subcritical-exact-grid-transition} are independent and take values
in $[0,M_\varrho]$. Hence Hoeffding's inequality gives, for every $u>0$,
\begin{equation}\label{eq:metastable-hoeffding}
  \Pp\left(
    \abs{\mathscr H_{t+1}^{(N)}-V(\mathscr H_t^{(N)})}>u
    \,\middle|\, \mathscr H_t^{(N)}
  \right)
  \leq
  2\exp\left(-\frac{2Nu^2}{M_\varrho^2}\right).
\end{equation}

\begin{theorem}[Metastability from the rightmost positive-drift component]
\label{thm:rounded-qualitative-metastability}
  Assume that $A>A_{\mathrm{ex}}(\varrho)$, and let
  $(h_-,h_+)$ be the rightmost positive-drift component of
  $\mathcal O_{A,\varrho}$ in
  \cref{eq:metastable-positive-region}. Let
  $\eta_0>0$ be as in \cref{eq:metastable-right-stability}, and let
  $B\subset(h_-,h_+]$ be compact. For every $0<\eta<\eta_0$,
  there are $T_\eta\in\N_0$ and constants $c_0,c_1,C\in(0,\infty)$, depending
  only on $A,\varrho,B,\eta$, with $c_1<c_0$, such that the following holds. If
  $Y_0^{(N)}$ is deterministic and
  $\mathscr H_0^{(N)}\in B$, then
  \begin{equation}\label{eq:metastable-entrance-probability}
    \Pp\left(
      \abs{\mathscr H_{T_\eta}^{(N)}-h_+}>\frac\eta2
    \right)
    \leq
    C\exp(-c_0N),
  \end{equation}
  and, for every $T\in\N_0$,
  \begin{equation}\label{eq:metastable-exit-probability}
    \Pp\left(
      \max_{0\leq s\leq T}
      \abs{\mathscr H_{T_\eta+s}^{(N)}-h_+}>\eta
    \right)
    \leq
    C(1+T)\exp(-c_0N).
  \end{equation}
  In particular,
  \begin{equation}\label{eq:metastable-exponential-survival}
    \Pp\left(
      \tau_\varrho^{(N)}
      >
      T_\eta+\lfloor\exp(c_1N)\rfloor
    \right)
    \to1 .
  \end{equation}
  For every fixed $N$, however, the chain is eventually absorbed almost surely
  from every deterministic initial state.
\end{theorem}

\begin{proof}
  \emph{Step 1: $V^t(h)\to h_+$ uniformly for $h\in B$.}
  By \cref{eq:subcritical-rounding-monotonicity}, the map $V$ is
  nondecreasing. Since $(h_-,h_+)$ is a component of
  $\mathcal O_{A,\varrho}$,
  \begin{equation*}
    h<V(h)\leq h_+
    \quad\mbox{for }h\in(h_-,h_+),
  \end{equation*}
  and $V(h_+)=h_+$. Thus $V^t(h)$ is nondecreasing in $t$ and bounded above by
  $h_+$ for every $h\in(h_-,h_+]$. Its limit is a fixed point of $V$, and
  there is no fixed point in $(h_-,h_+)$. Hence $V^t(h)\to h_+$.

  Since $B$ is compact in $(h_-,h_+]$, choose $b>h_-$ such that
  $B\subset[b,h_+]$. On the compact interval
  $[b,h_+-\eta/4]$, when it is nonempty, the continuous function $V(h)-h$
  has a positive minimum. Thus every orbit from $B$ enters the invariant
  interval $[h_+-\eta/4,h_+]$ after a uniformly bounded number of steps.
  Hence there exists $T_\eta\in\N$ such that
  \begin{equation}\label{eq:metastable-deterministic-entrance}
    V^{T_\eta}(B)
    \subset
    [h_+-\eta/4,h_++\eta/4].
  \end{equation}

  \emph{Step 2: $\mathscr H_{T_\eta}^{(N)}$ lies within $\eta/2$ of $h_+$
  with exponentially high probability.}
  Define the modulus of continuity
  \begin{equation*}
    \omega(r)
    :=
    \sup\bigl\{\abs{V(u)-V(v)}:
    u,v\in[0,M_\varrho],\ \abs{u-v}\leq r\bigr\}.
  \end{equation*}
  Then $\omega(r)\to0$ as $r\downarrow0$. Set
  $q_s:=V^s(\mathscr H_0^{(N)})$ and
  $e_s:=\abs{\mathscr H_s^{(N)}-q_s}$. Since $e_0=0$,
  \begin{equation*}
    e_{s+1}
    \leq
    \abs{\mathscr H_{s+1}^{(N)}-V(\mathscr H_s^{(N)})}
    +\omega(e_s).
  \end{equation*}
  For $\delta>0$, define $E_0(\delta):=0$ and
  $E_{s+1}(\delta):=\delta+\omega(E_s(\delta))$. For each fixed $s$,
  $E_s(\delta)\to0$ as $\delta\downarrow0$. Since $T_\eta$ is fixed, we may
  choose $\delta_0>0$ such that
  $\max_{0\leq s\leq T_\eta}E_s(\delta_0)\leq\eta/4$.

  Suppose that
  $\mathscr H_0^{(N)}\in B$ and
  \begin{equation*}
    \abs{
      \mathscr H_{s+1}^{(N)}-V(\mathscr H_s^{(N)})
    }
    \leq\delta_0
    \qquad\mbox{for }0\leq s<T_\eta.
  \end{equation*}
  Then induction gives
  $e_s\leq E_s(\delta_0)\leq\eta/4$ for $s\leq T_\eta$. Hence
  \cref{eq:metastable-deterministic-entrance} yields
  \begin{equation*}
    \abs{\mathscr H_{T_\eta}^{(N)}-h_+}
    \leq
    e_{T_\eta}+\abs{q_{T_\eta}-h_+}
    \leq
    \frac\eta4+\frac\eta4
    =\frac\eta2.
  \end{equation*}
  A union bound and \cref{eq:metastable-hoeffding} give
  \begin{equation*}
    \Pp\left(
      \abs{\mathscr H_{T_\eta}^{(N)}-h_+}>\frac\eta2
    \right)
    \leq
    2T_\eta
      \exp\left(-\frac{2N\delta_0^2}{M_\varrho^2}\right).
  \end{equation*}
  Thus \cref{eq:metastable-entrance-probability} holds with constants
  $C_{\mathrm{ent}},c_{\mathrm{ent}}>0$ depending only on
  $A,\varrho,B,\eta$.

  \emph{Step 3: The exit bound and survival up to $\exp(c_1N)$.}
  Put
  \begin{equation*}
    J_\eta:=[h_+-\eta,h_++\eta].
  \end{equation*}
  For $h\in[h_+-\eta,h_+)$, we have $h<V(h)\leq h_+$. By
  \cref{eq:metastable-right-stability} and the monotonicity of $V$,
  \begin{equation*}
    h_+\leq V(h)<h
    \quad\mbox{for }h\in(h_+,h_++\eta].
  \end{equation*}
  Consequently, $V(J_\eta)\subset J_\eta^\circ$. Since $J_\eta$ is compact,
  there exists
  $d_\eta>0$ such that
  \begin{equation*}
    \operatorname{dist}(V(J_\eta),J_\eta^c)\geq d_\eta .
  \end{equation*}
  Therefore, on the event $\mathscr H_t^{(N)}\in J_\eta$, the additional event
  \begin{equation*}
    \abs{\mathscr H_{t+1}^{(N)}-V(\mathscr H_t^{(N)})}\leq d_\eta/2
  \end{equation*}
  implies $\mathscr H_{t+1}^{(N)}\in J_\eta$. Applying
  \cref{eq:metastable-hoeffding} with $u=d_\eta/2$, and taking a union bound
  over $T_\eta,\ldots,T_\eta+T-1$, gives
  \begin{align*}
    &\Pp\left(
      \max_{0\leq s\leq T}
      \abs{\mathscr H_{T_\eta+s}^{(N)}-h_+}>\eta
    \right)
    \\
    &\qquad\leq
    C_{\mathrm{ent}}e^{-c_{\mathrm{ent}}N}
    +2T\exp\left(-\frac{Nd_\eta^2}{2M_\varrho^2}\right)
    \leq
    C(1+T)e^{-c_0N},
  \end{align*}
  where
  $c_0:=\min\{c_{\mathrm{ent}},d_\eta^2/(2M_\varrho^2)\}>0$.
  After enlarging $C$, this proves
  \cref{eq:metastable-entrance-probability,eq:metastable-exit-probability}
  with the same constants $C,c_0$.

  Since $J_\eta\subset(0,M_\varrho)$, the complementary event in
  \cref{eq:metastable-exit-probability} prevents absorption. Choose
  $c_1\in(0,c_0)$. Taking
  $T=\lfloor\exp(c_1N)\rfloor$, the right-hand side of
  \cref{eq:metastable-exit-probability} is bounded by
  \begin{equation*}
    2C\exp\bigl(-(c_0-c_1)N\bigr)
    \longrightarrow0.
  \end{equation*}
  This is \cref{eq:metastable-exponential-survival}.

  \emph{Step 4: $\tau_\varrho^{(N)}<\infty$ almost surely for fixed $N$.}
  Let
  \begin{equation*}
    \mathcal S_N
    :=
    (\varrho\Z)^N\cap[-1-\varrho/2,1+\varrho/2]^N .
  \end{equation*}
  The chain $Y^{(N)}$ is supported on the finite set $\mathcal S_N$. If
  $y\in\mathcal S_N\setminus\{0\}$, then, conditionally on $Y_t^{(N)}=y$, the
  coordinates of $\mathsf W_{t+1,A}^{(N)}y$ are independent nondegenerate
  Gaussian variables. Thus each coordinate has positive probability of
  satisfying $Q_\varrho(\tanh((\mathsf W_{t+1,A}^{(N)}y)_i))=0$. Independence
  therefore gives
  \begin{equation*}
    \Pp\left(Y_{t+1}^{(N)}=0\,\middle|\,Y_t^{(N)}=y\right)>0 .
  \end{equation*}
  Taking the minimum over $y\in\mathcal S_N\setminus\{0\}$ gives a constant
  $p_N>0$. The Markov property now gives
  $\Pp(\tau_\varrho^{(N)}>m)\leq(1-p_N)^m\to0$, which proves the claim.
\end{proof}

\section{Supercritical cutoff as \texorpdfstring{$N\to\infty$}{N tends to infinity}}
\label{sec:supercritical-dimension}

Fix $A>1$ and let $N\to\infty$. We use $Z^{(N)}$, $K_{A,N}$, and $V_A$ from
\cref{subsec:radius-observables}. We first prove uniqueness of the nonzero
invariant laws and identify the cutoff scale from the decay of
$V_A^t(q_0)-q_\ast$. Negative moments and concentration then localize
$\nu_{A,N}$ near $q_\ast$. Before the cutoff scale, the deterministic bias
separates the law of $Z^{(N)}$ from $\nu_{A,N}$. After that scale, synchronous
coupling makes the separation from a stationary copy of order $N^{-1/2}$, and
one further transition converts this bound into convergence in total
variation. Generic constants depend only on the fixed parameters in the
surrounding statement.

\subsection{Transition density and uniqueness of invariant laws}
\label{subsec:gaussian-exact-radius-chain}

For $q>0$, set $Y_q:=\tanh^2(A\sqrt q\,G)$. For $y\in(0,1)$, the change of
variables $u=\arctanh(\sqrt y)$ gives the strictly positive density
\begin{equation}\label{eq:gaussian-one-coordinate-density}
  f_q(y)
  =
  \frac{1}{A\sqrt{2\pi q}}
  \frac{1}{\sqrt y(1-y)}
  \exp\left(
    -\frac{\arctanh(\sqrt y)^2}{2A^2q}
  \right).
\end{equation}
The $N$-fold convolution of $f_q$, followed by scaling by $N^{-1}$, gives a
strictly positive transition density for $K_{A,N}$ on $(0,1)$. This yields
uniqueness of $\nu_{A,N}$. The kernel $J_{A,N}$ then yields uniqueness of
$\pi_{A,N}$.
\begin{proposition}
\label{prop:gaussian-unique-nonzero-invariant}
  Fix $A>0$ and $N\geq1$. If $K_{A,N}$ admits a nonzero invariant probability
  $\nu_{A,N}$ on $(0,1]$, then this probability is unique and has a strictly
  positive density on $(0,1)$.
\end{proposition}

\begin{proof}
  \emph{Step 1: The transition density is strictly positive.}
  The density of $\sum_{i=1}^N Y_{i,q}$ is the convolution
  $f_q^{\ast N}$. It is strictly positive on $(0,N)$: for $s\in(0,N)$,
  choose $y_1,\ldots,y_N\in(0,1)$ with $\sum_i y_i=s$ and integrate over a
  small neighborhood of $(y_1,\ldots,y_{N-1})$ on which all coordinates,
  including the remaining coordinate, stay in $(0,1)$. After scaling by
  $N^{-1}$, the transition kernel $K_{A,N}(q,\cdot)$ therefore has a density
  $p_{A,N}(q,r)$ that is strictly positive for every $r\in(0,1)$.

  \emph{Step 2: The invariant law is unique and has a positive density.}
  The positivity of $p_{A,N}$ makes $Z^{(N)}$ on $(0,1]$
  Lebesgue-irreducible. By Nummelin~\cite{nummelin1984}, it therefore admits at
  most one invariant probability. If $\nu K_{A,N}=\nu$, then invariance and
  Fubini's theorem give
  \begin{equation*}
    \nu(\mathrm{d}r)
    =
    \left(
      \int_{(0,1]}p_{A,N}(q,r)\,\nu(\mathrm{d}q)
    \right)\mathrm{d}r.
  \end{equation*}
  The density in parentheses is strictly positive for every $r\in(0,1)$.
\end{proof}

\begin{remark}\label{rem:gaussian-full-space-invariant}
  On $[0,1]$, the point $0$ is absorbing. Hence, whenever the nonzero
  invariant law exists,
  the invariant probabilities on $[0,1]$ are exactly
  \begin{equation}\label{eq:gaussian-full-space-invariant-laws}
    \alpha\delta_0+(1-\alpha)\nu_{A,N},
    \qquad 0\leq\alpha\leq1.
  \end{equation}
  Indeed, any invariant law has a fixed mass $\alpha$ at $\{0\}$, and its
  normalized restriction to $(0,1]$, if nonzero, is invariant for the
  restricted chain.
\end{remark}

\begin{corollary}[Uniqueness of the nonzero vector invariant law]
\label{cor:gaussian-unique-vector-invariant}
  Fix $A>0$ and $N\geq1$ for which $K_{A,N}$ admits a nonzero invariant law
  $\nu_{A,N}$ on $(0,1]$. Then $P_{A,N}$ has a unique invariant probability
  $\pi_{A,N}$ on $[-1,1]^N$ with $\pi_{A,N}(\{0\})=0$, namely
  \begin{equation}\label{eq:gaussian-unique-vector-invariant}
    \pi_{A,N}=\nu_{A,N}J_{A,N}.
  \end{equation}
\end{corollary}

\begin{proof}
  \emph{Existence.}
  By \cref{prop:gaussian-tv-reduction},
  $\nu_{A,N}J_{A,N}$ is invariant for $P_{A,N}$. It has no mass at $0$ because
  $J_{A,N}(q,\cdot)$ is absolutely continuous for every $q>0$.

  \emph{Uniqueness.}
  Let $\pi$ be invariant for $P_{A,N}$ with $\pi(\{0\})=0$. By
  \cref{prop:gaussian-tv-reduction}, $\nu:=(r_N)_\#\pi$ is invariant for
  $K_{A,N}$, and since $r_N(x)=0$ if and only if $x=0$, we have
  $\nu(\{0\})=0$. By \cref{prop:gaussian-unique-nonzero-invariant},
  $\nu=\nu_{A,N}$. Invariance and
  \cref{eq:gaussian-vector-kernel-factorization} give
  \begin{equation*}
    \pi
    =
    \pi P_{A,N}
    =
    \bigl((r_N)_\#\pi\bigr)J_{A,N}
    =
    \nu_{A,N}J_{A,N}.
  \end{equation*}
\end{proof}

\subsection{Deterministic bias and cutoff scale}
\label{subsec:gaussian-mean-field-cutoff}

For $A>1$, \cref{lem:gaussian-mean-field-concavity} gives the unique nonzero fixed point
$q_\ast$ and the multiplier $\mu_A=V_A'(q_\ast)\in(0,1)$.

Let $q_0\in(0,1]$ with $q_0\neq q_\ast$, and let $C(q_0)$ be the nonzero
constant in \cref{eq:gaussian-deterministic-asymptotic}.
By \cref{eq:gaussian-deterministic-asymptotic}, the deterministic bias decays
like $C(q_0)\mu_A^t$. The cutoff scale is therefore the time at which this
bias reaches the $N^{-1/2}$ fluctuation scale. Recall from
\cref{eq:gaussian-cutoff-time,eq:gaussian-integer-cutoff-time} that
\begin{equation*}
  t_N(q_0)
  :=
  \frac{\frac12\log N+\log\abs{C(q_0)}}{\abs{\log\mu_A}},
  \qquad
  n_N(c)
  :=
  \max\{0,\lfloor t_N(q_0)+c\rfloor\} .
\end{equation*}

\subsection{Localization of \texorpdfstring{$\nu_{A,N}$}{nu(A,N)} near
\texorpdfstring{$q_\ast$}{q*}}
\label{subsec:gaussian-process-cutoff}

We prove that $\nu_{A,N}$ assigns exponentially small mass to the complement
of a compact interval around $q_\ast$ on which $V_A$ is contractive. Negative
moments control a neighborhood of zero. Away from zero, uniform deterministic
entrance and fixed-time concentration give the localization estimate. The
auxiliary estimates are proved in \cref{app:supercritical-cutoff-estimates}.

\begin{assumption}[Macroscopic initial radius]
\label{ass:gaussian-stable-selected}
  We fix $A>1$ and let $q_\ast$ and $\mu_A$ be the unique nonzero fixed
  point and multiplier from \cref{lem:gaussian-mean-field-concavity}. We also fix
  $q_0\in(0,1]$ with $q_0\neq q_\ast$. Every deterministic initial-radius
  sequence used below is assumed to converge to this $q_0$.
\end{assumption}

Throughout the remainder of this subsection,
\cref{ass:gaussian-stable-selected} is in force.

For a deterministic initial-radius sequence $q_N\to q_0$, put
\begin{equation}\label{eq:gaussian-q-deterministic-orbit}
  q_{0,N}:=q_N,
  \qquad
  q_{t+1,N}:=V_A(q_{t,N}).
\end{equation}

The deterministic orbit must first enter a region where the mean map is
contractive. We record that entrance estimate here.
\begin{lemma}
\label{lem:gaussian-deterministic-entrance}
  Let \cref{ass:gaussian-stable-selected} hold, i.e. $q_N\to q_0$. There are
  an integer $s_\ast$, a compact interval $I_\ast\Subset(0,1)$, constants
  $\delta_\ast>0$ and $\kappa<1$, and $N_0<\infty$ such that, for all
  $N\geq N_0$,
  \begin{equation}\label{eq:gaussian-deterministic-entrance}
    q_{t,N}\in I_\ast\quad\mbox{for every }t\geq s_\ast,
    \qquad
    \inf_{t\geq s_\ast}
    \operatorname{dist}(q_{t,N},\partial I_\ast)
    \geq
    \delta_\ast,
  \end{equation}
  and
  \begin{equation}\label{eq:gaussian-stable-basin-derivative}
    \sup_{q\in I_\ast}\abs{V_A'(q)}\leq\kappa.
  \end{equation}
\end{lemma}

\begin{proof}
  Since $q_\ast$ is stable, choose a compact interval $I_\ast\Subset(0,1)$
  containing $q_\ast$ in its interior such that
  $\sup_{I_\ast}\abs{V_A'}<1$. Since $q_0$ lies in the basin of $q_\ast$,
  the orbit $V_A^t(q_0)$ enters the interior of $I_\ast$ and
  remains in a smaller compact subinterval of $I_\ast$ after a finite time.
  By continuity of $V_A^t$ for each fixed $t$, the same is true uniformly
  for all $q_N$ sufficiently close to $q_0$, after increasing the entrance time
  if necessary. This proves \cref{eq:gaussian-deterministic-entrance}, while
  the derivative bound follows from the choice of $I_\ast$.
\end{proof}

The next estimates control invariant mass near zero. They use negative
moments of order $\gamma N$, first for $\ell_{A,N}$ and then uniformly for
small radii.

\begin{remark}
\label{rem:gaussian-ell-negative-moments}
  For every positive random variable $X$, Markov's inequality gives
  $\Pp(X\leq r)\leq r^p\Ee X^{-p}$, so negative moments provide small-ball
  bounds. Near $q=0$,
  \begin{equation*}
    \frac{F_{A,N}(q,G)}{q}
    \approx
    A^2\frac{\chi_N^2}{N}.
  \end{equation*}
  Here $\chi_N$ is the Euclidean norm of a standard Gaussian vector in
  $\R^N$, as in \cref{eq:gaussian-ell}.
  To obtain an exponentially small bound on $\nu_{A,N}((0,r])$, we take
  $p=\gamma N$. If $Z_N\sim\chi_N^2$ and
  $\gamma\in(0,1/2)$, direct integration gives
  \begin{equation*}
    \Ee\left[
      \left(
        A^2\frac{Z_N}{N}
      \right)^{-\gamma N}
    \right]
    =
    A^{-2\gamma N}
    \left(\frac N2\right)^{\gamma N}
    \frac{\Gamma(N/2-\gamma N)}{\Gamma(N/2)}.
  \end{equation*}
  The moment is finite precisely when $\gamma<1/2$, and Stirling's formula
  yields
  \begin{equation*}
    \frac1N
    \log
    \Ee\left[
      \left(
        A^2\frac{Z_N}{N}
      \right)^{-\gamma N}
    \right]
    \to
    \Lambda_A(\gamma)
    :=
    -2\gamma\log A
    +
    \gamma
    +
    \left(\frac12-\gamma\right)\log(1-2\gamma).
  \end{equation*}
  Since $\Lambda_A(0)=0$ and
  $\Lambda_A'(0)=-2\log A<0$ for $A>1$, this exponent is negative for all
  sufficiently small $\gamma>0$. The truncated estimate below makes this
  linearized calculation uniform over a neighborhood of the origin.
\end{remark}

\begin{lemma}
\label{lem:gaussian-truncated-cramer}
  Let $A>1$. There are constants
  $\gamma\in(0,1/2)$, $\kappa>0$, $L\geq1$, and
  $\varepsilon\in(0,1)$ such that, for all sufficiently large $N$,
  \begin{equation}\label{eq:gaussian-truncated-cramer}
    \Ee\left[
      \left(
        (1-\varepsilon)A^2
        \frac1N\sum_{i=1}^N\min(G_i^2,L^2)
      \right)^{-\gamma N}
    \right]
    \leq
    e^{-\kappa N}.
  \end{equation}
\end{lemma}

For sufficiently small $R_0$, the ratio $F_{A,N}(q,G)/q$ dominates the
truncated quadratic average in \cref{lem:gaussian-truncated-cramer}, uniformly
for $0<q\leq R_0$. Since $x\mapsto x^{-\gamma N}$ is decreasing, the lemma
then gives the required negative-moment contraction near zero.
\begin{proposition}
\label{prop:gaussian-near-zero-negative-moment}
  Let \cref{ass:gaussian-stable-selected} hold, and let $I_\ast$ be the compact
  interval selected in \cref{lem:gaussian-deterministic-entrance}. There are
  $R_0\in(0,\inf I_\ast)$, $\gamma\in(0,1/2)$, and $\kappa>0$ such that, for
  all sufficiently large $N$,
  \begin{equation}\label{eq:gaussian-lyapunov-near-zero-input}
    \sup_{0<q\leq R_0}
    \Ee\left[
      \left(
        \frac{F_{A,N}(q,G)}{q}
      \right)^{-\gamma N}
    \right]
    \leq
    e^{-\kappa N}.
  \end{equation}
\end{proposition}

Away from zero, there is $c_{A,R_0}>0$ such that
$\tanh^2(A\sqrt q\,g)\geq c_{A,R_0}\min(g^2,1)$ for $q\in[R_0,1]$.
After averaging over the coordinates, this reduces the estimate to a
negative-moment bound for a truncated Gaussian square average.
\begin{lemma}
\label{lem:gaussian-outside-negative-moment}
  Let $R_0\in(0,1)$ and $\gamma\in(0,1/2)$. There is $b<\infty$ such that, for
  every $N\geq1$,
  \begin{equation}\label{eq:gaussian-lyapunov-outside-input}
    \sup_{q\in[R_0,1]}
    \Ee\left[
      F_{A,N}(q,G)^{-\gamma N}
    \right]
    \leq
    e^{bN}.
  \end{equation}
\end{lemma}

See \cref{app:supercritical-negative-moments} for the proofs of \cref{lem:gaussian-truncated-cramer,prop:gaussian-near-zero-negative-moment,lem:gaussian-outside-negative-moment}.

By \cref{prop:gaussian-nonzero-invariant-existence,prop:gaussian-unique-nonzero-invariant},
and the fact that $A_c(N)\downarrow1$, $K_{A,N}$ has a unique nonzero
invariant law for all sufficiently large $N$. We denote it by $\nu_{A,N}$.

\begin{remark}
\label{rem:gaussian-power-law-lyapunov}
  For fixed $A>1$ and all sufficiently large $N$,
  \begin{equation*}
    \Ee\left[\left(A^2\frac{\chi_N^2}{N}\right)^{-1}\right]
    =A^{-2}\frac{N}{N-2}<1,
  \end{equation*}
  so the choice $V(q)=q^{-1}$ in
  \cref{prop:gaussian-nonzero-invariant-existence} gives a Foster--Lyapunov
  estimate. Consequently, every Ces\`aro limit of the laws of $Z^{(N)}$
  considered there assigns no mass to $\{0\}$. To prove
  \cref{prop:gaussian-stable-localization}, we need
  $\nu_{A,N}((0,r])\leq Ce^{-cN}$ for some fixed $r>0$. We therefore set
  $V_N(q)=q^{-\gamma N}$. A bound of the form
  $\int V_N\,\mathrm{d}\nu_{A,N}\leq Ce^{bN}$ gives
  \begin{equation*}
    \nu_{A,N}((0,r])
    \leq
    r^{\gamma N}\int V_N\,\mathrm{d}\nu_{A,N}
    \leq
    C\exp\{N(b+\gamma\log r)\}.
  \end{equation*}
  Choosing $r<e^{-b/\gamma}$ makes the exponent negative.
\end{remark}

Combining the near-zero contraction estimate with the bound on $[R_0,1]$
yields the invariant-law estimate below.
\begin{proposition}
\label{prop:gaussian-stationary-negative-moment}
  Let \cref{ass:gaussian-stable-selected} hold. Then, for all sufficiently large
  $N$, the unique nonzero invariant law $\nu_{A,N}$ of $K_{A,N}$ satisfies
  \begin{equation}\label{eq:gaussian-stationary-negative-moment}
    \int_0^1 q^{-\gamma N}\,\nu_{A,N}(\mathrm{d}q)
    \leq
    \frac{e^{bN}}{1-e^{-\kappa N}},
  \end{equation}
  where $R_0,\gamma,\kappa$ are as in
  \cref{prop:gaussian-near-zero-negative-moment}, and $b$ is the constant in
  \cref{lem:gaussian-outside-negative-moment} for this choice of $R_0,\gamma$.
\end{proposition}

To prove the next proposition, choose $r<R_0$ so that the preceding
negative-moment estimate gives $\nu_{A,N}((0,r])\leq Ce^{-cN}$. Uniform
deterministic entrance from $[r,1]$ into $I_\ast$, together with a fixed-time
concentration estimate, then gives, for some fixed $m$,
$\sup_{q\in[r,1]}K_{A,N}^m(q,I_\ast^c)\leq Ce^{-cN}$. Invariance combines
these two bounds.
\begin{proposition}
\label{prop:gaussian-stable-localization}
  Let \cref{ass:gaussian-stable-selected} hold. Let $\nu_{A,N}$ be the unique
  nonzero invariant law of $K_{A,N}$. Then there are constants $C,c>0$ such that
  \begin{equation}\label{eq:gaussian-stationary-exponential-localization}
    \nu_{A,N}(I_\ast^c)\leq C e^{-cN}
  \end{equation}
  for all sufficiently large $N$. In particular,
  \begin{equation}\label{eq:gaussian-stationary-localization}
    \nu_{A,N}(I_\ast^c)\leq\frac{C}{N}
  \end{equation}
  after enlarging $C$.
\end{proposition}

The proofs of
\cref{prop:gaussian-stationary-negative-moment,prop:gaussian-stable-localization}
are in \cref{app:supercritical-negative-moments,app:supercritical-localization-concentration}.

\subsection{Concentration and coupling at the cutoff scale}
\label{subsec:gaussian-dynamic-tracking-smoothing}

Once two radii lie in $I_\ast$, the following estimate converts their
separation into total variation after one step. We then combine it with
concentration and synchronous coupling.
\begin{lemma}
\label{lem:gaussian-score-smoothing}
  Let $I\Subset(0,1)$. There is a constant $C_I<\infty$ such that, for every
  $q,q'\in I$,
  \begin{equation}\label{eq:gaussian-score-tv-smoothing}
    \norm{K_{A,N}(q,\cdot)-K_{A,N}(q',\cdot)}_{\TV}
    \leq
    C_I\sqrt N\,\abs{q-q'} .
  \end{equation}
\end{lemma}

The stationary chain satisfies the analogous concentration estimate.
\begin{proposition}
\label{prop:gaussian-stationary-concentration}
  Let \cref{ass:gaussian-stable-selected} hold, and let $\nu_{A,N}$ be the unique
  nonzero invariant law of $K_{A,N}$. Then
  \begin{equation}\label{eq:gaussian-stationary-concentration}
    \int_0^1 (q-q_\ast)^2\,\nu_{A,N}(\mathrm{d}q)
    \leq
    \frac{C}{N}.
  \end{equation}
\end{proposition}

The next lemma controls $Z^{(N)}$ around $(q_{t,N})_{t\geq0}$.
\begin{lemma}
\label{lem:gaussian-orbit-concentration}
  Let \cref{ass:gaussian-stable-selected} hold. Then, for every
  $0<C_0<\infty$ there are
  constants $C,c_0>0$ such that, with
  $T_N:=\lfloor C_0\log N\rfloor$ and
  \begin{equation*}
    \tau_\ast^{(N)}
    :=
    \inf\{t\geq s_\ast:Z_t^{(N)}\notin I_\ast\},
  \end{equation*}
  one has, for all sufficiently large $N$,
  \begin{equation}\label{eq:gaussian-orbit-no-exit}
    \Pp\left(\tau_\ast^{(N)}\leq T_N\right)
    \leq
    C T_N\exp(-c_0N),
  \end{equation}
  and
  \begin{equation}\label{eq:gaussian-orbit-concentration}
    \sup_{0\leq t\leq T_N}
    \Ee\left[
      \left(Z_t^{(N)}-q_{t,N}\right)^2
    \right]
    \leq \frac{C}{N} .
  \end{equation}
\end{lemma}

See \cref{app:supercritical-localization-concentration} for the proofs.

We use this lemma with a horizon constant chosen before the cutoff offset.
Fix
\begin{equation}\label{eq:gaussian-uniform-log-horizon-constant}
  C_0>
  \frac{1}{2\abs{\log\mu_A}}.
\end{equation}
For every fixed $c\in\R$, the definitions of $t_N(q_0)$ and $n_N(c)$ give
\begin{equation*}
  \frac{n_N(c)}{\log N}
  \to
  \frac{1}{2\abs{\log\mu_A}}.
\end{equation*}
Hence, by \cref{eq:gaussian-uniform-log-horizon-constant},
\begin{equation}\label{eq:gaussian-uniform-log-horizon}
  0\leq n_N(c)\leq C_0\log N
\end{equation}
for all sufficiently large $N$. Thus the constants supplied by
\cref{lem:gaussian-orbit-concentration} for this choice of $C_0$ are
independent of $c$. Only the lower bound on $N$ may depend on the fixed value
of $c$.

The next lemma controls the synchronous coupling at time
$\pospart{n_N(c)-1}$, one step before $n_N(c)$. The expected separation of the
two radii is $O(\mu_A^cN^{-1/2})$. Applying
\cref{lem:gaussian-score-smoothing} at the next transition then gives a
total-variation error of order $\mu_A^c$.
\begin{lemma}
\label{lem:gaussian-synchronous-cutoff-scale}
  Let \cref{ass:gaussian-stable-selected} hold. Let $q_N\to q_0$ with
  $q_N>0$, let $\nu_{A,N}$ be the unique nonzero invariant law of
  $K_{A,N}$, and let
  $(Z_t^{(N)},\widetilde Z_t^{(N)})$ be the synchronous coupling driven by the
  same Gaussian vectors, with $Z_0^{(N)}=q_N$ and
  $\widetilde Z_0^{(N)}\sim\nu_{A,N}$. Then there is a constant $C<\infty$
  such that, for every $c\geq0$,
  \begin{equation}\label{eq:gaussian-synchronous-cutoff-scale}
    \limsup_{N\to\infty}
    \sqrt N\,
    \Ee\abs{
      Z_{\pospart{n_N(c)-1}}^{(N)}
      -
      \widetilde Z_{\pospart{n_N(c)-1}}^{(N)}
    }
    \leq
    C\mu_A^c .
  \end{equation}
\end{lemma}

See \cref{app:supercritical-cutoff-contraction} for the proof.

\subsection{Proof of the supercritical cutoff theorem}
\label{subsec:gaussian-cutoff-proof}

The lower bound separates the evolving and stationary laws before the
deterministic cutoff scale. The coupling estimate gives the upper bound after
that scale.

\begin{proof}[Proof of \cref{thm:gaussian-process-cutoff}]

  \emph{Step 1: Setup.}
  The theorem assumptions give \cref{ass:gaussian-stable-selected}. Moreover,
  $q_N>0$ for all sufficiently large $N$ because $q_N\to q_0>0$.
  Since $A>1$ and $A_c(N)\downarrow1$, the existence of a nonzero invariant
  law for all sufficiently large $N$ follows from
  \cref{prop:gaussian-nonzero-invariant-existence}. Uniqueness follows from
  \cref{prop:gaussian-unique-nonzero-invariant}. We denote this law by
  $\nu_{A,N}$.

  \emph{Step 2: The lower bound.}
  Fix $c>0$ and put $t_N^-:=n_N(-c)$, with $n_N$ defined in
  \cref{eq:gaussian-integer-cutoff-time}. By
  \cref{eq:gaussian-uniform-log-horizon}, this time is positive and bounded by
  $C_0\log N$ for all large $N$, with the same $C_0$ for every fixed $c$. Let
  $
    d_N:=|{q_{t_N^-,N}-q_\ast}|
  $.
  Since $q_0\neq q_\ast$, the orbit $V_A^t(q_0)$ never crosses $q_\ast$.
  Hence $d_N>0$ for all large $N$. We separate the dynamic law and the
  invariant law by the half-line
  \begin{equation}\label{eq:gaussian-lower-separating-set}
    H_N
    :=
    \left\{
      q\in[0,1]:
      \operatorname{sgn}(q_{t_N^-,N}-q_\ast)(q-q_\ast)
      \geq \frac12 d_N
    \right\}.
  \end{equation}
  Write
  $\varepsilon_N:=\operatorname{sgn}(q_{t_N^-,N}-q_\ast)\in\{-1,1\}$.
  Then
  $
    \varepsilon_N(q_{t_N^-,N}-q_\ast)=d_N
  $.
  If $Z_{t_N^-}^{(N)}\notin H_N$, then, by the definition of $H_N$,
  \begin{equation*}
    \varepsilon_N(Z_{t_N^-}^{(N)}-q_\ast)<\frac12d_N .
  \end{equation*}
  Subtracting this inequality from the preceding identity gives
  \begin{equation*}
    \varepsilon_N(q_{t_N^-,N}-Z_{t_N^-}^{(N)})
    =
    \varepsilon_N(q_{t_N^-,N}-q_\ast)
    -
    \varepsilon_N(Z_{t_N^-}^{(N)}-q_\ast)
    >
    \frac12d_N .
  \end{equation*}
  In particular,
  \begin{equation*}
    \{Z_{t_N^-}^{(N)}\notin H_N\}
    \subseteq
    \left\{
      \abs{Z_{t_N^-}^{(N)}-q_{t_N^-,N}}\geq \frac12 d_N
    \right\}.
  \end{equation*}
  Chebyshev's inequality and \cref{lem:gaussian-orbit-concentration} give
  \begin{align}
    K_{A,N}^{t_N^-}(q_N,H_N^c)
    &=
    \Pp_{q_N}(Z_{t_N^-}^{(N)}\notin H_N)
    \notag\\
    &\leq
    \Pp_{q_N}\left(
      \abs{Z_{t_N^-}^{(N)}-q_{t_N^-,N}}\geq \frac12 d_N
    \right)
    \notag\\
    &\leq
    \frac{4}{d_N^2}
    \Ee\left[
      \left(Z_{t_N^-}^{(N)}-q_{t_N^-,N}\right)^2
    \right]
    \leq
    \frac{C}{N d_N^2}.
    \label{eq:gaussian-lower-dynamic-error}
  \end{align}
  In the last step we used \cref{eq:gaussian-orbit-concentration}, since
  $t_N^-\leq C_0\log N$.
  Similarly, if $q\in H_N$, then $\abs{q-q_\ast}\geq d_N/2$. Chebyshev's
  inequality and \cref{prop:gaussian-stationary-concentration} therefore give
  \begin{equation}\label{eq:gaussian-lower-stationary-error}
    \nu_{A,N}(H_N)
    \leq
    \frac{4}{d_N^2}
    \int_0^1(q-q_\ast)^2\,\nu_{A,N}(\mathrm{d}q)
    \leq
    \frac{C}{N d_N^2}.
  \end{equation}

  Since total variation is the supremum over measurable sets,
  \begin{align*}
    \norm{
      K_{A,N}^{t_N^-}(q_N,\cdot)-\nu_{A,N}
    }_{\TV}
    &\geq
    K_{A,N}^{t_N^-}(q_N,H_N)-\nu_{A,N}(H_N)  \\
    &\geq
    1
    -
    K_{A,N}^{t_N^-}(q_N,H_N^c)
    -
    \nu_{A,N}(H_N).
  \end{align*}
  The locally
  uniform asymptotic \cref{eq:gaussian-deterministic-asymptotic}, applied at
  $t_N^-=n_N(-c)$, gives
  \begin{equation*}
    d_N
    =
    \abs{C(q_N)}\mu_A^{t_N^-}(1+o(1)).
  \end{equation*}
  Moreover, $C(q_N)\to C(q_0)\neq0$ and the definition of the cutoff center
  gives
  $\sqrt N\,\abs{C(q_0)}\mu_A^{t_N(q_0)}=1$. Since
  $t_N^--(t_N(q_0)-c)\in[-1,0]$, the integer-time rounding is uniformly
  bounded. Hence there is $a_0>0$, independent of $c$, such that
  \begin{equation*}
    \liminf_{N\to\infty}\sqrt N\,d_N
    \geq
    a_0\mu_A^{-c}.
  \end{equation*}
  Combining this bound with
  \cref{eq:gaussian-lower-dynamic-error,eq:gaussian-lower-stationary-error}
  yields
  \begin{equation*}
    \liminf_{N\to\infty}
    \norm{
      K_{A,N}^{t_N^-}(q_N,\cdot)-\nu_{A,N}
    }_{\TV}
    \geq
    1-C\mu_A^{2c}.
  \end{equation*}
  Here $C$ is independent of $c$.
  Sending $c\to\infty$ proves \cref{eq:gaussian-process-lower}.

  \emph{Step 3: The upper bound.}
  Fix $c\geq0$ and put
  $t_N^+:=\pospart{n_N(c)-1}$. Let
  $(Z_t^{(N)},\widetilde Z_t^{(N)})$ be the synchronous coupling from
  \cref{lem:gaussian-synchronous-cutoff-scale}, with
  $Z_0^{(N)}=q_N$ and $\widetilde Z_0^{(N)}\sim\nu_{A,N}$. By invariance,
  $\widetilde Z_{t_N^+}^{(N)}\sim\nu_{A,N}$. The Markov property and convexity
  of total variation imply the standard coupling inequality
  \begin{align}
    \norm{
      K_{A,N}^{t_N^++1}(q_N,\cdot)-\nu_{A,N}
    }_{\TV}
    &=
    \norm{
      \Ee\left[
        K_{A,N}(Z_{t_N^+}^{(N)},\cdot)
        -
        K_{A,N}(\widetilde Z_{t_N^+}^{(N)},\cdot)
      \right]
    }_{\TV}
    \notag\\
    &\leq
    \Ee\,
    \norm{
      K_{A,N}(Z_{t_N^+}^{(N)},\cdot)
      -
      K_{A,N}(\widetilde Z_{t_N^+}^{(N)},\cdot)
    }_{\TV}.
    \label{eq:gaussian-upper-smoothing-reduction}
  \end{align}
  Set
  \begin{equation*}
    E_N
    :=
    \left\{
      Z_{t_N^+}^{(N)}\in I_\ast,\,
      \widetilde Z_{t_N^+}^{(N)}\in I_\ast
    \right\}.
  \end{equation*}
  We claim that $\Pp(E_N^c)=o_N(1)$. Indeed,
  \cref{eq:gaussian-uniform-log-horizon} gives
  $t_N^+\leq C_0\log N$ for all large $N$, with $C_0$ independent of $c$, and
  $t_N^+\geq s_\ast$. Thus
  \cref{lem:gaussian-orbit-concentration} gives
  \begin{equation*}
    \Pp(Z_{t_N^+}^{(N)}\notin I_\ast)
    \leq
    C(\log N)e^{-c_0N}
    =
    o_N(1),
  \end{equation*}
  while stationarity of $\widetilde Z^{(N)}$ and
  \cref{prop:gaussian-stable-localization} give
  \begin{equation*}
    \Pp(\widetilde Z_{t_N^+}^{(N)}\notin I_\ast)
    =
    \nu_{A,N}(I_\ast^c)
    \leq
    \frac{C}{N}.
  \end{equation*}
  These two bounds give $\Pp(E_N^c)=o_N(1)$.

  On $E_N$, the one-step total variation estimate
  \cref{lem:gaussian-score-smoothing}, applied with $I=I_\ast$, gives
  \begin{equation*}
    \norm{
      K_{A,N}(Z_{t_N^+}^{(N)},\cdot)
      -
      K_{A,N}(\widetilde Z_{t_N^+}^{(N)},\cdot)
    }_{\TV}
    \leq
    C\sqrt N\,
    \abs{Z_{t_N^+}^{(N)}-\widetilde Z_{t_N^+}^{(N)}}.
  \end{equation*}
  On $E_N^c$ we only use the trivial bound by one. Hence
  \cref{eq:gaussian-upper-smoothing-reduction} yields
  \begin{equation*}
    \norm{
      K_{A,N}^{t_N^++1}(q_N,\cdot)-\nu_{A,N}
    }_{\TV}
    \leq
    C\sqrt N\,
    \Ee\abs{
      Z_{t_N^+}^{(N)}
      -
      \widetilde Z_{t_N^+}^{(N)}
    }
    +
    \Pp(E_N^c)
  \end{equation*}
  For all sufficiently large $N$, one has $n_N(c)\geq1$ and hence
  $t_N^++1=n_N(c)$. Since $t_N^+=\pospart{n_N(c)-1}$,
  \cref{lem:gaussian-synchronous-cutoff-scale} gives
  \begin{equation*}
    \limsup_{N\to\infty}
    \norm{
      K_{A,N}^{n_N(c)}(q_N,\cdot)-\nu_{A,N}
    }_{\TV}
    \leq
    C\mu_A^c.
  \end{equation*}
  Here $C$ is independent of $c$.
  Sending $c\to\infty$ proves
  \cref{eq:gaussian-process-upper}.
\end{proof}

\begin{proof}[Proof of \cref{cor:gaussian-vector-cutoff}]
  The lower bound follows from \cref{prop:gaussian-tv-reduction}, the left inequality in
  \cref{eq:gaussian-tv-bracket}, and \cref{thm:gaussian-process-cutoff}. The
  right inequality in \cref{eq:gaussian-tv-bracket} gives
  \begin{equation*}
    \norm{P_{A,N}^{n_N(c)}(x_N,\cdot)-\pi_{A,N}}_{\TV}
    \leq
    \norm{K_{A,N}^{\pospart{n_N(c)-1}}(q_N,\cdot)-\nu_{A,N}}_{\TV}.
  \end{equation*}
  For all sufficiently large $N$, one has
  $\pospart{n_N(c)-1}=n_N(c-1)$.
  Thus \cref{eq:gaussian-process-upper} proves the upper bound after sending
  $c\to\infty$. By \cref{cor:gaussian-unique-vector-invariant},
  $\pi_{A,N}=\nu_{A,N}J_{A,N}$ is the unique $P_{A,N}$-invariant probability
  that assigns no mass to $\{0\}$.
\end{proof}

\section{The supercritical stationary law at fixed \texorpdfstring{$N$}{N}}
\label{sec:nd-supercritical-stationary}

Fix $N$ and assume $A>A_c(N)$. We study the invariant law $\pi_{A,N}$ with
$\pi_{A,N}(\{0\})=0$. We first derive its stationary
density equation and rewrite it in log-polar coordinates. Near the origin, the
nonlinear update is a quadratic perturbation of its linearization. Gaussian
isotropy refreshes the angular variable in one step, so exponential tilting at
$\beta_{A,N}$ reduces the linearized equation to a scalar renewal equation. We
then show that the nonlinear correction gives a directly Riemann integrable
forcing and apply the renewal theorem to prove the directional and radial
power laws in \cref{thm:nd-power-singularity:intro}. This is the
implicit-renewal method for stationary random recursions \cite{kesten1973},
\cite{goldie1991}, and \cite{buraczewski-damek-mikosch2016}.

We use the Gaussian matrix $\mathsf W_A^{(N)}$ from
\cref{eq:gaussian-weights}, without a time index. The chain $X^{(N)}$ and its
transition kernel $P_{A,N}$ are given in
\cref{eq:unrounded-chain,prop:gaussian-tv-reduction}. Below, $W$ denotes an
independent copy of $\mathsf W_A^{(N)}$. We use $\ell_{A,N}$ from
\cref{eq:gaussian-ell}. Set
\begin{equation*}
  \Kcal_N:=[-1,1]^N .
\end{equation*}
The origin is absorbing, so $\delta_0$ is always invariant. We write
$\pi_{A,N}$ for the nonzero invariant law, which exists under the present
assumption by \cref{prop:gaussian-nonzero-invariant-existence}. Its uniqueness
follows from \cref{cor:gaussian-unique-vector-invariant}.

\subsection{The stationary density equation}
\label{subsec:nd-stationary-density}

For $x\neq0$, the coordinatewise change of variables $z_i=\arctanh(y_i)$
shows that one Gaussian step has a product density. We record the resulting
stationary equation.
\begin{proposition}
\label{prop:nd-stationary-equation}
  Let $\pi_{A,N}$ be an invariant probability for $P_{A,N}$ on $\Kcal_N$ with
  $\pi_{A,N}(\{0\})=0$. Then, for every bounded measurable
  $\varphi:\Kcal_N\to\R$,
  \begin{equation}\label{eq:nd-weak-stationary-equation}
    \int_{\Kcal_N}\varphi(y)\,\pi_{A,N}(\mathrm{d}y)
    =
    \int_{\Kcal_N}
    \int_{\R^N}
    \varphi\bigl(\tanh(A\norm{x}_2 g/\sqrt N)\bigr)\,
    \mathcal G_N(\mathrm{d}g)\,\pi_{A,N}(\mathrm{d}x).
  \end{equation}
  Moreover, $\pi_{A,N}(\partial[-1,1]^N)=0$ and $\pi_{A,N}$ is absolutely
  continuous on $(-1,1)^N$. If $f_{A,N}$ denotes its density, then, for
  $x\neq0$ and $y\in(-1,1)^N$,
  \begin{equation}\label{eq:nd-density-kernel}
    \kappa_{A,N}(x,y)
    =
    \prod_{i=1}^N
    \frac{\sqrt N}{A\norm{x}_2\sqrt{2\pi}}
    \frac{1}{1-y_i^2}
    \exp\left(
      -\frac{N\,\arctanh(y_i)^2}
      {2A^2\norm{x}_2^2}
    \right)
  \end{equation}
  is the transition density from $x$ to $y$, and
  \begin{equation}\label{eq:nd-density-equation}
    f_{A,N}(y)
    =
    \int_{\Kcal_N\setminus\{0\}}
    \kappa_{A,N}(x,y)\,\pi_{A,N}(\mathrm{d}x).
  \end{equation}
  If $\pi_{A,N}(\mathrm{d}x)=f_{A,N}(x)\,\mathrm{d}x$, this becomes
  \begin{equation}\label{eq:nd-density-fixed-point}
    f_{A,N}(y)
    =
    \int_{\Kcal_N\setminus\{0\}}
    \kappa_{A,N}(x,y)f_{A,N}(x)\,\mathrm{d}x .
  \end{equation}
\end{proposition}

\begin{proof}
  \emph{Step 1: The transition density from $x\neq0$.}
  Fix $x\neq0$. By \cref{eq:gaussian-weights}, the rows of
  $\mathsf W_A^{(N)}$ are independent, and the coordinates of
  $\mathsf W_A^{(N)}x$ are independent centered Gaussians with variance
  $A^2\norm{x}_2^2/N$. The change of variables
  $y_i=\tanh z_i$, or $z_i=\arctanh(y_i)$, gives the product density
  \cref{eq:nd-density-kernel}. Moreover, $\tanh z\in(-1,1)$ for every finite
  $z$, so this transition law gives no mass to $\partial[-1,1]^N$.

  \emph{Step 2: The density of the invariant law.}
  The weak equation \cref{eq:nd-weak-stationary-equation} is the invariance
  relation $\pi_{A,N}=\pi_{A,N}P_{A,N}$. Since $\pi_{A,N}(\{0\})=0$, the
  boundary property from Step~1 and invariance give
  $\pi_{A,N}(\partial[-1,1]^N)=0$. Integrating the transition density
  \cref{eq:nd-density-kernel} against $\pi_{A,N}$ gives
  \cref{eq:nd-density-equation} and proves absolute continuity on
  $(-1,1)^N$. Substituting
  $\pi_{A,N}(\mathrm{d}x)=f_{A,N}(x)\,\mathrm{d}x$ gives
  \cref{eq:nd-density-fixed-point}.
\end{proof}

\subsection{Log-polar coordinates}
\label{subsec:nd-log-polar-coordinates}

To analyze the invariant law near $0$, we use log-polar coordinates. Let
$\mathbb S^{N-1}$ be the Euclidean unit sphere and write
$\hat z=z/\norm{z}_2$ for $z\neq0$. For $x\neq0$, set
\begin{equation*}
  r=\norm{x}_2,\qquad
  \theta=\frac{x}{\norm{x}_2},\qquad
  y=-\log r .
\end{equation*}
Since $r=e^{-y}$, the limit $r\downarrow0$ corresponds to $y\to\infty$, and
the factor $r^\beta$ becomes $e^{-\beta y}$.

Let $\widetilde{\pi}_{A,N}$ be the image of $\pi_{A,N}$ under
$x\mapsto(Y,\Theta)=(-\log\norm{x}_2,x/\norm{x}_2)$. Since
$0<\norm{x}_2\leq\sqrt N$ on $\Kcal_N\setminus\{0\}$, the variable $Y$ takes
values in $[-\frac12\log N,\infty)$. For $r>0$, $\theta\in\mathbb S^{N-1}$,
and $W\theta\neq0$, define
\begin{align}
  \eta(r,\theta,W)
  &:=
  \log
  \frac{r\norm{W\theta}_2}{\norm{\tanh(rW\theta)}_2},
  \label{eq:nd-eta-definition}
  \\
  \Theta_+(r,\theta,W)
  &:=
  \frac{\tanh(rW\theta)}
  {\norm{\tanh(rW\theta)}_2},
  \qquad
  \Theta_0(\theta,W):=\widehat{W\theta}.
  \label{eq:nd-angular-definitions}
\end{align}
Here $\eta$ is the correction to the logarithmic radius caused by the
nonlinearity, $\Theta_+$ is the direction after the nonlinear update, and
$\Theta_0$ is the direction after the corresponding linear update.
Since $W\theta\neq0$ almost surely, the stationary equation \cref{eq:nd-weak-stationary-equation} becomes in these coordinates
\begin{equation}\label{eq:nd-log-polar-stationary}
  \begin{split}
  \int \varphi(y,\theta)\,\widetilde{\pi}_{A,N}(\mathrm{d}y,\mathrm{d}\theta)
  =
  \int
  \Ee\Bigl[
  \varphi\Bigl(
    y-\log\norm{W\theta}_2+\eta(e^{-y},\theta,W),\,
    \Theta_+(e^{-y},\theta,W)
  \Bigr)
  \Bigr]\,
  \widetilde{\pi}_{A,N}(\mathrm{d}y,\mathrm{d}\theta)
  \end{split}
\end{equation}
for every bounded measurable $\varphi$.

\subsection{The linear renewal problem}
\label{subsec:nd-linear-renewal}

In \cref{eq:nd-log-polar-stationary}, set $\eta=0$ and replace $\Theta_+$ by
$\Theta_0$. This gives the linearized chain
\begin{equation}\label{eq:nd-linear-markov-additive}
  (Y,\Theta)
  \mapsto
  \left(
    Y+\Delta,\Theta_0
  \right),
  \qquad
  \Delta=-\log\norm{\mathsf W_A^{(N)}\Theta}_2 .
\end{equation}
For a general matrix ensemble, $\Theta$ is a Markov chain and $\Delta$ is an
additive functional of it \cite{bougerol-lacroix1985}. Renewal theory for
this chain is developed in \cite{kesten1973} and further discussed in
\cite{buraczewski-damek-mikosch2016}. For Gaussian $W$, the law of $W\Theta$
does not depend on $\Theta$. The angle is therefore refreshed at each step.

Let $\bar\sigma_N$ denote normalized surface measure on $\mathbb S^{N-1}$. We
equip $C(\mathbb S^{N-1})$ with the supremum norm. Define the linear log-polar
transition kernel by
\begin{equation*}
  P_{A,N}^{\mathrm{lin}}(\theta,\cdot,\cdot)
  :=
  \operatorname{Law}\left(
    -\log\norm{\mathsf W_A^{(N)}\theta}_2,\,
    \widehat{\mathsf W_A^{(N)}\theta}
  \right).
\end{equation*}
For $0\leq\beta<N$, define the transfer operator on
$C(\mathbb S^{N-1})$ by
\begin{equation}\label{eq:nd-transfer-operator}
  \Lcal_{A,N,\beta}\varphi(\theta)
  :=
  \int e^{\beta z}\varphi(\theta')\,
  P_{A,N}^{\mathrm{lin}}(\theta,\mathrm{d}z,\mathrm{d}\theta')
  =
  \Ee\left[
    \norm{\mathsf W_A^{(N)}\theta}_2^{-\beta}
    \varphi\bigl(\widehat{\mathsf W_A^{(N)}\theta}\bigr)
  \right].
\end{equation}
The Gaussian transfer operator can be computed explicitly.
\begin{lemma}\label{lem:nd-gaussian-transfer}
  For every $0\leq\beta<N$, the operator $\Lcal_{A,N,\beta}$ is bounded on
  $C(\mathbb S^{N-1})$ and satisfies
  \begin{equation}\label{eq:nd-transfer-rank-one}
    \Lcal_{A,N,\beta}\varphi(\theta)
    =
    \mathcal M_{A,N}(\beta)
    \int_{\mathbb S^{N-1}}\varphi\,\mathrm{d}\bar\sigma_N,
    \qquad
    \theta\in\mathbb S^{N-1},
  \end{equation}
  where
  \begin{equation}\label{eq:nd-gaussian-negative-moment}
    \mathcal M_{A,N}(\beta)
    :=
    \Ee\norm{\mathsf W_A^{(N)}\theta}_2^{-\beta}
    =
    \left(\frac{\sqrt N}{A}\right)^\beta
    2^{-\beta/2}
    \frac{\Gamma((N-\beta)/2)}{\Gamma(N/2)}.
  \end{equation}
  Consequently, its spectral radius $r_{A,N}(\beta)$ is given by
  \begin{equation}\label{eq:nd-pressure-scaling}
    r_{A,N}(\beta)
    =
    \mathcal M_{A,N}(\beta)
    =
    A^{-\beta}\mathcal M_{1,N}(\beta).
  \end{equation}
\end{lemma}

\begin{proof}
  Fix $\theta\in\mathbb S^{N-1}$. By
  \cref{eq:gaussian-weights}, the vector
  $\mathsf W_A^{(N)}\theta$ has law $(A/\sqrt N)G$, where
  $G\sim\mathcal N(0,I_N)$. Hence
  \begin{equation*}
    \Ee\norm{\mathsf W_A^{(N)}\theta}_2^{-\beta}
    =
    \left(\frac{\sqrt N}{A}\right)^\beta
    \Ee\norm{G}_2^{-\beta}.
  \end{equation*}
  The last expectation is finite exactly for $\beta<N$, since the standard
  Gaussian density is bounded and the singularity of $\norm{z}_2^{-\beta}$ at
  the origin is integrable in dimension $N$ precisely in this range. The
  displayed Gamma-function formula is the standard moment formula for
  $\chi_N=\norm{G}_2$.

  The same isotropy also shows that
  $\widehat{\mathsf W_A^{(N)}\theta}$ has law $\bar\sigma_N$ and is independent
  of $\norm{\mathsf W_A^{(N)}\theta}_2$. Therefore
  \cref{eq:nd-transfer-rank-one} holds. A rank-one positive operator of the
  form $\mathcal M_{A,N}(\beta)\bar\sigma_N(\varphi)\mathbf 1$ has spectral
  radius $\mathcal M_{A,N}(\beta)$.
\end{proof}

The Cramér exponent is the unique positive solution of the criticality equation
for this operator.
\begin{lemma}\label{lem:nd-gaussian-cramer-exponent}
  Assume $A>A_c(N)$, with $A_c(N)$ defined in
  \cref{eq:fixed-width-critical}. Then there is a unique
  $\beta_{A,N}\in(0,N)$ satisfying
  \begin{equation}\label{eq:nd-beta-equation}
    r_{A,N}(\beta_{A,N})=1.
  \end{equation}
  Equivalently, $\beta_{A,N}$ is the unique positive solution of
  \begin{equation}\label{eq:nd-beta-equation-explicit}
    A^{-\beta}
    N^{\beta/2}
    2^{-\beta/2}
    \frac{\Gamma((N-\beta)/2)}{\Gamma(N/2)}
    =
    1,
    \qquad 0<\beta<N.
  \end{equation}
\end{lemma}

\begin{proof}
  Recall $\ell_{A,N}$ from \cref{eq:gaussian-ell}, and set
  \begin{equation*}
    F(\beta):=\log\Ee \ell_{A,N}^{-\beta}.
  \end{equation*}
  By \cref{eq:fixed-width-critical}, the condition $A>A_c(N)$ is equivalent to
  $\Ee\log \ell_{A,N}>0$. Thus $F(0)=0$ and
  $F'(0)=-\Ee\log \ell_{A,N}<0$.
  Moreover $F$ is strictly convex by H\"older's inequality and the
  nondegeneracy of $\ell_{A,N}$, and
  $F(\beta)\to\infty$ as $\beta\uparrow N$, because
  $\Ee\chi_N^{-\beta}\uparrow\infty$. Therefore $F(\beta)=0$ has exactly one
  solution in $(0,N)$. The explicit equation is just
  \cref{eq:nd-gaussian-negative-moment}.
\end{proof}

By \cref{lem:nd-gaussian-transfer}, tilting at $\beta_{A,N}$ changes only the
law of the increment. The updated direction remains independent of the
increment and has law $\bar\sigma_N$. The renewal equation therefore reduces
to a scalar convolution in $y$. The following lemma applies Feller's whole-line
renewal theorem
\cite{feller1971vol2} to this convolution and recovers the angular limit as a
multiple of $\bar\sigma_N$. A related formulation is given in
\cite{gut2009stopped}.
\begin{lemma}\label{lem:nd-gaussian-renewal}
  Assume $A>A_c(N)$ and let $\beta_{A,N}$ be the exponent from
  \cref{lem:nd-gaussian-cramer-exponent}. Using $\ell_{A,N}$ from
  \cref{eq:gaussian-ell}, define
  \begin{equation}\label{eq:nd-tilted-increment-law}
    \int g(z)\,\widehat\mu_{A,N}(\mathrm{d}z)
    :=
    \Ee\left[
      \ell_{A,N}^{-\beta_{A,N}}
      g(-\log \ell_{A,N})
    \right].
  \end{equation}
  Then $\widehat\mu_{A,N}$ is a nonlattice probability measure. Its drift
  satisfies
  \begin{equation*}
    \widehat m_{A,N}
    :=
    \int_\R z\,\widehat\mu_{A,N}(\mathrm{d}z)
    \in(0,\infty).
  \end{equation*}
  Set
  $\widehat P_{A,N}(\theta,\mathrm{d}z,\mathrm{d}\theta')
  :=\bar\sigma_N(\mathrm{d}\theta')\,\widehat\mu_{A,N}(\mathrm{d}z)$.

  Let $(\Hcal_y)_{y\in\R}$ be a locally bounded family of finite signed
  measures on $\mathbb S^{N-1}$ satisfying the boundary conditions
  \begin{equation}\label{eq:nd-renewal-boundary-conditions}
      \lim_{y\to-\infty}\norm{\Hcal_y}_{\mathrm{TV}}=0, \qquad
      \lim_{y\to\infty}
      e^{-\beta_{A,N}y}\norm{\Hcal_y}_{\mathrm{TV}}=0.
  \end{equation}
  Assume also that the forcing $(\Psi_y)_{y\in\R}$ satisfies
  \begin{equation}\label{eq:nd-dri-definition}
    \sum_{k\in\Z}
    \sup_{y\in[k,k+1]}\norm{\Psi_y}_{\mathrm{TV}}
    <\infty .
  \end{equation}
  For every continuous $\varphi$, assume that
  $y\mapsto\Psi_y(\varphi)$ is continuous and that
  \begin{equation}\label{eq:nd-abstract-renewal-equation}
    \Hcal_y(\varphi)
    =
    \int_{\mathbb S^{N-1}}
    \int
    \varphi(\theta')\,
    \widehat P_{A,N}(\theta,\mathrm{d}z,\mathrm{d}\theta')\,
    \Hcal_{y-z}(\mathrm{d}\theta)
    +
    \Psi_y(\varphi).
  \end{equation}
  Then
  \begin{equation}\label{eq:nd-abstract-renewal-limit}
    \Hcal_y(\varphi)
    \to
    \frac{1}{\widehat m_{A,N}}
    \left(\int_\R \Psi_t(\mathbf 1)\,\mathrm{d}t\right)
    \int_{\mathbb S^{N-1}}\varphi\,\mathrm{d}\bar\sigma_N,
    \qquad
    y\to\infty,
  \end{equation}
  where $\mathbf 1$ denotes the constant function on $\mathbb S^{N-1}$.
\end{lemma}

\begin{proof}
  \emph{Step 1: Tilted law.}
  The identity $r_{A,N}(\beta_{A,N})=1$ says precisely that
  $\widehat\mu_{A,N}$ has total mass one. Its law has a density, since
  $\chi_N$ has a density on $(0,\infty)$, and hence it is nonlattice. The
  drift is finite because $\Ee\chi_N^{-\beta}\abs{\log\chi_N}<\infty$ for
  $\beta<N$. It is positive since
  \begin{equation*}
    \widehat m_{A,N}
    =
    \frac{\mathrm{d}}{\mathrm{d}\beta}
    \log\mathcal M_{A,N}(\beta)\bigg|_{\beta=\beta_{A,N}},
  \end{equation*}
  and this derivative is positive at the unique positive zero
  $\beta_{A,N}$ of the strictly convex function
  $\log\mathcal M_{A,N}(\beta)$.

  \emph{Step 2: Scalar renewal representation.}
  Put $h_y:=\Hcal_y(\mathbf 1)$. Taking $\varphi=\mathbf 1$ in
  \cref{eq:nd-abstract-renewal-equation} gives the scalar renewal equation
  \begin{equation*}
    h_y
    =
    \int_\R h_{y-z}\,\widehat\mu_{A,N}(\mathrm{d}z)
    +
    \Psi_y(\mathbf 1).
  \end{equation*}
  Let $S_n=Z_1+\cdots+Z_n$, where the variables $Z_i$ are independent with law
  $\widehat\mu_{A,N}$. Iterating the last display gives
  \begin{equation*}
    h_y
    =
    \Ee h_{y-S_n}
    +
    \sum_{k=0}^{n-1}
    \Ee\Psi_{y-S_k}(\mathbf 1).
  \end{equation*}
  We claim that, for this fixed $y$,
  $\Ee h_{y-S_n}\to0$ as $n\to\infty$. Since
  $\widehat m_{A,N}>0$, we have $S_n\to\infty$ almost surely. Fix $L<\infty$.
  On the half-line $(-\infty,L]$, local boundedness together with
  the first condition in \cref{eq:nd-renewal-boundary-conditions} gives
  $\sup_{u\leq L}\norm{\Hcal_u}_{\mathrm{TV}}<\infty$. Since
  $y-S_n\to-\infty$ almost surely, dominated convergence gives
  \begin{equation*}
    \Ee\left[
      \norm{\Hcal_{y-S_n}}_{\mathrm{TV}}
      \one_{\{y-S_n\leq L\}}
    \right]
    \to0 .
  \end{equation*}
  On the complementary event,
  the second condition in \cref{eq:nd-renewal-boundary-conditions} gives, for
  $L$ sufficiently large,
  \begin{equation*}
    \Ee\left[
      \norm{\Hcal_{y-S_n}}_{\mathrm{TV}}
      \one_{\{y-S_n>L\}}
    \right]
    \leq
    \varepsilon e^{\beta_{A,N}y}
    \Ee e^{-\beta_{A,N}S_n}
    =
    \varepsilon e^{\beta_{A,N}y}.
  \end{equation*}
  The equality uses the defining tilt \cref{eq:nd-tilted-increment-law}.
  Adding the two contributions and then letting $\varepsilon\downarrow0$ gives
  \begin{equation*}
    \Ee\norm{\Hcal_{y-S_n}}_{\mathrm{TV}}\to0 .
  \end{equation*}
  Since $\norm{\mathbf 1}_\infty=1$, this total variation bound also controls
  $\Ee h_{y-S_n}$:
  \begin{equation*}
    \abs{\Ee h_{y-S_n}}
    =
    \abs{\Ee\,\Hcal_{y-S_n}(\mathbf 1)}
    \leq
    \Ee\norm{\Hcal_{y-S_n}}_{\mathrm{TV}}
    \to0 .
  \end{equation*}
  Therefore, after first fixing $y$ and then letting $n\to\infty$,
  \begin{equation*}
    h_y
    =
    \sum_{k=0}^{\infty}
    \Ee\Psi_{y-S_k}(\mathbf 1).
  \end{equation*}
  \emph{Step 3: Scalar renewal limit.}
  Let $y\to\infty$ and decompose the signed forcing as
  $\Psi_y(\mathbf 1)=\Psi_y(\mathbf 1)^+-\Psi_y(\mathbf 1)^-$. Both parts are
  nonnegative and directly Riemann integrable by
  \cref{eq:nd-dri-definition}. The increment law $\widehat\mu_{A,N}$ is
  nonarithmetic with mean
  $\widehat m_{A,N}\in(0,\infty)$ and charges both half-lines, so the renewal
  measure $U:=\sum_{k\geq0}\mathcal L(S_k)$ obeys the general renewal theorem on
  the whole line: $U\{I+y\}\to\abs I/\widehat m_{A,N}$ as $y\to+\infty$ and
  $U\{I+y\}\to0$ as $y\to-\infty$ for every finite interval $I$
  \cite[Sec.~XI.9, (9.2)--(9.3)]{feller1971vol2}. Applying the
  upper- and lower-step-function argument of
  \cite[Sec.~XI.1, (1.10)--(1.17)]{feller1971vol2} to these
  whole-line interval limits extends the conclusion to any directly Riemann
  integrable integrand in the sense of \cref{eq:nd-dri-definition}. Apply it
  separately to
  $\Psi_\cdot(\mathbf 1)^+$ and $\Psi_\cdot(\mathbf 1)^-$. The identity
  $\int_\R z(y-s)\,U(\mathrm{d}s)=\sum_{k\geq0}\Ee\,z(y-S_k)$ then gives
  \begin{equation*}
    \sum_{k=0}^\infty\Ee\,\Psi_{y-S_k}(\mathbf 1)^\pm
    \to
    \frac{1}{\widehat m_{A,N}}\int_\R \Psi_t(\mathbf 1)^\pm\,\mathrm{d}t,
    \qquad y\to\infty,
  \end{equation*}
  and subtracting the two limits yields
  \begin{equation*}
    h_y
    \to
    \frac{1}{\widehat m_{A,N}}\int_\R \Psi_t(\mathbf 1)\,\mathrm{d}t .
  \end{equation*}
  \emph{Step 4: Angular reconstruction.}
  For a general continuous $\varphi$, the product form of
  $\widehat P_{A,N}$ gives
  \begin{equation*}
    \Hcal_y(\varphi)
    =
    \left(\int_{\mathbb S^{N-1}}\varphi\,\mathrm{d}\bar\sigma_N\right)
    \int_\R h_{y-z}\,\widehat\mu_{A,N}(\mathrm{d}z)
    +
    \Psi_y(\varphi).
  \end{equation*}
  Using the scalar renewal equation to replace the convolution by
  $h_y-\Psi_y(\mathbf 1)$, and using direct Riemann integrability to get
  $\Psi_y(\varphi)\to0$ and $\Psi_y(\mathbf 1)\to0$, yields
  \cref{eq:nd-abstract-renewal-limit}.
\end{proof}

\subsection{The nonlinear renewal equation}
\label{subsec:nd-nonlinear-renewal}

We compare \cref{eq:nd-log-polar-stationary} with its linearization
\cref{eq:nd-linear-markov-additive}. After exponential tilting, their
difference is the forcing $\Psi^\pi$ defined below. Its integrability uses
exponential moments of order less than $\beta_{A,N}$.
\begin{lemma}
\label{lem:nd-subcritical-exponential-moments}
  Assume $A>A_c(N)$ and let $\pi_{A,N}$ be a nonzero invariant law as in
  \cref{prop:nd-stationary-equation}. Let $(Y,\Theta)$ have law
  $\widetilde{\pi}_{A,N}$. Then, for every $0\leq\alpha<\beta_{A,N}$,
  \begin{equation}\label{eq:nd-subcritical-exponential-moment}
    \Ee e^{\alpha Y}<\infty .
  \end{equation}
\end{lemma}

\begin{proof}
  The case $\alpha=0$ is trivial, so assume $0<\alpha<\beta_{A,N}$. Let
  $\nu$ be the image of $\pi_{A,N}$ under the squared-radius map
  $r_N(x)=N^{-1}\norm{x}_2^2$, defined before
  \cref{prop:gaussian-tv-reduction}. By
  \cref{prop:gaussian-tv-reduction}, $\nu$ is a nonzero invariant law for
  $K_{A,N}$.

  \emph{Step 1: Contraction of $q^{-\alpha/2}$ near zero.}
  Since $\alpha<\beta_{A,N}$ and $\beta_{A,N}$ is the first positive zero of
  $\log\mathcal M_{A,N}$, we have
  \begin{equation*}
    \Ee\left[
      \left(\frac A{\sqrt N}\chi_N\right)^{-\alpha}
    \right]
    <1 .
  \end{equation*}
  Put $p=\alpha/2$ and $V(q)=q^{-p}$. Since
  \begin{equation*}
    \frac{F_{A,N}(q,G)}q
    =
    \frac1N\sum_{i=1}^N
    A^2G_i^2
    \left(
      \frac{\tanh(A\sqrt q\,G_i)}{A\sqrt q\,G_i}
    \right)^2
    \xrightarrow[q\downarrow0]{}
    \frac{A^2}{N}\chi_N^2
    \qquad\mbox{a.s.},
  \end{equation*}
  the fixed-$N$ version of the truncation argument in
  \cref{prop:gaussian-near-zero-negative-moment} gives
  \begin{equation*}
    \Ee\left[
      \left(\frac{F_{A,N}(q,G)}q\right)^{-p}
    \right]
    \longrightarrow
    \Ee\left[
      \left(\frac A{\sqrt N}\chi_N\right)^{-\alpha}
    \right]
    <1 .
  \end{equation*}
  Hence we may choose $R_0>0$ and $a<1$ such that, for $0<q\leq R_0$,
  \begin{equation*}
    K_{A,N}V(q)
    =
    q^{-p}
    \Ee\left[
      \left(\frac{F_{A,N}(q,G)}q\right)^{-p}
      \right]
      \leq aV(q).
  \end{equation*}

  \emph{Step 2: A finite negative moment of $\nu$.}
  Since $\alpha<\beta_{A,N}<N$, the fixed-$N$ Gaussian small-ball estimate used
  above gives
  \begin{equation*}
    B:=\sup_{R_0\leq q\leq1} K_{A,N}V(q)<\infty .
  \end{equation*}
  Therefore
  \begin{equation*}
    K_{A,N}V(q)\leq aV(q)+B,\qquad q\in(0,1].
  \end{equation*}
  The Cesàro construction in the proof of
  \cref{prop:gaussian-nonzero-invariant-existence} therefore produces a nonzero
  invariant law $\overline\nu$ with $\int q^{-p}\,\overline\nu(\mathrm{d}q)<\infty$.
  By the uniqueness of the nonzero invariant law in
  \cref{prop:gaussian-unique-nonzero-invariant}, $\overline\nu=\nu$.

  \emph{Step 3: The exponential moment of $Y$.}
  If $X\sim\pi_{A,N}$, then $r_N(X)\sim\nu$ and
  \begin{equation*}
    \Ee e^{\alpha Y}
    =
    N^{-\alpha/2}
    \int_{(0,1]}q^{-\alpha/2}\,\nu(\mathrm{d}q)
    <\infty .
  \end{equation*}
  This proves \cref{eq:nd-subcritical-exponential-moment}.
\end{proof}

The next lemma bounds the difference between the nonlinear and linearized
updates.
\begin{lemma}
\label{lem:nd-gaussian-polar-perturbation}
  Assume $A>A_c(N)$ and let $\beta_{A,N}\in(0,N)$ be the exponent from
  \cref{lem:nd-gaussian-cramer-exponent}. Fix $\alpha\in(0,2)$ and set
  $h(t):=\min\{1,e^{-Nt}\}$. Then there exists $C<\infty$ such that
  \begin{equation}\label{eq:nd-polar-envelope-dri}
    \sum_{k\in\Z}
    \sup_{t\in[k,k+1]} e^{\beta_{A,N}t}h(t)<\infty
  \end{equation}
  and, for every $0<r\leq\sqrt N$ and every $t\in\R$,
  \begin{equation}\label{eq:nd-polar-kernel-error}
    \left\|
      \Ee\left[
        \one_{\{\norm{\tanh(rU)}_2<re^{-t}\}}
        \delta_{\widehat{\tanh(rU)}}
      \right]
      -
      \Ee\left[
        \one_{\{\norm U_2<e^{-t}\}}\delta_{\hat U}
      \right]
    \right\|_{\mathrm{TV}}
    \leq
    C r^\alpha h(t),
  \end{equation}
  where $\delta_z$ denotes the Dirac mass at $z$, $U=(A/\sqrt N)G$, and
  $G\sim\mathcal N(0,I_N)$.
\end{lemma}

\begin{proof}
  \emph{Step 1: Reduction to the densities of $U$ and $T_rU$.}
  Put
  \begin{equation*}
    T_r u:=r^{-1}\tanh(ru),
    \qquad u\in\R^N .
  \end{equation*}
  Then
  $\widehat{\tanh(rU)}=\widehat{T_rU}$ and
  $\{\norm{\tanh(rU)}_2<re^{-t}\}=\{\norm{T_rU}_2<e^{-t}\}$. Let
  $p_0$ be the density of $U$. The map $T_r$ is a diffeomorphism from
  $\R^N$ onto $(-r^{-1},r^{-1})^N$, and the density of $T_rU$ is
  \begin{equation}\label{eq:nd-nonlinear-linearized-density}
    p_r(v)
    =
    p_0\left(r^{-1}\arctanh(rv)\right)
    \prod_{i=1}^N(1-r^2v_i^2)^{-1}
    \one_{\{\abs{rv_i}<1,\ 1\leq i\leq N\}} .
  \end{equation}
  Since angular projection and restriction to the ball
  $\{\norm v_2<e^{-t}\}$ cannot increase total variation, the left side of
  \cref{eq:nd-polar-kernel-error} is at most
  \begin{equation}\label{eq:nd-density-error-reduction}
    \int_{\{\norm v_2<e^{-t}\}}\abs{p_r(v)-p_0(v)}\,\mathrm{d}v .
  \end{equation}

  It remains to prove that, for every $\rho>0$ and $0<r\leq\sqrt N$,
  \begin{equation}\label{eq:nd-density-error-ball}
    \int_{\{\norm v_2<\rho\}}\abs{p_r(v)-p_0(v)}\,\mathrm{d}v
    \leq
    C r^\alpha \min\{1,\rho^N\}.
  \end{equation}

  \emph{Step 2: The density estimate for $1/4\leq r\leq\sqrt N$.}
  The densities $p_r$ are uniformly bounded for
  $1/4\leq r\leq\sqrt N$. Indeed, by
  \cref{eq:nd-nonlinear-linearized-density}, it is enough to note that the
  one-dimensional function
  \begin{equation*}
    s\mapsto
    \exp\{-c\,\arctanh(s)^2\}(1-s^2)^{-1},
    \qquad \abs{s}<1,
  \end{equation*}
  is bounded for every $c>0$. In the density $p_r$, the corresponding value is
  $c=N/(2A^2r^2)$, which ranges over a compact subset of $(0,\infty)$. Hence
  \begin{equation*}
    \int_{\{\norm v_2<\rho\}}\abs{p_r(v)-p_0(v)}\,\mathrm{d}v
    \leq C\rho^N .
  \end{equation*}
  The same integral is at most $2$ because $p_r$ and $p_0$ are probability
  densities. Since $r^\alpha\geq4^{-\alpha}$ on this range, these two bounds
  give \cref{eq:nd-density-error-ball}.

  \emph{Step 3: The density estimate for $0<r\leq1/4$.}
  We first prove the pointwise bound
  \begin{equation*}
    \abs{p_r(v)-p_0(v)}
    \leq
    C r^2(1+\norm v_2^4)e^{-c\norm v_2^2},
    \qquad \norm v_2\leq(2r)^{-1}.
  \end{equation*}
  First suppose
  $\norm v_2\leq r^{-1/2}$. Then
  $\abs{rv_i}\leq\sqrt r\leq1/2$, and hence
  \begin{equation*}
    r^{-1}\arctanh(rv_i)
    =
    v_i+O(r^2\abs{v_i}^3),
    \qquad
    -\log(1-r^2v_i^2)=O(r^2v_i^2),
  \end{equation*}
  uniformly in $i$. Expanding the logarithm of the density ratio
  $p_r(v)/p_0(v)$ gives
  \begin{equation*}
    \abs{\log(p_r(v)/p_0(v))}
    \leq
    C r^2(1+\norm v_2^4),
  \end{equation*}
  whose right side is bounded on this region. Since
  $p_r=p_0\exp\{\log(p_r/p_0)\}$ and $p_0(v)\leq Ce^{-c\norm v_2^2}$, the
  inequality $\abs{e^x-1}\leq C\abs x$ for bounded $x$ gives the pointwise
  bound above, after decreasing $c$. On the remaining
  annulus $r^{-1/2}<\norm v_2\leq(2r)^{-1}$, both $p_0(v)$ and $p_r(v)$ are
  bounded by $Ce^{-c\norm v_2^2}$, while
  $r^2(1+\norm v_2^4)\geq1$. Thus the pointwise bound also holds on this
  annulus, after decreasing $c$.
  Integrating over
  $\{\norm v_2<\rho,\ \norm v_2\leq(2r)^{-1}\}$, and using that
  $\int_{\{\norm v_2<\rho\}}(1+\norm v_2^4)e^{-c\norm v_2^2}\,\mathrm{d}v
  \leq C\min\{1,\rho^N\}$, since the integrand is bounded near the origin and
  integrable over all of $\R^N$, this contributes at most
  $Cr^2\min\{1,\rho^{N}\}\leq Cr^\alpha\min\{1,\rho^N\}$.
  The part of $\{\norm v_2<\rho\}$ where
  $\norm v_2>(2r)^{-1}$ is empty when $\rho<1$. When $\rho\geq1$, its
  contribution is at most
  \begin{align*}
    &\int_{\{\norm v_2>(2r)^{-1}\}}
      \bigl(p_0(v)+p_r(v)\bigr)\,\mathrm{d}v
    \\
    &\qquad\leq
    \Pp\bigl(\norm U_2>(2r)^{-1}\bigr)
    +\Pp\bigl(\norm{T_rU}_2>(2r)^{-1}\bigr)
    \leq Ce^{-c/r^2},
  \end{align*}
  where the last inequality uses $\norm{T_rU}_2\leq\norm U_2$ and the
  Gaussian tail bound. Since $e^{-c/r^2}\leq Cr^\alpha$, this proves
  \cref{eq:nd-density-error-ball}.

  \emph{Step 4: The kernel bound and summability.}
  Taking $\rho=e^{-t}$ in \cref{eq:nd-density-error-ball}, and using
  $\min\{1,\rho^N\}=h(t)$, proves
  \cref{eq:nd-polar-kernel-error}. For $t<0$,
  $e^{\beta_{A,N}t}h(t)=e^{\beta_{A,N}t}$, while for $t\geq0$ it equals
  $e^{-(N-\beta_{A,N})t}$. Since $0<\beta_{A,N}<N$, both tails are summable,
  which proves \cref{eq:nd-polar-envelope-dri}.
\end{proof}

\begin{definition}[The forcing $\Psi^\pi$]
\label{def:nd-nonlinear-forcing}
  Fix an invariant law $\pi_{A,N}$ as in
  \cref{prop:nd-stationary-equation}, let
  $(Y,\Theta)\sim\widetilde{\pi}_{A,N}$, and let $W$ be an independent copy of
  $\mathsf W_A^{(N)}$. Put
  \begin{equation*}
    \Delta=-\log\norm{W\Theta}_2,\qquad
    \Theta_0=\widehat{W\Theta},\qquad
    \Theta_+=\Theta_+(e^{-Y},\Theta,W),\qquad
    \eta=\eta(e^{-Y},\Theta,W).
  \end{equation*}
  For each continuous test function $\varphi$ on $\mathbb S^{N-1}$, define
  \begin{equation}\label{eq:nd-forcing-definition}
    \Psi_y^\pi(\varphi)
    :=
    e^{\beta_{A,N}y}
    \Ee\Bigl[
      \varphi(\Theta_+)
      \one_{\{Y+\Delta+\eta>y\}}
      -
      \varphi(\Theta_0)
      \one_{\{Y+\Delta>y\}}
    \Bigr].
  \end{equation}
  This bounded linear functional on $C(\mathbb S^{N-1})$ defines a finite
  signed measure, also denoted by $\Psi_y^\pi$.
  Set
  \begin{equation}\label{eq:nd-forcing-constant}
    c_{\Psi^\pi}
    :=
    \frac{1}{\widehat m_{A,N}}
    \int_\R \Psi_t^\pi(\mathbf 1)\,\mathrm{d}t ,
  \end{equation}
  where $\widehat m_{A,N}:=\int_\R z\,\widehat\mu_{A,N}(\mathrm{d}z)$ is the
  tilted drift.
  We say that the nonlinear renewal forcing is admissible if
  $(\Psi_y^\pi)_{y\in\R}$ satisfies
  \cref{eq:nd-dri-definition} and $c_{\Psi^\pi}>0$.
\end{definition}

We next verify that this forcing is admissible.
\begin{proposition}
\label{prop:nd-forcing-admissibility}
  Assume $A>A_c(N)$ and let $\pi_{A,N}$ be a nonzero invariant law as in
  \cref{prop:nd-stationary-equation}. Then the forcing defined in
  \cref{def:nd-nonlinear-forcing} is admissible.
\end{proposition}

\begin{proof}
  \emph{Step 1: A summable bound for $\norm{\Psi_y^\pi}_{\mathrm{TV}}$.}
  Choose $\alpha\in(0,\min\{2,\beta_{A,N}\})$. Conditional on
  $(Y,\Theta)$, Gaussian isotropy shows that $W\Theta$ has the same law as
  $U:=(A/\sqrt N)G$. This law does not depend on $(Y,\Theta)$. Set
  $r=e^{-Y}$ and $t=y-Y$. Since $Y\geq-\frac12\log N$, we have
  $0<r\leq\sqrt N$.
  Conditional on $(Y,\Theta)$, the signed measure in
  \cref{eq:nd-forcing-definition} is therefore the difference estimated in
  \cref{lem:nd-gaussian-polar-perturbation}. It follows that
  \begin{equation}\label{eq:nd-forcing-envelope-bound}
    \norm{\Psi_y^\pi}_{\mathrm{TV}}
    \leq
    C e^{\beta_{A,N}y}
    \Ee\left[
      e^{-\alpha Y}h(y-Y)
    \right]
    =
    C\Ee\left[
      e^{(\beta_{A,N}-\alpha)Y}
      H(y-Y)
    \right],
  \end{equation}
  where $h(t)=\min\{1,e^{-Nt}\}$. Put
  $H(t):=e^{\beta_{A,N}t}h(t)$, so that
  \begin{equation*}
    H(t)
    =
    \begin{cases}
      e^{\beta_{A,N}t}, & t<0,\\
      e^{-(N-\beta_{A,N})t}, & t\geq0.
    \end{cases}
  \end{equation*}
  Since $0<\beta_{A,N}<N$, both tails decay exponentially and
  \begin{equation*}
    \sup_{a\in\R}
    \sum_{k\in\Z}\sup_{y\in[k,k+1]}H(y-a)
    \leq C .
  \end{equation*}
  Using this bound in \cref{eq:nd-forcing-envelope-bound} gives
  \begin{equation*}
    \begin{split}
    \sum_{k\in\Z}
    \sup_{y\in[k,k+1]}\norm{\Psi_y^\pi}_{\mathrm{TV}}
    &\leq
    C\Ee\left[
      e^{(\beta_{A,N}-\alpha)Y}
      \sum_{k\in\Z}
      \sup_{y\in[k,k+1]}H(y-Y)
    \right]
    \\
    &\leq
    C\Ee e^{(\beta_{A,N}-\alpha)Y}.
    \end{split}
  \end{equation*}
  The last expectation is finite by
  \cref{lem:nd-subcritical-exponential-moments}, applied with exponent
  $\beta_{A,N}-\alpha$. This proves \cref{eq:nd-dri-definition}.

  \emph{Step 2: Continuity of $y\mapsto\Psi_y^\pi(\varphi)$.}
  Fix a continuous $\varphi$. Conditional on $(Y,\Theta)$, the vector
  $W\Theta$ has the nondegenerate Gaussian law
  $\mathcal N(0,(A^2/N)I_N)$. The map
  \begin{equation*}
    v\longmapsto\tanh(e^{-Y}v)
  \end{equation*}
  is a diffeomorphism from $\R^N$ onto $(-1,1)^N$. Thus both $W\Theta$ and
  $\tanh(e^{-Y}W\Theta)$ have conditional densities. Since Euclidean spheres
  have Lebesgue measure zero, the two variables
  \begin{equation*}
    Y+\Delta=Y-\log\norm{W\Theta}_2,
    \qquad
    Y+\Delta+\eta
    =-\log\norm{\tanh(e^{-Y}W\Theta)}_2
  \end{equation*}
  have no atoms, conditionally and hence unconditionally. If $y_n\to y$, the
  two indicators in \cref{eq:nd-forcing-definition} therefore converge almost
  surely. The boundedness of $\varphi$ and dominated convergence show that
  $\Psi_{y_n}^\pi(\varphi)\to\Psi_y^\pi(\varphi)$.

  \emph{Step 3: Positivity of $c_{\Psi^\pi}$.}
  Taking $\varphi=\mathbf 1$ and using $\eta\geq0$, we obtain
  \begin{equation*}
    \Psi_y^\pi(\mathbf 1)
    =
    e^{\beta_{A,N}y}
    \Pp\left(Y+\Delta\leq y<Y+\Delta+\eta\right)
    \geq0 .
  \end{equation*}
  The bound in Step~1 implies that this function is integrable. Tonelli's
  theorem gives
  \begin{equation*}
    \int_\R\Psi_y^\pi(\mathbf 1)\,\mathrm{d}y
    =
    \frac1{\beta_{A,N}}
    \Ee\left[
      e^{\beta_{A,N}(Y+\Delta)}
      \left(e^{\beta_{A,N}\eta}-1\right)
    \right].
  \end{equation*}
  Since $Y<\infty$ and $W\Theta\neq0$ almost surely, the strict inequality
  \begin{equation*}
    \norm{\tanh(e^{-Y}W\Theta)}_2
    <
    e^{-Y}\norm{W\Theta}_2
    \qquad\mbox{almost surely}
  \end{equation*}
  implies $\eta>0$ almost surely. Hence the last expectation is strictly
  positive. Since $\widehat m_{A,N}>0$ by
  \cref{lem:nd-gaussian-renewal}, \cref{eq:nd-forcing-constant} gives
  $c_{\Psi^\pi}>0$, which proves the proposition.
\end{proof}

\begin{proof}[Proof of \cref{thm:nd-power-singularity:intro}]
  \emph{Step 1: Boundary conditions for $\Hcal_y$.}
  Let
  $(Y,\Theta)$ have law $\widetilde{\pi}_{A,N}$ and define, for $y\in\R$,
  \begin{equation}\label{eq:nd-weighted-tail-measure}
    \Hcal_y(B)
    :=
    e^{\beta_{A,N}y}
    \Ee\left[
      \one_{\{Y>y\}}\one_{\{\Theta\in B\}}
    \right].
  \end{equation}
  Each $\Hcal_y$ is a nonnegative finite measure on $\mathbb S^{N-1}$. Since
  $\widetilde{\pi}_{A,N}$ is the image of $\pi_{A,N}$ under
  $x\mapsto(-\log\norm{x}_2,\hat x)$, we have
  \begin{equation}\label{eq:nd-tail-measure-identity}
    \Hcal_y(B)
    =
    e^{\beta_{A,N}y}\,
    \pi_{A,N}\!\left(0<\norm{x}_2<e^{-y},\ \hat x\in B\right).
  \end{equation}
  By \cref{prop:nd-stationary-equation}, $\pi_{A,N}$ is absolutely continuous on
  $(-1,1)^N$ and therefore charges no sphere $\{\norm{x}_2=s\}$. Thus the strict
  radial inequality in \cref{eq:nd-tail-measure-identity} may be replaced by
  $\norm{x}_2\leq e^{-y}$.

  Since $Y\geq-\frac12\log N$,
  $\norm{\Hcal_y}_{\mathrm{TV}}=e^{\beta_{A,N}y}$ for
  $y<-\frac12\log N$, and hence
  $\norm{\Hcal_y}_{\mathrm{TV}}\to0$ as $y\to-\infty$.
  Moreover,
  \begin{equation*}
    e^{-\beta_{A,N}y}\norm{\Hcal_y}_{\mathrm{TV}}
    =
    \Pp(Y>y)
    \to0,
    \qquad y\to\infty.
  \end{equation*}
  The bound $\norm{\Hcal_y}_{\mathrm{TV}}\leq e^{\beta_{A,N}y}$ also gives local
  boundedness. Hence $\Hcal_y$ satisfies
  \cref{eq:nd-renewal-boundary-conditions}.

  \emph{Step 2: The renewal equation for $\Hcal_y$.}
  Let $W$ be an independent copy of
  $\mathsf W_A^{(N)}$, independent of $(Y,\Theta)$, and use the notation of
  \cref{def:nd-nonlinear-forcing}.
  For a continuous test function $\varphi$ on $\mathbb S^{N-1}$,
  stationarity and \cref{eq:nd-log-polar-stationary} give
  \begin{equation}\label{eq:nd-tail-stationarity}
    \Hcal_y(\varphi)
    =
    e^{\beta_{A,N}y}
    \Ee\left[
      \varphi(\Theta_+)
      \one_{\{Y+\Delta+\eta>y\}}
    \right].
  \end{equation}
  Add and subtract the linearized term with $(\Theta_0,Y+\Delta)$. The difference
  between the two terms is precisely
  $\Psi_y^\pi(\varphi)$ from \cref{eq:nd-forcing-definition}.
  Since $\widehat P_{A,N}$ is the $e^{\beta_{A,N}z}$-tilt of the linear
  log-polar transition kernel $P_{A,N}^{\mathrm{lin}}$ from
  \cref{eq:nd-transfer-operator},
  \begin{equation*}
    \begin{split}
    &\int_{\mathbb S^{N-1}}
    \int
    \varphi(\theta')\,
    \widehat P_{A,N}(\theta,\mathrm{d}z,\mathrm{d}\theta')\,
    \Hcal_{y-z}(\mathrm{d}\theta)
    \\
    &\qquad =
    e^{\beta_{A,N}y}
    \Ee\left[
      \varphi(\Theta_0)
      \one_{\{Y+\Delta>y\}}
    \right].
    \end{split}
  \end{equation*}
  Therefore
  \begin{equation}\label{eq:nd-renewal-equation}
    \Hcal_y(\varphi)
    =
    \int_{\mathbb S^{N-1}}
    \int
    \varphi(\theta')\,
    \widehat P_{A,N}(\theta,\mathrm{d}z,\mathrm{d}\theta')\,
    \Hcal_{y-z}(\mathrm{d}\theta)
    +
    \Psi_y^\pi(\varphi).
  \end{equation}
  This is \cref{eq:nd-abstract-renewal-equation}.

  \emph{Step 3: The limit from the renewal lemma.}
  Step~1 verifies the boundary conditions and local boundedness of $\Hcal_y$,
  and Step~2 gives the renewal equation. By
  \cref{prop:nd-forcing-admissibility}, $\Psi^\pi$ satisfies
  \cref{eq:nd-dri-definition}. Continuity of
  $y\mapsto\Psi_y^\pi(\varphi)$ was also established in its proof. Thus the
  assumptions of \cref{lem:nd-gaussian-renewal} hold. The constant
  $c_{\Psi^\pi}$ in that lemma is positive by
  \cref{prop:nd-forcing-admissibility}. By the uniqueness in
  \cref{cor:gaussian-unique-vector-invariant}, it depends only on $(A,N)$.
  Set $c_{A,N}:=c_{\Psi^\pi}$. The lemma gives
  \begin{equation}\label{eq:nd-weighted-renewal-limit}
    \Hcal_y(\varphi)
    \to
    c_{A,N}
    \int_{\mathbb S^{N-1}}\varphi(\theta)\,
    \bar\sigma_N(\mathrm{d}\theta)
    \qquad \mbox{as } y\to\infty .
  \end{equation}

  \emph{Step 4: The directional and radial asymptotics.}
  The convergence in \cref{eq:nd-weighted-renewal-limit} means that
  $\Hcal_y$ converges weakly to $c_{A,N}\bar\sigma_N$. By the Portmanteau
  theorem,
  $\Hcal_y(B)\to c_{A,N}\bar\sigma_N(B)$ for every Borel set $B$ with
  $\bar\sigma_N(\partial B)=0$. Taking $y=-\log s$ in
  \cref{eq:nd-tail-measure-identity} and using that $\pi_{A,N}$ charges no
  sphere gives
  \begin{equation*}
    s^{-\beta_{A,N}}
    \pi_{A,N}\!\left(0<\norm{x}_2\leq s,\ \hat x\in B\right)
    =\Hcal_{-\log s}(B)
    \longrightarrow c_{A,N}\bar\sigma_N(B),
    \qquad s\downarrow0.
  \end{equation*}
  This is the directional asymptotic. Taking $B=\mathbb S^{N-1}$, which has
  empty boundary, gives the radial asymptotic.
\end{proof}

\section{Numerical illustrations}
\label{sec:numerics}

The figures compare the theoretical predictions with numerical calculations.
We evaluate $V_A$ and $V_{A,\varrho}$ by Gauss--Hermite
quadrature or Poisson-summation formulas. These curves contain no sampling
error. The chains $R^{(N)}$, $Z^{(N)}$, and $\mathscr H^{(N)}$ are simulated
with fixed random seeds. The dashed curves
in the supercritical cutoff plot are heuristic finite-$N$ Gaussian
approximations included for comparison. The scripts are available from the
author upon request.

\emph{The map $V_A$.}\par\noindent
As $N\to\infty$, the iterates of $V_A(q)=\Ee\tanh^2(A\sqrt q\,G)$ describe the
deterministic limit of $Z^{(N)}$. \Cref{fig:meanfield} shows the transition at
$A=1$ and, in the supercritical regime, convergence to the unique positive
fixed point $q_\ast$. Its multiplier $\mu_A=V_A'(q_\ast)\in(0,1)$ determines
the cutoff scale in dimension.

\begin{figure}[H]
  \centering
  \includegraphics[width=\textwidth]{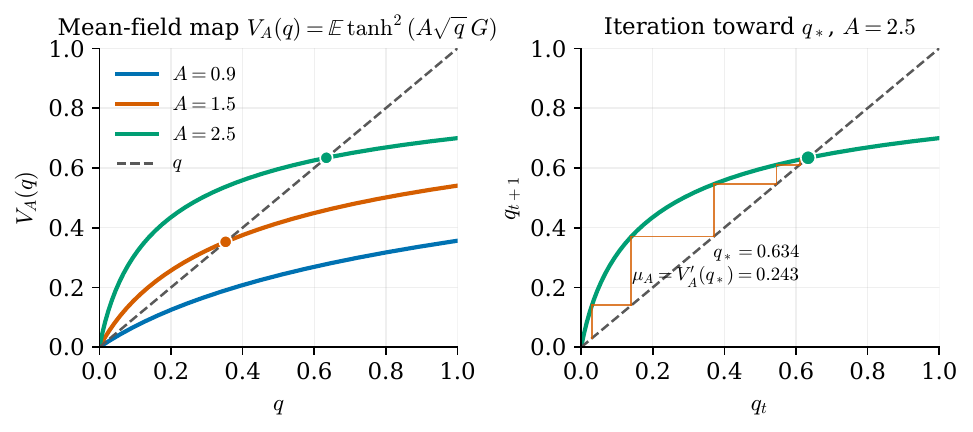}
  \caption{The map $V_A$. Left: $V_A$ and the diagonal for three
  gains, showing the positive fixed point for $A>1$. Right: iteration of
  $V_A$ toward $q_\ast$ for $A=2.5$.}
  \label{fig:meanfield}
\end{figure}

\emph{Supercritical dimension cutoff.}
\Cref{fig:cutoff} shows the empirical
total variation distance $\norm{K_{A,N}^t(q_0,\cdot)-\nu_{A,N}}_{\TV}$ of
$Z^{(N)}$ to its nonzero invariant law, plotted against the centered time
$t-t_N(q_0)$ for a range of widths $N$. The curves stay near one before the
cutoff time and drop toward zero afterwards, consistently with
\cref{thm:gaussian-process-cutoff}.

\begin{figure}[H]
  \centering
  \includegraphics[width=0.66\textwidth]{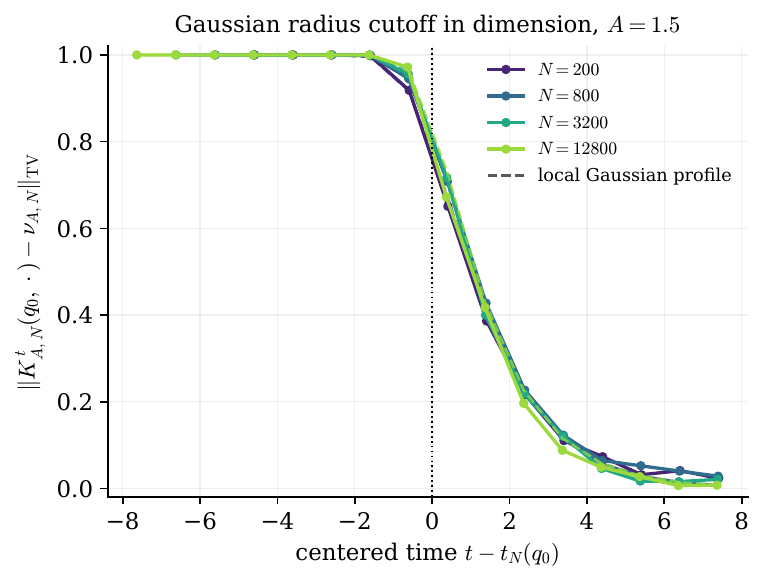}
  \caption{Numerical supercritical dimension cutoff for $Z^{(N)}$, with
  $A=1.5$ and $q_0=0.05$. The solid curves are histogram estimates
  of the total variation distance to $\nu_{A,N}$ from $4{,}000$ simulated
  paths. The dashed curves are heuristic local Gaussian approximations.}
  \label{fig:cutoff}
\end{figure}

\emph{Invariant law and power-law singularity.}
In the supercritical regime, $X^{(N)}$ has a nonzero invariant law
characterized by \cref{prop:nd-stationary-equation}. \Cref{fig:radiuslaw}
compares a long run of $R^{(N)}$ with a numerical solution of a
discretized invariant equation. For $N=1$, the transition probabilities are
computed from the exact conditional distribution on a mesh refined near zero.
The logarithmic lower-tail plot also shows a
reference line with the Cram\'er slope $\beta_{A,N}$ from
\cref{thm:nd-power-singularity:intro}.

\begin{figure}[H]
  \centering
  \includegraphics[width=\textwidth]{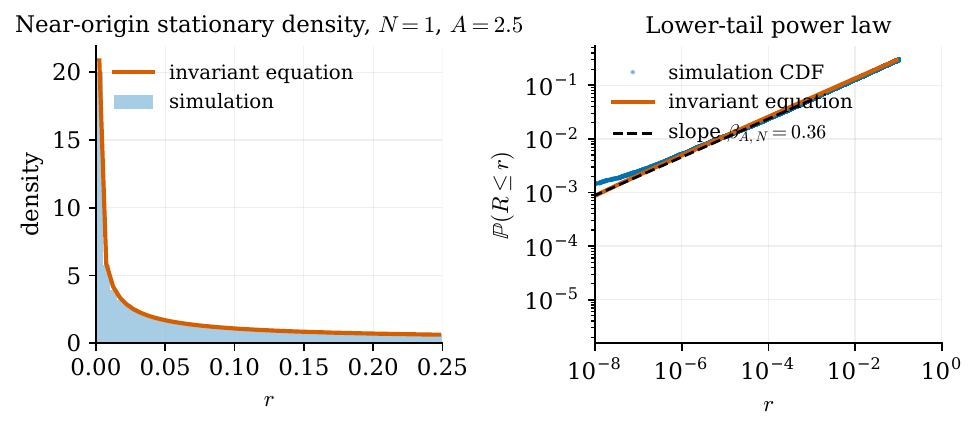}
  \caption{Supercritical stationary Euclidean radius for $N=1$ and $A=2.5$.
  Left: the near-origin density from simulation and a numerical solution of
  the invariant equation. The density increases toward zero because
  $\beta_{A,1}<1$. Right: the empirical and numerical lower tails and a reference line of
  slope $\beta_{A,1}$. The observed slope is consistent with
  \cref{thm:nd-power-singularity:intro}.}
  \label{fig:radiuslaw}
\end{figure}

\emph{Vanishing-mesh cutoff at fixed width.}
At fixed width and subcritical gain, \cref{eq:tv-absorption} identifies the
survival probability with the total variation distance to $\delta_0$.
\Cref{fig:precision} plots this probability on the theorem scale, with center
$L_\varrho/\gamma_{A,N}$ and window
$\sigma_N\gamma_{A,N}^{-3/2}\sqrt{L_\varrho}$. As $\varrho$ decreases, the
curves approach the Gaussian first-passage profile
$\Phi_{\mathrm G}(-a)$ of
\cref{thm:rounded-gaussian-nearest-cutoff,prop:rounded-gaussian-threshold-cutoff}.
For the smallest displayed $\varrho$, the figure also includes the random-walk
approximation, which shows the lattice correction.

\begin{figure}[H]
  \centering
  \includegraphics[width=0.66\textwidth]{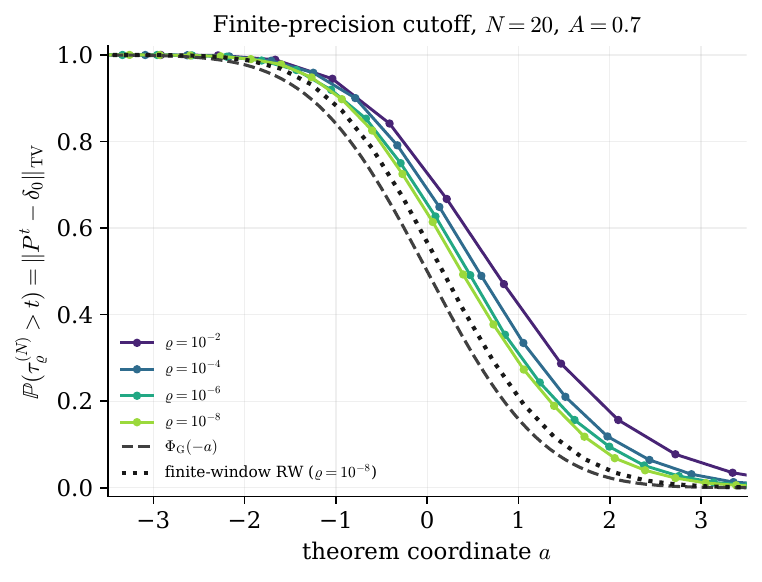}
  \caption{Vanishing-mesh cutoff at fixed width, with $N=20$ and $A=0.7$.
  The survival probability, equal to the total variation distance to
  $\delta_0$, is plotted using the center and window in
  \cref{thm:rounded-gaussian-nearest-cutoff}. The dashed curve is
  $\Phi_{\mathrm G}(-a)$. The dotted curve is the random-walk approximation at
  $\varrho=10^{-8}$.}
  \label{fig:precision}
\end{figure}

\emph{Subcritical cutoff as $N\to\infty$.}
Fix the mesh size and let the dimension grow with $A<A_{\mathrm{lat}}$. For
the displayed parameters,
$A=0.7<A_{\mathrm{lat}}\approx0.8663$ and $\varrho=0.5$. By
\cref{thm:subcritical-dimension-cutoff}, the survival probability remains near
one at time $\bar t_N-1$ and is near zero by time $\bar t_N+2$, where
$\bar t_N$ is the first time at which $(\bar h_{t,N})_{t\geq0}$ enters
the scale $(\log N)^{-1}$. \Cref{fig:subcritical-dimension-cutoff} plots the
empirical survival probability against the centered integer time
$t-\bar t_N$. For this figure, we sample each step of $\mathscr H^{(N)}$
exactly.

\begin{figure}[H]
  \centering
  \includegraphics[width=0.66\textwidth]{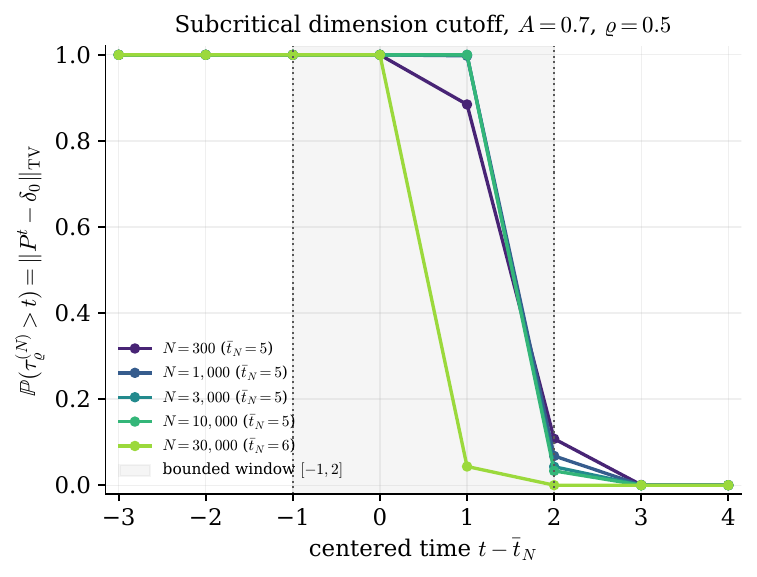}
  \caption{Fixed-precision subcritical dimension cutoff ($A=0.7$,
  $\varrho=0.5$, and $x_N=(1,\ldots,1)$, with $40{,}000$ paths per
  dimension). The empirical survival probability
  $\Pp(\tau_\varrho^{(N)}>t)=
  \norm{P_{\varrho,A,N}^t(Q_\varrho x_N,\cdot)-\delta_0}_{\TV}$ is plotted
  against $t-\protect\bar t_N$. The dotted lines mark the lower and upper time
  bounds in \cref{thm:subcritical-dimension-cutoff}, namely
  $\protect\bar t_N-1$ and $\protect\bar t_N+2$.}
  \label{fig:subcritical-dimension-cutoff}
\end{figure}

\emph{Metastability at fixed precision.}
When the rightmost positive-drift component $(h_-,h_+)$ exists,
$\mathscr H^{(N)}$ approaches $h_+$ and remains nearby until a rare fluctuation
triggers absorption. See \cref{prop:metastable-positive-fixed-points}.
\Cref{fig:metastability} shows $V_{A,\varrho}$, sample paths exhibiting the two
time scales, and the growth of the absorption time with width.

\begin{figure}[H]
  \centering
  \includegraphics[width=\textwidth]{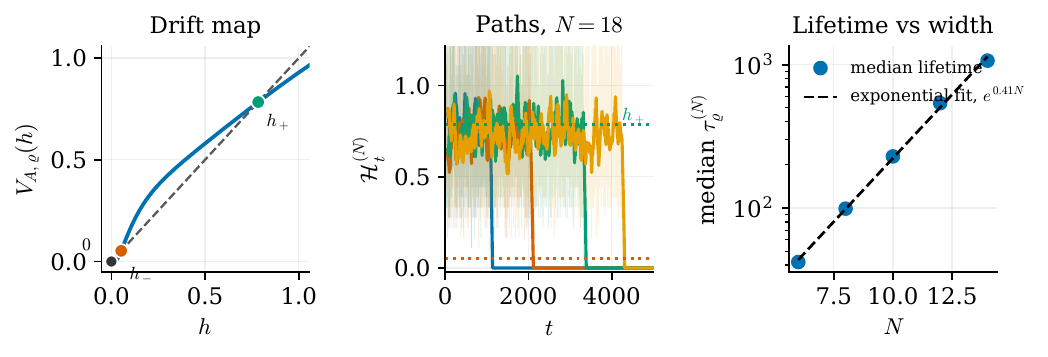}
  \caption{Metastability for $A=1.15$ and $\varrho=0.5$. Left: the
  map $V_{A,\varrho}$ with unstable endpoint $h_-$ and stable endpoint $h_+$.
  Middle: representative paths and moving averages with window $60$.
  Right: median absorption times and an exponential fit. The observed growth
  is consistent with \cref{thm:rounded-qualitative-metastability}.}
  \label{fig:metastability}
\end{figure}

\emph{Critical curves.}
\Cref{fig:thresholds} compares $A_c(N)$ from
\cref{eq:fixed-width-critical}, which decreases to $1$, with
$A_{\mathrm{lat}}$ and $A_{\mathrm{ex}}(\varrho)$.

\begin{figure}[H]
  \centering
  \includegraphics[width=\textwidth]{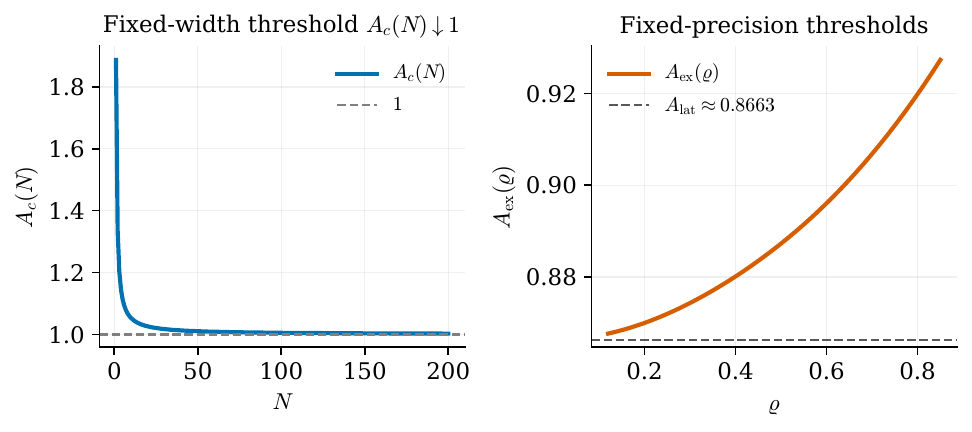}
  \caption{Thresholds. Left: $A_c(N)$ decreases to $1$. Right:
  $A_{\mathrm{ex}}(\varrho)$, evaluated numerically from its variational
  definition, and $A_{\mathrm{lat}}$.}
  \label{fig:thresholds}
\end{figure}

\clearpage
\appendix

\section{Proofs of auxiliary probabilistic results}
\label{app:common-probabilistic-estimates}

This appendix contains proofs of
\cref{lem:common-stopped-orbit-tracking,lem:negative-drift-affine-entrance}
and
\cref{prop:gaussian-compactness-selection,prop:gaussian-nonzero-invariant-existence}.

\subsection{Proof of
\texorpdfstring{\Cref{lem:common-stopped-orbit-tracking}}
{Lemma~\ref{lem:common-stopped-orbit-tracking}}}
\label{app:common-stopped-orbit-tracking}

\begin{proof}
  Put $T:=T_N$ and suppress the index $N$ on $\mathcal R_t$, $A_t$,
  $\eta_t$, $\alpha_t$, and $\sigma_t$ throughout the proof.

  \emph{Step 1: The moment before $\tau_N$.}
  Fix $0\leq t<T$. Since $\tau_N$ is a stopping time,
  $\{\tau_N>t\}\in\mathcal F_t$. Moreover,
  $\{\tau_N>t+1\}\subseteq\{\tau_N>t\}$, and on the latter event the
  recursion \cref{eq:common-tracking-recursion} applies. Consequently,
  \begin{equation*}
    \mathcal R_{t+1}^2\one_{\{\tau_N>t+1\}}
    \leq
    \one_{\{\tau_N>t\}}
    (A_t\mathcal R_t+\eta_{t+1})^2.
  \end{equation*}
  The variables $\one_{\{\tau_N>t\}}$, $A_t$, and $\mathcal R_t$ are
  $\mathcal F_t$-measurable. Thus the conditional cross term vanishes by
  \cref{eq:common-tracking-noise}, and
  \begin{equation*}
    \Ee\left[
      \mathcal R_{t+1}^2\one_{\{\tau_N>t+1\}}
      \,\middle|\,\mathcal F_t
    \right]
    \leq
    \one_{\{\tau_N>t\}}
    \left(
      A_t^2\mathcal R_t^2
      +\Ee[\eta_{t+1}^2\mid\mathcal F_t]
    \right)
    \leq
    \one_{\{\tau_N>t\}}
    \left(
      \alpha_t^2\mathcal R_t^2
      +\frac{C\sigma_t^2}{N}
    \right).
  \end{equation*}
  Set
  $m_t:=\Ee[\mathcal R_t^2\one_{\{\tau_N>t\}}]$. Taking expectations in
  the preceding display and using $\Pp(\tau_N>t)\leq1$ gives
  \begin{equation*}
    m_{t+1}
    \leq
    \alpha_t^2m_t+\frac{C\sigma_t^2}{N}.
  \end{equation*}
  Iterating this deterministic recursion, and using
  $m_0\leq\Ee\mathcal R_0^2$, yields
  \begin{equation*}
    m_t
    \leq
    \left(\prod_{u=0}^{t-1}\alpha_u^2\right)
    \Ee\mathcal R_0^2
    +\frac{C}{N}\sum_{s=0}^{t-1}
    \sigma_s^2\prod_{u=s+1}^{t-1}\alpha_u^2.
  \end{equation*}
  This is \cref{eq:common-killed-orbit-tracking}.

  \emph{Step 2: The moment of the stopped process.}
  Define the deterministic backward weights
  \begin{equation*}
    W_t:=\prod_{u=t}^{T-1}(1\vee\alpha_u^2),
    \qquad 0\leq t\leq T.
  \end{equation*}
  These weights satisfy
  \begin{equation*}
    W_T=1,
    \qquad
    W_{t+1}\leq W_t,
    \qquad
    \alpha_t^2W_{t+1}\leq W_t.
  \end{equation*}
  In addition, the definition of $K_N$ gives
  \begin{equation*}
    W_0^{1/2}
    =\prod_{u=0}^{T-1}(1\vee\alpha_u)
    \leq K_N,
    \qquad
    \sigma_tW_{t+1}^{1/2}
    =\sigma_t\prod_{u=t+1}^{T-1}(1\vee\alpha_u)
    \leq K_N.
  \end{equation*}

  For the one-step weighted estimate, note that on
  $\{\tau_N\leq t\}$ the stopped process is constant, whereas on
  $\{\tau_N>t\}$ the recursion applies. Hence
  \begin{equation*}
    \overline{\mathcal R}_{t+1,N}
    =
    \one_{\{\tau_N\leq t\}}\overline{\mathcal R}_{t,N}
    +
    \one_{\{\tau_N>t\}}(A_t\mathcal R_t+\eta_{t+1}).
  \end{equation*}
  The two events in this decomposition are disjoint. Taking conditional
  expectations, using \cref{eq:common-tracking-noise}, and applying the
  preceding bounds therefore gives
  \begin{align*}
    &\Ee\left[
      W_{t+1}\overline{\mathcal R}_{t+1,N}^2
      \,\middle|\,\mathcal F_t
    \right]
    \\
    &\quad=
    \one_{\{\tau_N\leq t\}}W_{t+1}
      \overline{\mathcal R}_{t,N}^2
    +
    \one_{\{\tau_N>t\}}W_{t+1}
      \left(A_t^2\mathcal R_t^2
      +\Ee[\eta_{t+1}^2\mid\mathcal F_t]\right)
    \\
    &\quad\leq
    \one_{\{\tau_N\leq t\}}W_t
      \overline{\mathcal R}_{t,N}^2
    +
    \one_{\{\tau_N>t\}}W_t\mathcal R_t^2
    +\frac{CK_N^2}{N}
    \\
    &\quad=
    W_t\overline{\mathcal R}_{t,N}^2
    +\frac{CK_N^2}{N}.
  \end{align*}
  Taking expectations and using the tower property gives, for every
  $0\leq t<T$,
  \begin{equation*}
    \Ee\left[W_{t+1}\overline{\mathcal R}_{t+1,N}^2\right]
    -\Ee\left[W_t\overline{\mathcal R}_{t,N}^2\right]
    \leq
    \frac{CK_N^2}{N}.
  \end{equation*}
  Summing from $t=0$ to $T-1$, the left-hand side telescopes, and we obtain
  \begin{equation*}
    \Ee\left[W_T\overline{\mathcal R}_{T,N}^2\right]
    -\Ee\left[W_0\overline{\mathcal R}_{0,N}^2\right]
    \leq
    \frac{CK_N^2T}{N}.
  \end{equation*}
  Since $W_T=1$, $W_0$ is deterministic, and
  $\overline{\mathcal R}_{0,N}=\mathcal R_{0,N}$, this is equivalent to
  \begin{equation*}
    \Ee\overline{\mathcal R}_{T,N}^2
    \leq
    W_0\Ee\mathcal R_{0,N}^2+\frac{CK_N^2T}{N}.
  \end{equation*}
  Since $W_0\leq K_N^2$, this proves
  \cref{eq:common-stopped-orbit-tracking}.

  \emph{Step 3: The probability of exit.}
  Suppose that $\tau_N$ is the first time at which
  $\abs{\mathcal R_{t,N}}>\delta$. On $\{\tau_N\leq T_N\}$ we have
  $T_N\wedge\tau_N=\tau_N$. Therefore, by the definition of
  the stopped process and the strict inequality in the definition of
  $\tau_N$,
  \begin{equation*}
    \overline{\mathcal R}_{T_N,N}
    =\mathcal R_{\tau_N,N},
    \qquad
    \abs{\mathcal R_{\tau_N,N}}>\delta.
  \end{equation*}
  Hence
  \begin{equation*}
    \delta^2\one_{\{\tau_N\leq T_N\}}
    \leq
    \overline{\mathcal R}_{T_N,N}^2.
  \end{equation*}
  Taking expectations and applying
  \cref{eq:common-stopped-orbit-tracking}, we obtain
  \begin{align*}
    \delta^2\Pp(\tau_N\leq T_N)
    &=
    \delta^2\Ee\one_{\{\tau_N\leq T_N\}}
    \\
    &\leq
    \Ee\overline{\mathcal R}_{T_N,N}^2
    \\
    &\leq
    CK_N^2\left(
      \Ee\mathcal R_{0,N}^2+\frac{T_N}{N}
    \right).
  \end{align*}
  Dividing by $\delta^2$ proves \cref{eq:common-orbit-exit-bound} and
  completes the proof.
\end{proof}

\subsection{Proof of
\texorpdfstring{\Cref{lem:negative-drift-affine-entrance}}
{Lemma~\ref{lem:negative-drift-affine-entrance}}}
\label{app:common-affine-entrance}

\begin{proof}
  \emph{Step 1: The recursion before $T$.}
  By our assumptions we can choose $\eps>0$ so small that
  $\mu_\eps:=\Ee\log(M_1+\eps)<0$. Take $K_\ast\geq b/\eps$. Whenever
  $U_t\geq K_\ast$,
  \begin{equation*}
    U_{t+1}
    =
    M_{t+1}U_t+b
    \leq
    M_{t+1}U_t+\eps K_\ast
    \leq
    (M_{t+1}+\eps)U_t .
  \end{equation*}
  Let $T:=\inf\{t:U_t\leq K_\ast\}$. On the event $\{T>n\}$ we have
  $U_t>K_\ast$ for every $t\leq n$, so the previous bound applies at each step
  and
  \begin{equation*}
    \log U_n\leq\log K+\widetilde S_n,
    \qquad
    \widetilde S_n:=\sum_{j=1}^n\log(M_j+\eps),
  \end{equation*}
  a random walk with increment mean $\mu_\eps<0$. Since $U_n>K_\ast$ on
  $\{T>n\}$, this forces $\widetilde S_n>-\log(K/K_\ast)$.

  \emph{Step 2: The tail of $T$.}
  By hypothesis $\log(M_1+\eps)$ has exponential moments near the origin, so
  $\Lambda(s):=\log\Ee(M_1+\eps)^s$ is finite for $s$ in a neighborhood of $0$,
  with $\Lambda(0)=0$ and $\Lambda'(0)=\mu_\eps<0$. Fix $s>0$ small enough that
  $\Lambda(s)<0$ and put $c_3:=-\Lambda(s)>0$. Chernoff's inequality gives, for
  every integer $n\geq1$,
  \begin{equation*}
    \Pp(T>n)
    \leq
    \Pp\!\left(\widetilde S_n>-\log(K/K_\ast)\right)
    \leq
    \left(\frac K{K_\ast}\right)^{s}\left(\Ee(M_1+\eps)^s\right)^{n}
    =
    \left(\frac K{K_\ast}\right)^{s}e^{-c_3n}.
  \end{equation*}
  Since $T$ is integer valued,
  $\Pp(T>c_1\log K+r)=\Pp(T>\lfloor c_1\log K+r\rfloor)$, and using
  $\lfloor c_1\log K+r\rfloor>c_1\log K+r-1$ with the choice $c_1:=s/c_3$,
  \begin{equation*}
    \Pp(T>c_1\log K+r)
    \leq
    \left(\frac K{K_\ast}\right)^{s}e^{c_3}e^{-c_3(c_1\log K+r)}
    =
    K_\ast^{-s}e^{c_3}\,K^{s-c_3c_1}e^{-c_3r}
    =
    K_\ast^{-s}e^{c_3}e^{-c_3r},
  \end{equation*}
  because $s-c_3c_1=0$. This is \cref{eq:negative-drift-affine-entrance} with
  $c_2:=\max\{K_\ast^{-s}e^{c_3},1\}$ and completes the proof.
\end{proof}

\subsection{Proof of
\texorpdfstring{\Cref{prop:gaussian-compactness-selection}}
{Proposition~\ref{prop:gaussian-compactness-selection}}}
\label{app:common-invariant-law}

\begin{proof}
  We use the Krylov--Bogolyubov averaging argument. Compare with
  \cite{kryloff-bogoliouboff1937}.

  \emph{Step 1: $K_{A,N}$ is Feller.}
  The space $[0,1]$ is compact, and hence the set of probability measures on
  $[0,1]$ is weakly compact. The transition kernel is
  Feller: if $\phi$ is continuous and bounded on $[0,1]$, then
  \begin{equation*}
    K_{A,N}\phi(q)
    =
    \int_{\R^N}
    \phi\left(F_{A,N}(q,g)\right)\,\mathcal G_N(\mathrm{d}g)
  \end{equation*}
  is continuous in $q$, by the continuity of $F_{A,N}$ and dominated
  convergence.

  \emph{Step 2: A Ces\`aro limit is invariant.}
  Let $\overline\nu_{T_j}^{\,q}\Rightarrow\overline\nu$. For every continuous
  $\phi$,
  \begin{align*}
    \overline\nu_{T_j}^{\,q}K_{A,N}\phi
    -
    \overline\nu_{T_j}^{\,q}\phi
    &=
    \frac1{T_j}
    \left(
      K_{A,N}^{T_j}\phi(q)-\phi(q)
    \right).
  \end{align*}
  The right-hand side tends to zero. Passing to the limit, using the Feller
  property for $K_{A,N}\phi$, gives
  $\overline\nu K_{A,N}\phi=\overline\nu\phi$ for every continuous $\phi$.
  Since probability measures on $[0,1]$ are determined by their integrals
  against continuous functions, $\overline\nu$ is invariant and satisfies
  \cref{eq:gaussian-q-stationary-equation}.

  \emph{Step 3: $\nu^+$ is invariant on $(0,1]$.}
  Since $K_{A,N}(0,\cdot)=\delta_0$ and
  $K_{A,N}(s,\{0\})=0$ for $s>0$ (all Gaussian coordinates would have to
  vanish), an invariant law
  $\overline\nu\neq\delta_0$ decomposes as
  \begin{equation*}
    \overline\nu
    =
    \overline\nu(\{0\})\delta_0
    +
    \left(1-\overline\nu(\{0\})\right)\nu^+,
  \end{equation*}
  with $\nu^+$ invariant on $(0,1]$. The assertion for the vector law is then
  the invariant-law conclusion of \cref{prop:gaussian-tv-reduction}.
  The stationary equation for $X^{(N)}$ is
  \cref{eq:gaussian-vector-weak-stationary-equation}.
\end{proof}

\subsection{Proof of
\texorpdfstring{\Cref{prop:gaussian-nonzero-invariant-existence}}
{Proposition~\ref{prop:gaussian-nonzero-invariant-existence}}}
\label{app:common-nonzero-invariant-existence}

\begin{proof}
  \emph{Step 1: Negative moments of $S_N$.}
  Let $G_1,\ldots,G_N$ be independent standard Gaussians and put
  \begin{equation*}
    S_N:=\frac1N\sum_{i=1}^N\min(G_i^2,1).
  \end{equation*}
  If $0<u<1/N$, then $S_N\leq u$ implies
  $\abs{G_i}\leq\sqrt{Nu}$ for every $i$. Hence
  \begin{equation*}
    \Pp(S_N\leq u)
    \leq
    \prod_{i=1}^N\Pp(\abs{G_i}\leq\sqrt{Nu})
    \leq
    C_N u^{N/2}.
  \end{equation*}
  Consequently, $\Ee S_N^{-p}<\infty$ for every $p<N/2$.

  \emph{Step 2: Contraction of $r^{-p}$ near zero.}
  Since $A>A_c(N)$,
  \begin{equation*}
    \frac{d}{dp} \Ee\left[
      \left(A^2\frac{\chi_N^2}{N}\right)^{-p}
    \right]_{p=0}
    =
    -\Ee\log\left(A^2\frac{\chi_N^2}{N}\right)
    =2\gamma_{A,N}<0,
  \end{equation*}
  where $\ell_{A,N}$ is as in \cref{eq:gaussian-ell}, and
  $\gamma_{A,N}=-\Ee\log\ell_{A,N}<0$ in the present supercritical regime.
  Therefore there exists $p\in(0,N/2)$ such that
  \begin{equation}\label{eq:gaussian-simple-linear-negative-moment}
    \Ee\left[
      \left(A^2\frac{\chi_N^2}{N}\right)^{-p}
    \right]
    <1 .
  \end{equation}
  We claim that, for some $R_0\in(0,1)$ and $a<1$,
  \begin{equation}\label{eq:gaussian-simple-near-zero-drift}
    \sup_{0<r\leq R_0}
    \Ee\left[
      \left(\frac{F_{A,N}(r,G)}{r}\right)^{-p}
    \right]
    \leq a .
  \end{equation}
  Indeed, for fixed $g$ the quotient
  $r^{-1}\tanh^2(A\sqrt r\,g)$ decreases in $r$ and converges to
  $A^2g^2$ as $r\downarrow0$. Therefore
  $F_{A,N}(r,G)/r$ converges almost surely to
  $A^2\chi_N^2/N$. Moreover, for each fixed $R>0$ and $0<q\leq R$, the same
  monotonicity gives
  \begin{equation*}
    \frac{F_{A,N}(q,G)}q
    \geq
    \frac{F_{A,N}(R,G)}R .
  \end{equation*}
  The random variable on the right is bounded below by $c_{A,R,N}S_N$.
  The estimate from Step 1 gives
  $\Ee S_N^{-p}<\infty$. Dominated convergence therefore gives the convergence
  of the expectations, and \cref{eq:gaussian-simple-near-zero-drift} follows from
  \cref{eq:gaussian-simple-linear-negative-moment} by taking $R_0$ small.

  \emph{Step 3: A bound away from zero.}
  There is $B<\infty$ such that
  \begin{equation}\label{eq:gaussian-simple-outside-bound}
    \sup_{r\in[R_0,1]}\Ee F_{A,N}(r,G)^{-p}\leq B .
  \end{equation}
  This follows from the pointwise lower bound
  \begin{equation*}
    \tanh^2(A\sqrt r\,g)\geq c_{A,R_0}\min(g^2,1),
    \qquad r\in[R_0,1],
  \end{equation*}
  together with the estimate for $S_N$ from Step 1.

  Let $V(r):=r^{-p}$ for $r>0$. Combining the two estimates gives
  \begin{equation}\label{eq:gaussian-simple-foster}
    K_{A,N}V(r)\leq aV(r)+B,
    \qquad r>0.
  \end{equation}

  \emph{Step 4: The Ces\`aro limit does not charge zero.}
  Let $Z^{(N)}$ start from $q\in(0,1]$. Iterating
  \cref{eq:gaussian-simple-foster} gives
  \begin{equation*}
    \frac1T\sum_{t=0}^{T-1}\Ee V(Z_t^{(N)})
    \leq
    \frac{V(q)}{T(1-a)}+\frac{B}{1-a}.
  \end{equation*}
  Choose a weakly convergent Cesàro subsequence
  $\overline\nu_{T_j}^{\,q}\Rightarrow\overline\nu$. For
  $M<\infty$, the function $V^M(r):=\min\{r^{-p},M\}$ for $r>0$ and
  $V^M(0):=M$ is continuous on $[0,1]$. Passing to the limit gives a bound
  independent of $M$. Since
  $M\overline\nu(\{0\})\leq\int V^M\,\mathrm{d}\overline\nu$, letting
  $M\to\infty$ first gives $\overline\nu(\{0\})=0$. Monotone convergence on
  $(0,1]$ then gives
  \begin{equation*}
    \int_{(0,1]}r^{-p}\,\overline\nu(\mathrm{d}r)<\infty,
    \qquad
    \overline\nu(\{0\})=0.
  \end{equation*}
  Thus $\overline\nu\neq\delta_0$, and
  \cref{prop:gaussian-compactness-selection} shows that
  $\overline\nu\in\mathcal I_{A,N}^+(q)$. The vector invariant law follows
  from \cref{eq:gaussian-compactness-vector-law}.
\end{proof}

\section{Proofs of estimates for supercritical cutoff as
\texorpdfstring{$N\to\infty$}{N tends to infinity}}
\label{app:supercritical-cutoff-estimates}

This appendix contains the proofs of the estimates used in
\cref{subsec:gaussian-process-cutoff}, with the notation introduced there for
$Z^{(N)}$ and $(q_{t,N})_{t\geq0}$.

\subsection{Proofs of
\texorpdfstring{\Cref{lem:gaussian-truncated-cramer,prop:gaussian-near-zero-negative-moment,lem:gaussian-outside-negative-moment,prop:gaussian-stationary-negative-moment}}
{the results on negative moments}}
\label{app:supercritical-negative-moments}

\begin{proof}[Proof of \cref{lem:gaussian-truncated-cramer}]
  \emph{Step 1: Negative moments of the truncated average.}
  Let
  \begin{equation*}
    U_i:=\min(G_i^2,1),
    \qquad
    S_N:=\frac1N\sum_{i=1}^N U_i.
  \end{equation*}
  For $\lambda\geq1$,
  \begin{equation*}
    \Ee e^{-\lambda U_1}
    \leq
    \int_{-1}^1 e^{-\lambda x^2}\frac{e^{-x^2/2}}{\sqrt{2\pi}}\,\mathrm{d}x
    +
    e^{-\lambda}
    \leq
    C\lambda^{-1/2}.
  \end{equation*}
  Chernoff's bound with $\lambda=u^{-1}$ gives
  \begin{equation*}
    \Pp(S_N\leq u)
    \leq
    e^N
    \left(\Ee e^{-U_1/u}\right)^N
    \leq
    C^N u^{N/2},
  \end{equation*}
  for $0<u\leq1$.

  For $x,p>0$, the Gamma integral and the change of variables $v=ux$ give
  \begin{equation*}
    x^{-p}
    =
    \frac1{\Gamma(p)}
    \int_0^\infty u^{p-1}e^{-ux}\,\mathrm{d}u.
  \end{equation*}
  Let $p=\alpha N$ with $\alpha\in(0,1/2)$. By Tonelli's theorem and
  independence,
  \begin{align*}
    \Ee S_N^{-p}
    &=
    \frac1{\Gamma(p)}
    \int_0^\infty
    u^{p-1}
    \left(\Ee e^{-(u/N)U_1}\right)^N
    \,\mathrm{d}u .
  \end{align*}
  We split the integral at $u=N$. On $[0,N]$ we use the trivial bound
  $\Ee e^{-(u/N)U_1}\leq1$. On $[N,\infty)$ we use the preceding
  one-coordinate Laplace bound with $\lambda=u/N\geq1$. Since
  $p=\alpha N<N/2$, this gives
  \begin{align*}
    \Ee S_N^{-p}
    &\leq
    \frac{N^p}{\Gamma(p+1)}
    +
    \frac{C^N N^{N/2}}{\Gamma(p)}
    \int_N^\infty u^{p-1-N/2}\,\mathrm{d}u
    \\
    &\leq
    C_\alpha^N
    \frac{N^p}{\Gamma(p+1)}.
  \end{align*}
  Stirling's formula, after increasing the constant to handle bounded $N$,
  shows that the last expression is bounded by $e^{C_\alpha N}$.
  If $W_{L,i}:=\min(G_i^2,L^2)$ and $L\geq1$, then $W_{L,i}\geq U_i$.
  Consequently,
  \begin{equation}\label{eq:gaussian-truncated-negative-moment}
    \sup_{L\geq1}
    \Ee\left[
      \left(
        \frac1N\sum_{i=1}^N W_{L,i}
      \right)^{-\alpha N}
    \right]
    \leq
    e^{C_\alpha N}.
  \end{equation}

  \emph{Step 2: Choice of $L$, $\varepsilon$, and $\delta$.}
  Put $\mu_L:=\Ee W_{L,1}$. Choose first $L$ large, then
  $\varepsilon>0$ small enough that $(1-\varepsilon)A^2\mu_L>1$, and finally
  $\delta>0$ such that
  \begin{equation*}
    a:=(1-\varepsilon)A^2(\mu_L-\delta)>1.
  \end{equation*}
  If $\overline W_{L,N}:=N^{-1}\sum_i W_{L,i}$, Hoeffding's inequality gives
  \begin{equation*}
    \Pp(\overline W_{L,N}<\mu_L-\delta)\leq e^{-c_0N}
  \end{equation*}
  with $c_0>0$ depending on $L$ and $\delta$.

  \emph{Step 3: Splitting at
  $\overline W_{L,N}=\mu_L-\delta$.}
  Fix $\alpha\in(0,1/2)$. By
  \cref{eq:gaussian-truncated-negative-moment}, after increasing
  $C_\alpha$ to absorb the fixed factor
  $((1-\varepsilon)A^2)^{-\alpha N}$,
  \begin{equation*}
    \Ee\left[
      \left((1-\varepsilon)A^2\overline W_{L,N}\right)^{-\alpha N}
    \right]
    \leq
    e^{C_\alpha N}.
  \end{equation*}
  On the event $\{\overline W_{L,N}\geq\mu_L-\delta\}$, the integrand in
  \cref{eq:gaussian-truncated-cramer} is at most $a^{-\gamma N}$. On the
  complementary event, H\"older's inequality gives, for $0<\gamma<\alpha$,
  \begin{align*}
    \Ee\left[
      \left((1-\varepsilon)A^2\overline W_{L,N}\right)^{-\gamma N}
      \one_{\{\overline W_{L,N}<\mu_L-\delta\}}
    \right]
    &\leq
    e^{-c_0N(1-\gamma/\alpha)}
    e^{C_\alpha N\gamma/\alpha}.
  \end{align*}
  Choose $\gamma>0$ so small that
  $\gamma < \alpha c_0/(c_0+C_\alpha)$. Then choose $\kappa>0$ such that
  \begin{equation*}
    \kappa < \min\left\{
      \gamma\log a,
      c_0(1-\gamma/\alpha)-C_\alpha\gamma/\alpha
    \right\}.
  \end{equation*}
  The good- and bad-event estimates then give
  \cref{eq:gaussian-truncated-cramer} and complete the proof.
\end{proof}

\begin{proof}[Proof of \cref{prop:gaussian-near-zero-negative-moment}]
  Let $\gamma,\kappa,L,\varepsilon$ be as in
  \cref{lem:gaussian-truncated-cramer}, decreasing $\kappa$ if necessary.
  For $q>0$ define
  \begin{equation*}
    h_q(g)
    :=
    \frac{\tanh^2(A\sqrt q\,g)}{q}
    =
    A^2g^2
    \left(
      \frac{\tanh(A\sqrt q\,g)}{A\sqrt q\,g}
    \right)^2,
  \end{equation*}
  with the continuous interpretation at $g=0$. Since
  $x\mapsto \tanh x/x$ is decreasing on $(0,\infty)$, the map
  $q\mapsto h_q(g)$ is decreasing for each fixed $g$.

  Choose $R_0\in(0,\inf I_\ast)$ so small that
  \begin{equation}\label{eq:gaussian-ratio-lower-bound}
    \left(
      \frac{\tanh(A\sqrt{R_0}u)}{A\sqrt{R_0}u}
    \right)^2
    \geq
    1-\varepsilon
    \qquad \mbox{for every }\abs u\leq L.
  \end{equation}
  For $\abs g\leq L$, \cref{eq:gaussian-ratio-lower-bound} gives
  \begin{equation*}
    h_{R_0}(g)\geq(1-\varepsilon)A^2g^2.
  \end{equation*}
  For $\abs g>L$, the map
  $\abs g\mapsto h_{R_0}(g)=R_0^{-1}\tanh^2(A\sqrt{R_0}g)$ is increasing, and
  hence
  \begin{equation*}
    h_{R_0}(g)
    \geq
    h_{R_0}(L)
    \geq
    (1-\varepsilon)A^2L^2.
  \end{equation*}
  Combining the two cases gives
  \begin{equation}\label{eq:gaussian-h-lower-truncated}
    h_{R_0}(g)
    \geq
    (1-\varepsilon)A^2\min(g^2,L^2)
    \qquad \mbox{for every }g\in\R.
  \end{equation}

  Hence, for $0<q\leq R_0$,
  \begin{equation*}
    \frac{F_{A,N}(q,G)}{q}
    =
    \frac1N\sum_{i=1}^N h_q(G_i)
    \geq
    \frac1N\sum_{i=1}^N h_{R_0}(G_i)
    \geq
    (1-\varepsilon)A^2
    \frac1N\sum_{i=1}^N\min(G_i^2,L^2).
  \end{equation*}
  Since $x\mapsto x^{-\gamma N}$ is decreasing on $(0,\infty)$, applying
  \cref{lem:gaussian-truncated-cramer} gives
  \begin{equation*}
    \sup_{0<q\leq R_0}
    \Ee\left[
      \left(
        \frac{F_{A,N}(q,G)}{q}
      \right)^{-\gamma N}
    \right]
    \leq
    e^{-\kappa N}
  \end{equation*}
  for all sufficiently large $N$.
\end{proof}

\begin{proof}[Proof of \cref{lem:gaussian-outside-negative-moment}]
  There is $c_{A,R_0}>0$ such that, for every $q\in[R_0,1]$ and every
  $g\in\R$,
  \begin{equation}\label{eq:gaussian-outside-pointwise-lower}
    \tanh^2(A\sqrt q\,g)
    \geq
    c_{A,R_0}\min(g^2,1).
  \end{equation}
  Indeed, for $\abs g\leq1$ this follows from the monotonicity of
  $\tanh x/x$ on $[0,A]$, while for $\abs g\geq1$ it follows from
  $\tanh^2(A\sqrt q\,\abs g)\geq\tanh^2(A\sqrt{R_0})$. Thus
  \begin{equation*}
    F_{A,N}(q,G)
    \geq
    c_{A,R_0}
    \frac1N\sum_{i=1}^N\min(G_i^2,1).
  \end{equation*}
  The bound follows from
  \cref{eq:gaussian-truncated-negative-moment}, after absorbing
  $c_{A,R_0}^{-\gamma N}$ into $e^{bN}$.
\end{proof}

\begin{proof}[Proof of \cref{prop:gaussian-stationary-negative-moment}]
  \emph{Step 1: The bound for $K_{A,N}V_N$.}
  Let $R_0,\gamma,\kappa$ be as in
  \cref{prop:gaussian-near-zero-negative-moment}, and let $b$ be the constant
  in \cref{lem:gaussian-outside-negative-moment} for this choice of
  $R_0,\gamma$. Define
  \begin{equation*}
    V_N(q):=q^{-\gamma N},
    \qquad q>0.
  \end{equation*}
  For $q>0$, the two negative-moment estimates give the Foster inequality
  \begin{equation}\label{eq:gaussian-stationary-negative-moment-drift}
    K_{A,N}V_N(q)
    \leq
    e^{-\kappa N}V_N(q)+e^{bN}.
  \end{equation}
  Indeed, if $q\leq R_0$ then
  \begin{equation*}
    K_{A,N}V_N(q)
    =
    q^{-\gamma N}
    \Ee\left[\left(F_{A,N}(q,G)/q\right)^{-\gamma N}\right]
    \leq
    e^{-\kappa N}V_N(q),
  \end{equation*}
  while for $q\in[R_0,1]$ the outside estimate gives
  $K_{A,N}V_N(q)\leq e^{bN}$.

  \emph{Step 2: The bound under $\nu_{A,N}$.}
  Let $(Z_t^{(N)})_{t\geq0}$ start from $1$. Iterating
  \cref{eq:gaussian-stationary-negative-moment-drift} yields
  \begin{equation}\label{eq:gaussian-stationary-negative-moment-iteration}
    \Ee V_N(Z_t^{(N)})
    \leq
    e^{-\kappa Nt}
    +
    \frac{e^{bN}}{1-e^{-\kappa N}}.
  \end{equation}
  Hence
  \begin{equation}\label{eq:gaussian-stationary-negative-moment-cesaro-bound}
    \int V_N(q)\,\overline\nu_T^{\,1}(\mathrm{d}q)
    \leq
    \frac{1}{T(1-e^{-\kappa N})}
    +
    \frac{e^{bN}}{1-e^{-\kappa N}}.
  \end{equation}
  Here $\overline\nu_T^{\,1}$ is the Cesàro average started from $q=1$ in
  \cref{eq:gaussian-cesaro-selection}.

  Take a weakly convergent subsequence
  $\overline\nu_{T_j}^{\,1}\Rightarrow\overline\nu$ with $T_j\to\infty$.
  For $M<\infty$ put $V_N^M(q):=\min\{V_N(q),M\}$ for $q>0$ and
  $V_N^M(0):=M$. Then $V_N^M$ is continuous on $[0,1]$. Applying
  \cref{eq:gaussian-stationary-negative-moment-cesaro-bound} with $T=T_j$ and
  then letting $j\to\infty$, the term containing $T_j^{-1}$ vanishes and
  \begin{equation*}
    \int V_N^M(q)\,\overline\nu(\mathrm{d}q)
    \leq
    \frac{e^{bN}}{1-e^{-\kappa N}}.
  \end{equation*}
  Letting $M\to\infty$ and using monotone convergence, we obtain
  $\int_{(0,1]}q^{-\gamma N}\,\overline\nu(\mathrm{d}q)<\infty$ and
  $\overline\nu(\{0\})=0$. By \cref{prop:gaussian-compactness-selection},
  $\overline\nu$ is a nonzero invariant law. By
  \cref{prop:gaussian-unique-nonzero-invariant},
  $\overline\nu=\nu_{A,N}$. The same monotone-convergence passage gives
  \cref{eq:gaussian-stationary-negative-moment}.
\end{proof}

\subsection{Proofs of
\texorpdfstring{\Cref{prop:gaussian-stable-localization,lem:gaussian-score-smoothing,prop:gaussian-stationary-concentration,lem:gaussian-orbit-concentration}}
{the localization and concentration results}}
\label{app:supercritical-localization-concentration}

\begin{proof}[Proof of \cref{prop:gaussian-stable-localization}]
  \emph{Step 1: $\nu_{A,N}((0,r])$ is exponentially small.}
  Let $R_0,\gamma$ be as in
  \cref{prop:gaussian-near-zero-negative-moment}, and let $b$ be the constant
  from \cref{lem:gaussian-outside-negative-moment} for this choice of
  $R_0,\gamma$. Set
  \begin{equation*}
    V_N(q):=q^{-\gamma N},
    \qquad q>0.
  \end{equation*}
  By \cref{prop:gaussian-stationary-negative-moment}, after increasing $C$ if
  necessary,
  \begin{equation}\label{eq:gaussian-lyapunov-moment-bound}
    \int_0^1 V_N(q)\,\nu_{A,N}(\mathrm{d}q)
    \leq
    C e^{bN}.
  \end{equation}

  Now choose a second threshold $r\in(0,R_0)$ so small that
  \begin{equation}\label{eq:gaussian-final-threshold-choice}
    b+\gamma\log r<0.
  \end{equation}
  Since $V_N(q)\geq r^{-\gamma N}$ on $(0,r]$,
  \begin{equation}\label{eq:gaussian-small-mass-bound}
    \nu_{A,N}((0,r])
    \leq
    r^{\gamma N}
    \int_0^1 V_N(q)\,\nu_{A,N}(\mathrm{d}q)
    \leq
    C e^{-cN}.
  \end{equation}

  \emph{Step 2: Entrance into $I_\ast$ after $m$ steps.}
  By
  \cref{lem:gaussian-mean-field-concavity}, every orbit of $V_A$ started from
  $[r,1]$ converges monotonically to $q_\ast$ without crossing it.
  Hence
  \begin{equation*}
    \abs{V_A^{t+1}(q)-q_\ast}
    \leq
    \abs{V_A^t(q)-q_\ast},
    \qquad q\in[r,1].
  \end{equation*}
  Dini's theorem shows that this convergence is uniform on $[r,1]$.
  Therefore there are $m\in\N$ and $\eta>0$ such that
  \begin{equation}\label{eq:gaussian-deterministic-fixed-entrance}
    V_A^m([r,1])
    \subset
    I_{\ast,-\eta}
    :=
    \left\{
      q\in I_\ast:\operatorname{dist}(q,I_\ast^c)\geq\eta
    \right\}.
  \end{equation}
  Let $L<\infty$ be a Lipschitz constant for $V_A$ on $[0,1]$ and set
  \begin{equation*}
    \eta_m:=\frac{\eta}{2(1+L+\cdots+L^{m-1})}.
  \end{equation*}
  If the noise variables in \cref{eq:gaussian-noise-decomposition} have
  absolute value at most $\eta_m$ during the first $m$ steps, then the
  stochastic radius at time $m$ is within $\eta/2$ of its deterministic
  counterpart and hence belongs to $I_\ast$. Hoeffding's inequality
  \cref{eq:gaussian-noise-hoeffding} and a union bound over these fixed
  $m$ steps give
  \begin{equation}\label{eq:gaussian-fixed-time-entrance}
    \sup_{q\in[r,1]}
    K_{A,N}^m(q,I_\ast^c)
    \leq
    C e^{-cN}.
  \end{equation}

  \emph{Step 3: Invariance gives the bound on $I_\ast^c$.}
  Since $\nu_{A,N}$ is invariant,
  \begin{equation*}
    \nu_{A,N}(I_\ast^c)
    =
    \int_0^1 K_{A,N}^m(q,I_\ast^c)\,\nu_{A,N}(\mathrm{d}q)
    \leq
    \nu_{A,N}((0,r])
    +
    \sup_{q\in[r,1]}K_{A,N}^m(q,I_\ast^c).
  \end{equation*}
  Combining \cref{eq:gaussian-small-mass-bound,eq:gaussian-fixed-time-entrance}
  proves \cref{eq:gaussian-stationary-exponential-localization}, and the
  weaker estimate \cref{eq:gaussian-stationary-localization} follows
  immediately.
\end{proof}

\begin{proof}[Proof of \cref{lem:gaussian-score-smoothing}]
  \emph{Step 1: The score of one coordinate.}
  Let $\mu_q$ be the law of $Y_q:=\tanh^2(A\sqrt q\,G)$ on $(0,1)$, and let
  $f_q$ be the density in \cref{eq:gaussian-one-coordinate-density}.
  Its $q$-score is
  \begin{equation}\label{eq:gaussian-one-coordinate-score}
    \ell_q(y)
    :=
    \partial_q\log f_q(y)
    =
    -\frac1{2q}
    +
    \frac{\arctanh(\sqrt y)^2}{2A^2q^2}.
  \end{equation}
  Since
  $\arctanh(\sqrt{Y_q})=\abs{A\sqrt q\,G}$, we have
  \begin{equation}\label{eq:gaussian-score-moments}
    \ell_q(Y_q)=\frac{G^2-1}{2q},
    \qquad
    \Ee \ell_q(Y_q)=0,
    \qquad
    \Ee \ell_q(Y_q)^2=\frac1{2q^2}.
  \end{equation}
  Fix a compact subinterval $I=[q_-,q_+]\Subset(0,1)$. On this interval the
  displayed density is a $C^1$ map into $L^1((0,1))$: after writing
  $u=\arctanh(\sqrt y)$, the substitution $y=\tanh^2u$ gives a
  uniform integrable Gaussian envelope for $\partial_qf_q=\ell_qf_q$. Thus
  $q\mapsto\mu_q$ is differentiable in total variation on $I$, with signed
  derivative measure $\partial_q\mu_q=\ell_q\mu_q$. Here and below
  $\abs\lambda$ denotes the total variation measure of a signed measure
  $\lambda$. With our convention for total variation distance between
  probability measures, one divides the signed variation norm by two.

  \emph{Step 2: The score of $N$ coordinates.}
  Let $Y_{1,q},\ldots,Y_{N,q}$ be independent copies of $Y_q$ and set
  \begin{equation*}
    M_{N,q}:=\frac1N\sum_{i=1}^N Y_{i,q},
    \qquad
    S_{N,q}:=\sum_{i=1}^N\ell_q(Y_{i,q}).
  \end{equation*}
  Differentiating the product measure gives
  \begin{equation}\label{eq:gaussian-product-score-derivative}
    \partial_q\mu_q^{\otimes N}
    =
    S_{N,q}\,\mu_q^{\otimes N}.
  \end{equation}
  Equivalently, for every bounded measurable $\varphi:(0,1)^N\to\R$,
  \begin{equation}\label{eq:gaussian-score-test-function}
    \partial_q\Ee\varphi(Y_{1,q},\ldots,Y_{N,q})
    =
    \Ee\left[
      \varphi(Y_{1,q},\ldots,Y_{N,q})S_{N,q}
    \right].
  \end{equation}
  Since $S_{N,q}$ is the density of $\partial_q\mu_q^{\otimes N}$ against
  $\mu_q^{\otimes N}$, its total variation norm is
  \begin{equation}\label{eq:gaussian-score-product-bound}
    \abs{\partial_q\mu_q^{\otimes N}}((0,1)^N)
    =
    \Ee\abs{S_{N,q}}
    \leq
    \left(\Ee S_{N,q}^2\right)^{1/2}
    =
    \frac{\sqrt N}{\sqrt2\,q},
  \end{equation}
  where the last identity uses \cref{eq:gaussian-score-moments} and the
  independence of the $Y_{i,q}$.

  \emph{Step 3: Passing to $K_{A,N}$.}
  The kernel $K_{A,N}(q,\cdot)$ is the law of $M_{N,q}$ and hence the
  pushforward of $\mu_q^{\otimes N}$ under the map
  $(y_1,\ldots,y_N)\mapsto N^{-1}\sum_i y_i$. Pushforward does not increase the
  total variation distance. Since $q\mapsto\mu_q^{\otimes N}$ is continuously
  differentiable in total variation,
  \begin{equation*}
    \mu_q^{\otimes N}-\mu_{q'}^{\otimes N}
    =
    \int_{q'}^q\partial_s\mu_s^{\otimes N}\,\mathrm{d}s
  \end{equation*}
  as signed measures. Integrating \cref{eq:gaussian-score-product-bound} and
  dividing the signed variation by two gives
  \begin{equation*}
    \norm{K_{A,N}(q,\cdot)-K_{A,N}(q',\cdot)}_{\TV}
    \leq
    \tfrac12\abs{\mu_q^{\otimes N}-\mu_{q'}^{\otimes N}}((0,1)^N)
    \leq
    \frac12\int_{q\wedge q'}^{q\vee q'}\frac{\sqrt N}{\sqrt2\,s}\,\mathrm{d}s
    \leq
    \frac{\sqrt N}{2\sqrt2\,q_-}\abs{q-q'} .
  \end{equation*}
  Since $q,q'\in I\Subset(0,1)$, this proves
  \cref{eq:gaussian-score-tv-smoothing} with $C_I=(2\sqrt2\,q_-)^{-1}$.
\end{proof}

\begin{proof}[Proof of \cref{prop:gaussian-stationary-concentration}]
  Put
  \begin{equation*}
    V_N:=
    \int_0^1 (q-q_\ast)^2\,\nu_{A,N}(\mathrm{d}q).
  \end{equation*}
  Invariance of $\nu_{A,N}$ and \cref{eq:gaussian-noise-moments} give
  \begin{align}
    V_N
    &=
    \int_0^1
    \Ee\left[
      \left(F_{A,N}(q,G)-q_\ast\right)^2
    \right]
    \nu_{A,N}(\mathrm{d}q)
    \notag\\
    &=
    \int_0^1
    \left(V_A(q)-q_\ast\right)^2
    \nu_{A,N}(\mathrm{d}q)
    +
    O(N^{-1}).
    \label{eq:gaussian-stationary-concentration-step}
  \end{align}
  On $I_\ast$, \cref{eq:gaussian-stable-basin-derivative} gives
  $
    \abs{V_A(q)-q_\ast}
    \leq
    \kappa\abs{q-q_\ast}
  $.
  On $I_\ast^c$, we use the trivial bound
  $
    \abs{V_A(q)-q_\ast}\leq1
  $
  and \cref{eq:gaussian-stationary-localization}. Hence
  $
    V_N
    \leq
    \kappa^2 V_N+\frac{C}{N}
  $.
  Since $\kappa<1$, this gives
  $
    V_N\leq \frac{C}{(1-\kappa^2)N}
  $.
  The claim follows after enlarging $C$.
\end{proof}

\begin{proof}[Proof of \cref{lem:gaussian-orbit-concentration}]
  \emph{Step 1: The error at time $s_\ast$.}
  Put
  $
    e_t^{(N)}:=Z_t^{(N)}-q_{t,N}
  $.
  Since the comparison orbit starts from the same finite-$N$ radius
  $q_{0,N}$, one has $e_0^{(N)}=0$. The map $V_A$ is Lipschitz on
  $[0,1]$. Since $s_\ast$ is fixed, finitely many iterations of
  \cref{eq:gaussian-noise-moments,eq:gaussian-noise-hoeffding} give
  \begin{equation}\label{eq:gaussian-entrance-error}
    \Ee[(e_{s_\ast}^{(N)})^2]\leq \frac{C}{N},
    \qquad
    \Pp\left(\abs{e_{s_\ast}^{(N)}}>\eta\right)
    \leq
    C\exp(-c_\eta N)
  \end{equation}
  for every fixed $\eta>0$.

  \emph{Step 2: The second moment before exit.}
  Use the one-step noise $\zeta_{t+1}^{(N)}$ and natural filtration
  $\mathcal F_t$ from
  \cref{eq:gaussian-noise-decomposition,eq:gaussian-noise-moments}. Then
  \begin{equation*}
    e_{t+1}^{(N)}
    =
    Z_{t+1}^{(N)}-q_{t+1,N}
    =
    V_A(Z_t^{(N)})-V_A(q_{t,N})
    +
    \zeta_{t+1}^{(N)}.
  \end{equation*}
  Moreover,
  \begin{equation*}
    \Ee[\zeta_{t+1}^{(N)}\mid \mathcal F_t]=0,
    \qquad
    \Ee[(\zeta_{t+1}^{(N)})^2\mid \mathcal F_t]\leq \frac{C}{N}.
  \end{equation*}

  On $\{t<\tau_\ast^{(N)}\}$, both $Z_t^{(N)}$ and $q_{t,N}$ lie in the
  stable interval $I_\ast$. Hence the mean value theorem and
  \cref{eq:gaussian-stable-basin-derivative} give
  \begin{equation*}
    \abs{V_A(Z_t^{(N)})-V_A(q_{t,N})}
    \leq
    \kappa\abs{e_t^{(N)}}.
  \end{equation*}
  Therefore, using the preceding decomposition and the martingale property of
  $\zeta_{t+1}^{(N)}$,
  \begin{equation*}
    \Ee\left[
      (e_{t+1}^{(N)})^2
      \,\middle|\,\mathcal F_t
    \right]
    =
    \left(
      V_A(Z_t^{(N)})-V_A(q_{t,N})
    \right)^2
    +
    \Ee\left[
      (\zeta_{t+1}^{(N)})^2
      \,\middle|\,\mathcal F_t
    \right]
    \leq
    \kappa^2(e_t^{(N)})^2+\frac{C}{N}
  \end{equation*}
  on $\{t<\tau_\ast^{(N)}\}$. Apply
  \cref{eq:common-killed-orbit-tracking} to the time-shifted process
  $(e_{s_\ast+j}^{(N)})_{j\geq0}$, with
  $\alpha_{j,N}=\kappa$ and $\sigma_{j,N}=C$. The geometric sum in that
  estimate is bounded by $(1-\kappa^2)^{-1}$, while
  \cref{eq:gaussian-entrance-error} supplies the $O(N^{-1})$ initial term.
  Hence
  \begin{equation}\label{eq:gaussian-stopped-orbit-concentration}
    \Ee\left[
      (e_t^{(N)})^2
      \one_{\{\tau_\ast^{(N)}>t\}}
    \right]
    \leq
    \frac{C}{N}.
  \end{equation}

  \emph{Step 3: Removing the stopping time.}
  Choose $\eta>0$ so small that
  \begin{equation*}
    \kappa\delta_\ast+\eta<\delta_\ast.
  \end{equation*}
  If
  \begin{equation*}
    \abs{e_{s_\ast}^{(N)}}\leq\eta
    \qquad\mbox{and}\qquad
    \abs{\zeta_j^{(N)}}\leq\eta
    \quad\mbox{for every }s_\ast<j\leq T_N,
  \end{equation*}
  then the recursion
  \begin{equation*}
    \abs{e_{t+1}^{(N)}}
    \leq
    \kappa\abs{e_t^{(N)}}+\abs{\zeta_{t+1}^{(N)}}
  \end{equation*}
  and the margin in \cref{eq:gaussian-deterministic-entrance} imply
  $Z_t^{(N)}\in I_\ast$ for every $s_\ast\leq t\leq T_N$. Hence
  \cref{eq:gaussian-noise-hoeffding,eq:gaussian-entrance-error} and a union
  bound give
  \begin{equation*}
    \Pp(\tau_\ast^{(N)}\leq T_N)
    \leq
    C T_N e^{-c_0N}.
  \end{equation*}

  Since $\abs{e_t^{(N)}}\leq1$, splitting according to whether $Z^{(N)}$ has
  exited before time $t$ gives
  \begin{equation*}
    \Ee[(e_t^{(N)})^2]
    =
    \Ee\left[
      (e_t^{(N)})^2\one_{\{\tau_\ast^{(N)}>t\}}
    \right]
    +
    \Ee\left[
      (e_t^{(N)})^2\one_{\{\tau_\ast^{(N)}\leq t\}}
    \right]
    \leq
    \Ee\left[
      (e_t^{(N)})^2\one_{\{\tau_\ast^{(N)}>t\}}
    \right]
    +
    \Pp(\tau_\ast^{(N)}\leq t).
  \end{equation*}
  Combining this with
  \cref{eq:gaussian-stopped-orbit-concentration,eq:gaussian-orbit-no-exit} and
  $T_N=O(\log N)$ gives, uniformly for $t\leq T_N$,
  \begin{equation*}
    \Ee[(e_t^{(N)})^2]
    \leq
    \frac{C}{N}
    +
    C(\log N)e^{-c_0N}.
  \end{equation*}
  Multiplying by $N$ and increasing $C$ for finitely many small values of $N$
  proves \cref{eq:gaussian-orbit-concentration}.
\end{proof}

\subsection{Proof of
\texorpdfstring{\Cref{lem:gaussian-synchronous-cutoff-scale}}
{Lemma~\ref{lem:gaussian-synchronous-cutoff-scale}}}
\label{app:supercritical-cutoff-contraction}

\begin{proof}
  \emph{Step 1: The last $\ell_N$ steps.}
  Take $C_0$ as in \cref{eq:gaussian-uniform-log-horizon-constant}. Fix
  $b>0$ and $c\geq0$. All constants below are independent of $c$. The
  lower bound on $N$ and the $o(1)$ terms may depend on $c$, since the limit in
  $N$ is taken first. By
  \cref{eq:gaussian-uniform-log-horizon},
  \begin{equation*}
    n_N(c)\leq C_0\log N
  \end{equation*}
  for all large $N$. The estimates from
  \cref{lem:gaussian-orbit-concentration} will be used only up to this time.
  For all sufficiently large $N$, put
  \begin{equation*}
    \ell_N:=\max\{0,\lfloor b\log\log N\rfloor\},
    \qquad
    s_N:=\max\{0,n_N(c)-1-\ell_N\}.
  \end{equation*}
  Since $n_N(c)-1-\ell_N>0$ eventually,
  $s_N+\ell_N=n_N(c)-1$ for all large $N$, and
  $\ell_N=o(\log N)$.
  Because $n_N(c)$ is of order $\log N$ while
  $\ell_N=o(\log N)$, we have $s_N\to\infty$. Hence, after increasing $N$ if
  necessary,
  \begin{equation*}
    s_N\geq s_\ast,
    \qquad
    0\leq s_N\leq s_N+\ell_N=n_N(c)-1\leq C_0\log N .
  \end{equation*}
  The orbit estimates from \cref{lem:gaussian-orbit-concentration} therefore
  apply at every time used in the last block.

  \emph{Step 2: The distance at time $s_N$.}
  By the deterministic asymptotic in
  \cref{eq:gaussian-deterministic-asymptotic}, together with the definition of
  $t_N(q_0)$ and the bounded integer-time rounding in $n_N(c)$,
  \begin{equation}\label{eq:gaussian-block-deterministic-start}
    \abs{q_{s_N,N}-q_\ast}
    \leq
    C N^{-1/2}\mu_A^{c-\ell_N}
  \end{equation}
  for all large $N$, with $C$ independent of $c$. Indeed,
  \begin{equation*}
    s_N-\bigl(t_N(q_0)+c-\ell_N\bigr)\in[-2,-1],
  \end{equation*}
  so the rounding contributes at most the fixed factor $\mu_A^{-2}$. Only the
  lower bound on $N$ in the locally uniform deterministic asymptotic may depend
  on $c$. Since $\widetilde Z_0^{(N)}\sim\nu_{A,N}$ and the law
  $\nu_{A,N}$ is invariant, also
  $\widetilde Z_{s_N}^{(N)}\sim\nu_{A,N}$. Therefore, by the triangle
  inequality, Cauchy's inequality,
  \cref{eq:gaussian-orbit-concentration}, and
  \cref{eq:gaussian-stationary-concentration},
  \begin{align*}
    \Ee\abs{Z_{s_N}^{(N)}-\widetilde Z_{s_N}^{(N)}}
    &\leq
    \Ee\abs{Z_{s_N}^{(N)}-q_{s_N,N}}
    +
    \abs{q_{s_N,N}-q_\ast}
    +
    \Ee\abs{\widetilde Z_{s_N}^{(N)}-q_\ast}
    \\
    &\leq
    C N^{-1/2}
    +
    C N^{-1/2}\mu_A^{c-\ell_N}
    +
    C N^{-1/2}.
  \end{align*}
  Hence
  \begin{equation}\label{eq:gaussian-block-start-distance}
    \Ee\abs{Z_{s_N}^{(N)}-\widetilde Z_{s_N}^{(N)}}
    \leq
    C N^{-1/2}\left(\mu_A^{c-\ell_N}+1\right).
  \end{equation}

  \emph{Step 3: Both chains stay near $q_\ast$.}
  Let
  \begin{equation*}
    \rho_N:=(\log N)^{-2},
    \qquad
    J_N:=[q_\ast-\rho_N,q_\ast+\rho_N]\cap[0,1].
  \end{equation*}
  For every $0\leq j\leq\ell_N$, the orbit $(q_{t,N})_{t\geq0}$ remains
  closer to $q_\ast$ as time increases. Hence \cref{eq:gaussian-block-deterministic-start}
  implies
  \begin{equation}\label{eq:gaussian-block-deterministic-local}
    q_{s_N+j,N}
    \in
    [q_\ast-\rho_N/2,q_\ast+\rho_N/2]
  \end{equation}
  for all large $N$, because
  $N^{-1/2}\mu_A^{c-\ell_N}=o(\rho_N)$. Chebyshev's inequality and
  \cref{eq:gaussian-orbit-concentration} therefore give
  \begin{equation*}
    \Pp\left(Z_{s_N+j}^{(N)}\notin J_N\right)
    \leq
    \Pp\left(
      \abs{Z_{s_N+j}^{(N)}-q_{s_N+j,N}}>\rho_N/2
    \right)
    \leq
    \frac{C}{N \rho_N^2}.
  \end{equation*}
  Similarly, stationarity of $\widetilde Z^{(N)}$ and
  \cref{eq:gaussian-stationary-concentration} imply
  \begin{equation*}
    \Pp\left(\widetilde Z_{s_N+j}^{(N)}\notin J_N\right)
    \leq
    \frac{C}{N \rho_N^2}.
  \end{equation*}
  Thus, with
  \begin{equation*}
    E_j
    :=
    \left\{
      Z_{s_N+j}^{(N)}\in J_N,\,
      \widetilde Z_{s_N+j}^{(N)}\in J_N
    \right\},
  \end{equation*}
  we have
  \begin{equation}\label{eq:gaussian-block-bad-probability}
    \Pp(E_j^c)
    \leq
    \frac{C}{N \rho_N^2}.
  \end{equation}

  \emph{Step 4: The distance contracts during the last block.}
  On $J_N$ the deterministic map has derivative close to its derivative at
  $q_\ast$. Since $V_A'(q_\ast)=\mu_A$ and $V_A''$ is bounded in a
  neighborhood of $q_\ast$, the mean value theorem gives, for all large $N$
  and all $q,q'\in J_N$,
  \begin{equation}\label{eq:gaussian-local-taylor-contraction}
    \abs{V_A(q)-V_A(q')}
    \leq
    \left(\mu_A+C \rho_N\right)\abs{q-q'}.
  \end{equation}
  We choose $N$ large enough that
  $
    a_N:=\mu_A+C \rho_N<1
  $.

  Let
  $
    D_j:=
    \Ee\abs{
      Z_{s_N+j}^{(N)}
      -
      \widetilde Z_{s_N+j}^{(N)}
    }
  $.
  Conditional on the two radii at time $s_N+j$, the next step uses the same
  Gaussian vector in both chains. For fixed $g\in\R^N$, the map
  $q\mapsto F_{A,N}(q,g)$ is nondecreasing. Indeed, for $q>0$,
  \begin{equation*}
    \partial_q\tanh^2(A\sqrt q\,g_i)
    =
    \frac{Ag_i}{\sqrt q}
    \tanh(A\sqrt q\,g_i)
    \operatorname{sech}^2(A\sqrt q\,g_i)
    \geq0,
  \end{equation*}
  and the endpoint $q=0$ follows by continuity. Monotonicity therefore
  removes the absolute value inside the expectation, and for every
  $q,q'\in[0,1]$,
  \begin{equation*}
    \Ee\abs{F_{A,N}(q,G)-F_{A,N}(q',G)}
    =
    \abs{V_A(q)-V_A(q')}.
  \end{equation*}
  Applying this identity conditionally gives
  \begin{align*}
    \Ee\left[
      \abs{
        Z_{s_N+j+1}^{(N)}
        -
        \widetilde Z_{s_N+j+1}^{(N)}
      }
      \,\middle|\, Z_{s_N+j}^{(N)},\widetilde Z_{s_N+j}^{(N)}
    \right]
    =
    \abs{
      V_A(Z_{s_N+j}^{(N)})
      -
      V_A(\widetilde Z_{s_N+j}^{(N)})
    }.
  \end{align*}
  On $E_j$, \cref{eq:gaussian-local-taylor-contraction} bounds this
  by
  $a_N\abs{Z_{s_N+j}^{(N)}-\widetilde Z_{s_N+j}^{(N)}}$. On
  $E_j^c$ we use only the trivial bound by one. Taking expectations
  and using \cref{eq:gaussian-block-bad-probability} gives
  \begin{equation}\label{eq:gaussian-block-recursion}
    D_{j+1}
    \leq
    a_ND_j
    +
    \frac{C}{N \rho_N^2}
    \qquad
    0\leq j<\ell_N .
  \end{equation}

  Iterating the recursion and using $a_N<1$ gives
  \begin{equation*}
    D_{\ell_N}
    \leq
    a_N^{\ell_N}D_0
    +
    \frac{C}{N \rho_N^2}
    \sum_{k=0}^{\ell_N-1}a_N^k
    \leq
    a_N^{\ell_N}D_0
    +
    \frac{C}{N \rho_N^2}.
  \end{equation*}
  Since $\rho_N=(\log N)^{-2}$,
  \begin{equation*}
    \frac{1}{N \rho_N^2}
    =
    \frac{(\log N)^4}{N}
    =
    o(N^{-1/2}).
  \end{equation*}
  Hence
  \begin{equation}\label{eq:gaussian-block-iterated}
    D_{\ell_N}
    \leq
    a_N^{\ell_N}D_0
    +
    o(N^{-1/2}).
  \end{equation}

  \emph{Step 5: The distance at time $n_N(c)-1$.}
  The block length recovers the factor $\mu_A^{\ell_N}$, while the
  perturbation $C\rho_N$ has no asymptotic effect. Indeed,
  \begin{equation*}
    a_N^{\ell_N}
    =
    \mu_A^{\ell_N}
    \left(1+\frac{C \rho_N}{\mu_A}\right)^{\ell_N},
  \end{equation*}
  and $\ell_N \rho_N\to0$. Therefore
  \begin{equation}\label{eq:gaussian-block-factor}
    a_N^{\ell_N}
    =
    \mu_A^{\ell_N}(1+o(1)).
  \end{equation}
  Combining
  \cref{eq:gaussian-block-start-distance,eq:gaussian-block-iterated,eq:gaussian-block-factor}
  gives
  \begin{equation*}
    \sqrt N\,D_{\ell_N}
    \leq
    C\mu_A^{\ell_N}
    \left(\mu_A^{c-\ell_N}+1\right)
    +
    o(1)
    \leq
    C\mu_A^c+C\mu_A^{\ell_N}+o(1)
    =
    C\mu_A^c+o(1).
  \end{equation*}
  Here $\mu_A^{\ell_N}\to0$ because $\mu_A\in(0,1)$ and
  $\ell_N\to\infty$. The constants in the preceding estimate are independent
  of $c$, while the $o(1)$ term may depend on its fixed value.
  Since $D_{\ell_N}$ is the expectation in
  \cref{eq:gaussian-synchronous-cutoff-scale}, the claimed bound follows.
\end{proof}

\section*{Acknowledgements}
The author used OpenAI's ChatGPT and Anthropic's Claude in preparing this manuscript. The author takes full responsibility for its contents.

\bibliographystyle{alpha}
\bibliography{refs}

\end{document}